\documentclass[11pt]{amsart}
\usepackage{amsmath, amsthm, amssymb, hyperref, geometry, url}
\usepackage{amsopn, latexsym, graphics, graphicx, enumerate}
\usepackage[all]{xy}
\usepackage{setspace}
\usepackage{float}
\restylefloat{table}

\usepackage{tikz}
\usetikzlibrary{matrix,arrows,decorations.pathmorphing}

\allowdisplaybreaks

\evensidemargin \oddsidemargin
\author{Jeffrey S. Meyer}
\address{Department of Mathematics\\
University of Oklahoma\\
Norman, OK 73019 USA}
\email[]{jmeyer@math.ou.edu}

\title{Totally Geodesic Spectra of Arithmetic Hyperbolic Spaces}

\usepackage{fancyhdr}

\fancypagestyle{plain}

\DeclareMathAlphabet{\curly}{U}{rsfs}{m}{n}

\makeatletter
\def\@tocline#1#2#3#4#5#6#7{\relax
  \ifnum #1>\c@tocdepth % then omit
  \else
    \par \addpenalty\@secpenalty\addvspace{#2}%
    \begingroup \hyphenpenalty\@M
    \@ifempty{#4}{%
      \@tempdima\csname r@tocindent\number#1\endcsname\relax
    }{%
      \@tempdima#4\relax
    }%
    \parindent\z@ \leftskip#3\relax \advance\leftskip\@tempdima\relax
    \rightskip\@pnumwidth plus4em \parfillskip-\@pnumwidth
    #5\leavevmode\hskip-\@tempdima
      \ifcase #1
       \or\or \hskip 4em \or \hskip 2em \else \hskip 3em \fi%
      #6\nobreak\relax
    \dotfill\hbox to\@pnumwidth{\@tocpagenum{#7}}\par
    \nobreak
    \endgroup
  \fi}
\makeatother

\numberwithin{figure}{section}
\numberwithin{equation}{section}
\numberwithin{table}{section}

\newcommand{\specialcell}[2][c]{\begin{tabular}[#1]{@{}l@{}}#2\end{tabular}}

\DeclareFontFamily{OT1}{rsfs}{}
\DeclareFontShape{OT1}{rsfs}{n}{it}{<-> rsfs10}{}
\DeclareMathAlphabet{\mathscr}{OT1}{rsfs}{n}{it}

\DeclareMathOperator{\N}{N}

\newtheorem{thm}[subsection]{Theorem}
\newtheorem*{thm*}{Theorem}
\newtheorem{mainthm}{Theorem}

\newtheorem{ques}{Question}
\newtheorem{lem}[subsection]{Lemma}
\newtheorem{cor}[subsection]{Corollary}
\newtheorem*{cor*}{Corollary}
\newtheorem{prop}[subsection]{Proposition}

\theoremstyle{definition}

\newtheorem{rem}[subsection]{Remark}
\newtheorem{ex}[subsection]{Example}
\newtheorem{Def}[subsection]{Definition}

\newtheorem{con}[subsection]{Construction}

\def\C{\mathbb{C}}
\def\Q{\mathbf{Q}}

\def\R{\mathbb{R}}

\def\N{\mathbf{N}}
\def\Q{\mathbb{Q}}
\def\Z{\mathbb{Z}}

\def\1{\mathbf{1}}

\makeatletter
\def\moverlay{\mathpalette\mov@rlay}
\def\mov@rlay#1#2{\leavevmode\vtop{%
   \baselineskip\z@skip \lineskiplimit-\maxdimen
   \ialign{\hfil$\m@th#1##$\hfil\cr#2\crcr}}}
\newcommand{\charfusion}[3][\mathord]{
    #1{\ifx#1\mathop\vphantom{#2}\fi
        \mathpalette\mov@rlay{#2\cr#3}
      }
    \ifx#1\mathop\expandafter\displaylimits\fi}
\makeatother

\makeatletter
\let\@@pmod\pmod
\DeclareRobustCommand{\pmod}{\@ifstar\@pmods\@@pmod}
\def\@pmods#1{\mkern4mu({\operator@font mod}\mkern 6mu#1)}
\makeatother

\begin{document}

\maketitle

\begin{abstract}
In this paper we show that totally geodesic subspaces determine the commensurability class of a standard arithmetic hyperbolic $n$-orbifold, $n\ge 4$.  
%analyze the extent to which totally geodesic submanifolds determine the geometry of standard arithmetic hyperbolic manifolds.  
%We extend the results of McReynolds--Reid to show that totally geodesic subspaces determine the commensurability class of a standard arithmetic hyperbolic $n$-orbifold, $n\ge 4$.  
Many of the results are more general and apply to locally symmetric spaces associated to arithmetic lattices in $\R$-simple Lie groups of type $B_n$ and $D_n$.
% the form $\prod_{i=1}^r \mathbf{SO}(p_i,m-p_i) \times (\mathbf{SO}_m(\C))^s$.  
We use a combination of techniques from algebraic groups and quadratic forms to prove several results about these spaces. 
%Our techniques allow us to prove several results about totally geodesic subspaces We produce a dictionary between the classical invariants of quadratic forms and the Tits index of its associated isometry group.  
%We are also able to use our methods to give an alternate proof of Maclachlan's parametrization of commensurability classes of even dimensional arithmetic hyperbolic manifolds.  
%We prove our Hyperbolic Subspace Dichotomy Theorem and extend our results to nonstandard hyperbolic spaces.
\end{abstract}

\section{Introduction}\label{sectionintroduction}

The goal of this paper is to determine the extent to which the geometry of an arithmetic hyperbolic $n$-manifold, $n\ge 4$, is encoded in the collection of its totally geodesic submanifolds.  
To put this goal in a broader context, we step back a moment and ask a natural question, one going back over a century: 
\textit{What topological and geometric properties of a space $M$ are encoded in certain interesting collections of geometric data associated to $M$?} 
One of the earliest examples of this line of inquiry was in 1911, when Weyl showed that the eigenvalues of the Laplace-Beltrami operator determine the dimension and volume of a closed Riemannian manifold \cite{W}.    
In 1966, Kac popularized this question by asking ``Can one hear the shape of a drum?'' \cite{Kac}.
Since that time many different collections of data, often called \textbf{spectra}, have been studied.  
Over the past few decades, one prominent spectrum has been the collection of lengths of closed geodesics.
The \textbf{weak length spectrum} (sometimes also referred to as the \textbf{length set}) of a Riemannian manifold $M$, is the set 
\begin{align}\label{lm1}
L(M)	&:=\{\lambda\in \R\ | \ \mbox{$\lambda$ is the length of a closed geodesic in $M$}\}.
\end{align} 
Observe that this collection can be equivalently formulated as 
\begin{align}\label{lm2}
L(M)	&\ = \{\mbox{Isometry classes of closed geodesics in $M$}\}.
\end{align}
We call two manifolds with the same weak length spectrum \textbf{weakly iso-length-spectral}.

\begin{ques} 
Are weakly iso-length-spectral spaces necessarily isometric?
\end{ques}

The answer is a resounding no, and since the 1960's, there have been many constructions of weakly iso-length-spectral spaces that are not isometric, the most famous of which being:
\begin{itemize}
\item 16-dimensional flat tori (Milnor, 1964 \cite{Mi}),
\item 2- and 3-dimensional hyperbolic manifolds, and more generally spaces spaces coming from quaternion algebras (Vign\'eras, 1980 \cite{Vi}),
\item General method based on covering space theory (Sunada, 1985 \cite{S}).
\end{itemize} 
However, these constructions produce manifolds that are \textit{almost} isometric in the sense that they are commensurable (Section \ref{sectionnotation}). 
When two Riemannian manifolds $M_1$ and $M_2$ are commensurable, the length of every closed geodesic in $M_1$ is a rational multiple of the length of a closed geodesic in $M_2$, and vice versa.  Motivated by this, \cite{CHLR} defined the \textbf{rational length spectrum} to be the set   
\begin{align}\label{qlm1}
\Q L(M)	&:=\{s\lambda\in \R\ | \ \mbox{$s\in \Q$ and $\lambda$ is the length of a closed geodesic in $M$}\}.
\end{align}
Again, we observe that this definition may be reformulated as follows:
\begin{align}\label{qlm2}
\Q L(M)	&\ = \{\mbox{Commensurability classes of closed geodesic in $M$}\}.
\end{align}

%\begin{align*}
%\Q L(M)	&:=\{s\lambda\in \R\ | \ \mbox{$s\in \Q$ and $\lambda$ is the length of a closed geodesic in $M$}\}\\
%	&\ = \{\mbox{Commensurability classes of closed geodesic in $M$}\}.
%\end{align*}
%

Two manifolds with the same rational length spectrum are said to be \textbf{length-commensurable}.  In particular, commensurable manifolds are length-commensurable.  One may then ask the following refined question. 

\begin{ques}
Are length-commensurable spaces necessarily commensurable?
%Does $\Q L(M)$ determine the commensurability class of $M$?
\end{ques}

In many cases, the answer is yes.
When $M_1$ and $M_2$ are arithmetic hyperbolic $n$-manifolds,  then $\Q L(M_1)=\Q L(M_2)$ implies $M_1$ and $M_2$ are commensurable in each of the following cases:
\begin{itemize}
\item $n=2$ (Reid, 1992 \cite{R}),
\item $n=3$ (Chinburg, Hamilton, Long, and Reid, 2008 \cite{CHLR}),
\item $n\ne3$, $n\ne 7$,\  \ $n\not\equiv 1 \ (\mbox{mod } 4)$ (Prasad and Rapinchuk, 2009 \cite{PR}),
\item $n=7$ (Garibaldi, 2013 \cite{G}).
\end{itemize}
However,  for each positive $n\equiv 1 \ (\mbox{mod } 4)$, $n>1$,  \cite{PR} produced examples of noncommensurable length-commensurable arithmetic hyperbolic $n$-manifolds.  
More generally, there are many constructions of families of pairwise noncommensurable length-commensurable arithmetic locally symmetric spaces of the same Killing--Cartain type (see \cite[Theorem 1]{LSV}, \cite[Construction 9.15]{PR}).

Our motivation is to find a collection of data that is complementary to length spectra and that distinguishes commensurability classes.   
In recent years, there have been many papers looking at certain higher dimensional analogues of geodesics: totally geodesic subspaces. 
For us, it will be sufficient to only consider \textit{nonflat} totally geodesic subspaces.  
Furthermore, in analogy with looking at closed geodesics, we only want to look at \textit{finite volume} subspaces. 
With this in mind, we define the \textbf{totally geodesic set} of a Riemannian manifold to be the set
\begin{align}\label{deftgm}
TG(M):=\left\{ \parbox{2.7in}{\begin{center}Isometry classes of nonflat finite volume totally geodesic subspaces of  $M$\end{center}}\right\}.
\end{align}

McReynolds and Reid \cite{McR} prove that if $M_1$ and $M_2$ are standard \cite[Theorem 10.2.3]{MaR2} arithmetic hyperbolic 3-manifolds such that $TG(M_1)=TG(M_2)$, then $M_1$ and $M_2$ are commensurable.
As was the case for the weak length spectrum, $TG(M)$ is not an invariant of commensurability class, and hence we define the \textbf{rational totally geodesic spectrum}  to be the set
\begin{align}\label{defqtgm}
\Q TG(M):=\left\{ \parbox{3.3in}{\begin{center}Commensurability classes of nonflat finite volume totally geodesic subspaces of  $M$\end{center}}\right\}.
\end{align}
Observe that $TG(M)$ and $\Q TG(M)$ are natural analogues of the second formulations of $L(M)$ and $\Q L(M)$ (see \ref{lm2} and \ref{qlm2}).  
The former is more rigid while the later is an invariant of the commensurability class of $M$.   
If two Riemannian orbifolds $M$ and $M'$ have the same rational totally geodesic spectrum, we say they are \textbf{totally-geodesic-commensurable}.
The goal of this paper is to investigate the following question:

\begin{ques}
Are totally-geodesic-commensurable spaces necessarily commensurable?
%Does $\Q TG(M)$ determine the commensurability class of $M$?
\end{ques}

In this paper we address this question in the case of locally symmetric spaces of type $B_n$ and $D_n$.  
In particular, we focus on standard arithmetic locally symmetric spaces associated to Lie groups of the form $\prod_{i=1}^r \mathbf{SO}(p_i,m-p_i) \times (\mathbf{SO}_m(\C))^s$, for $m\ge5$.  
These spaces are constructed via the isometry groups of quadratic forms over number fields (see Construction \ref{quadcon}).   
We call a locally symmetric space $\R$-\textbf{simple} if its associated Lie group is not a nontrivial product (i.e., if $r+s=1$).  
Note that standard arithmetic hyperbolic $n$-manifolds are $\R$-simple.   

The first step to proving that the rational totally geodesic spectrum determines the commensurability class is showing it determines the field of definition, which we do in Section \ref{sectionfieldthrma}
%When $M$ is $\R$-simple, $\Q TG(M)$ determines the field of definition.
% (Section \ref{sectionfieldthrma}).

\begin{mainthm}\label{thrmA}
Let $M_1$ and $M_2$ be $\R$-simple, arithmetic, locally symmetric orbifolds coming from quadratic forms of dimension $m\ge 5$ over number fields $k_1$ and $k_2$ respectively.  
If $\Q TG(M_1)= \Q TG(M_2)$, then  $k_1$ and $k_2$ are isomorphic.
\end{mainthm}

Using the technical results of Section \ref{sectionconstructions} on quadratic forms and their isometry groups, we are then able to prove our main theorem on commensurability.

%Furthermore, the rational totally geodesic spectrum determines the commensurability class of an $\R$-simple, standard arithmetic 
%We go on to show that when $\Q TG(M)$ determines the commensurability class of $M$.

\begin{mainthm}
\label{thrmB}
Let $M_1$ and $M_2$ be $\R$-simple, arithmetic locally symmetric orbifolds coming from quadratic forms of dimension $m\ge 5$.
%  over number fields $k_1$ and $k_2$ respectively.  
%Suppose that either $M_1$ and $M_2$ are $\R$-simple, or $m$ is odd and $k_1$ and $k_2$ are isomorphic.  
If $\Q TG(M_1)= \Q TG(M_2)$, then $M_1$ and $M_2$ are commensurable.
\end{mainthm}

Unlike \cite{PR} and \cite{G}, Theorem \ref{thrmB} is not dependent in $\R$-rank$\ge 2$ upon the truth of Schanuel's conjecture.  
%Furthermore, Theorem \ref{thrmB} covers a large class of non-$\R$-simple spaces of type $B_n$ that had not been covered under the results of \cite{PR}.  
Specializing to the $\R$-rank one case, we obtain a result about standard arithmetic hyperbolic orbifolds.
%Theorem \ref{thrmB} gives the following theorem.

\begin{mainthm}
\label{thrmC}
Let $M_1$ and $M_2$ be standard arithmetic hyperbolic $n$-orbifolds, $n\ge 4$.  If $\Q TG(M_1)= \Q TG(M_2)$, then $M_1$ and $M_2$ are commensurable.
\end{mainthm}

In fact, we show that codimension one and codimension two totally geodesic subspaces determine the commensurability class of a standard arithmetic hyperbolic orbifold.  
We go on to show in Theorem \ref{hypersurfacessuffice} that the commensurability class of an even dimensional arithmetic hyperbolic orbifold is totally determined by its codimension one totally geodesic subspaces. 
To complement these results, we then show that there are many commensurability classes of hyperbolic orbifolds with the exact same collection of totally geodesic subspaces in codimension greater than 2.  

\begin{mainthm}[Hyperbolic Subspace Dichotomy]
\label{thrmD}
Let $M_1$ and $M_2$ be standard arithmetic hyperbolic $n$-orbifolds, $n\ge 4$, with fields of definition $k_1$ and $k_2$, respectively.  
Then either 
\begin{enumerate}
\item $k_1\cong k_2$, in which case, for all $j\in \N$, $1< j<n-2$, up to commensurability, $M_1$ and $M_2$ have the exact same collection of $j$-dimensional finite volume totally geodesic subspaces, or
\item $k_1\not\cong k_2$, in which case, up to commensurability, $M_1$ and $M_2$  do not share a single finite volume totally geodesic subspace of dimension $\ge 2$.
\end{enumerate}
\end{mainthm}

The techniques that we use to prove Theorems \ref{thrmA} - \ref{thrmD} have many applications.
In Section \ref{sectionapplications} we use them to answer a question posed to us by Jean-Fran\c{c}ois Lafont, and in Appendix \ref{sectionmaclachlan} we give an alternate proof of Machlachlan's parametrization of commensurability classes of even dimensional arithmetic hyperbolic orbifolds.

Along the way, we construct several explicit examples of commensurability classes of standard arithmetic hyperbolic orbifolds with specific properties.  
In particular, in Example \ref{interestingex}, we construct a hyperbolic 3-orbifold $N$ and a hyperbolic 5-orbifold $M$ such that every totally geodesic surface in $N$ is commensurable to a totally geodesic surface in $M$, yet $N$ is not commensurable to a totally geodesic subspace of $M$.

While all even dimensional arithmetic hyperbolic orbifolds come from quadratic forms, there are odd dimensional ones that do not.  
To this end we also address some results on spaces coming from skew Hermitian forms over quaternion division algebras over number fields.   
Though there are considerable obstructions to finishing the analysis for groups coming from this construction, we do have initial results.
To start, an $\R$-simple, nonstandard arithmetic, locally symmetric space cannot be totally-geodesic-commensurable with a standard one.

\begin{mainthm}
\label{thrmE}
Let $M_1$ and $M_2$ be $\R$-simple, arithmetic locally symmetric spaces where $M_1$ comes from a quadratic form and $M_2$ comes from a skew Hermitian form over a division algebra.  Then $\Q TG(M_1)\ne \Q TG(M_2)$.
\end{mainthm}

Furthermore, the rational totally geodesic spectrum determines the field and algebra of definition of a nonstandard arithmetic lattice.  

\begin{mainthm}
\label{thrmF}
Let $M_1$ and $M_2$ be $\R$-simple arithmetic locally symmetric spaces coming from skew Hermitian forms of dimension $n\ge 4$ over  quaternion division algebras $D_1$ and $D_2$ over number fields $k_1$ and $k_2$ respectively.  
If $\Q TG(M_1)= \Q TG(M_2)$, then $k_1$ and $k_2$ are isomorphic and this isomorphism induces an isomorphism between $D_1$ and $D_2$.
\end{mainthm}

%-----------------------------------------------------------------------------------------------------------
%----------------------- NOTATION AND PRELIM SECTION ----------------------------------
%-----------------------------------------------------------------------------------------------------------

\section{Notation and Preliminary Results:\\ Commensurability, Totally Geodesic Subspaces, and Locally Symmetric Spaces}\label{sectionnotation}

In this paper $F$ is a field that is not of characteristic 2, $\overline{F}$ is a fixed algebraic closure of $F$, $k$ is a number field, and $\mathcal{O}_k$ is its ring of integers.
%, and  $\mathrm{Aut}(k/\Q)$ is the group of field automorphisms of $k$. 
%We consider number fields as finite extensions of $\Q$ lying within $\overline{\Q}$, $k$ will denote a number field,  and $\mathcal{O}_k$ its ring of integers, and  $\mathrm{Aut}(k/\Q)$ is the group of field automorphisms of $k$. 

Two subgroups $\Gamma_1$ and $\Gamma_2$ of a group $G$ are \textbf{commensurable} if $\Gamma_1\cap \Gamma_2$ is finite index in both $\Gamma_1$ and $\Gamma_2$.  
Following \cite{PR}, we shall say two subgroups $\Gamma_1, \Gamma_2$ of a $G$ are \textbf{commensurable up to $G$-automorphism} if there exists a $G$-automorphism $\varphi$ such that $\Gamma_1$ and $\varphi(\Gamma_2)$ are commensurable. 
(Note that some authors refer to this notion as \textbf{commensurable in the wide sense} \cite[Def. 1.3.4]{MaR2}.)
Commensurability up to $G$-automorphism is an equivalence relation among subgroups of $G$.  
Two Riemannian manifolds are \textbf{commensurable} if they have isometric finite sheeted covers.

Let $M$ be a Riemannian manifold and let $N\subset M$ be a connected immersed submanifold.  
Recall  that $N$ is \textbf{geodesic at $p\in N$} if every geodesic of $M$ starting at $p$ and tangent to $N$ at $p$ is a geodesic of $N$. 
If $N$ is geodesic at each of its points it is called \textbf{totally geodesic}.     
%The totally geodesic subspaces of hyperbolic space are also hyperbolic \cite[Ch. 8.]{DC}.  

Following \cite[Chp 13]{Th}, we call the quotient of a manifold by a properly discontinuous (not necessarily free) group action a \textbf{good orbifold}.
Since discrete subgroups of semisimple Lie groups often have torsion, good orbifolds naturally appear in the commensurability classes of locally symmetric manifolds.
When a good orbifold is a quotient of a Riemannian manifold by isometries, we call it a good Riemannian orbifold.  
Every good Riemannian orbifold naturally has a Riemannian manifold universal cover.  
A subspace of a Riemannian orbifold is defined to be totally geodesic if it is the image of a totally geodesic subspace in its universal cover.  
It follows  that the sets $TG(M)$ and $\Q TG(M)$ (Definitions \ref{deftgm} and \ref{defqtgm}) make sense for all good Riemannian orbifolds.

\begin{lem}Commensurable good Riemannian orbifolds are totally-geodesic-commensurable.\end{lem}

\begin{proof}
Let $M_1$ and $M_2$ be commensurable and $\widetilde{M}$ be a shared finite sheeted cover with projections $\pi_1$ and $\pi_2$.  
Pick a nonflat finite volume totally geodesic subspace $N_1\subset M_1$.  
Then $N_2:=\pi_2(N')$, where $N'$ is a connected componant of $\pi_1^{-1}(N_1)$, is a totally geodesic submanifold of $M_2$.  
Since $\pi_1$ and $\pi_2$ are finite sheeted covers, $N_2$ is also nonflat and of finite volume.  By symmetry of argument, the result follows.
\end{proof}

In general, totally geodesic subspaces are rare, and we should only expect to find such subspaces when we are considering an ambient space with many symmetries.  
As such, in what follows, we shall only consider locally symmetric spaces.  
A Riemannian manifold $M$ is a \textbf{globally symmetric space} if each point $p\in M$ is an isolated fixed point of an involutive isometry of $M$.  
Totally geodesic subspaces of a globally symmetric space are also globally symmetric \cite[Ch. IV  Prop 7.1]{H}.
One of the advantages to working with globally symmetric spaces is that questions about the spaces can be translated into questions about its isometry group.  
A globally symmetric space is of \textbf{noncompact type}  if $G:=\mathrm{Isom}^\circ(M)$ is a semisimple Lie group with no compact factors, in which case $M$ is isometric to $G/K$ where $K$ is a maximal compact subgroup of $G$.

\begin{lem}\label{prop24}
Let $M$ a connected globally symmetric space of noncompact type, $G=\mathrm{Isom}^\circ(M)$ and $K$ a stabilizer of a point $p_0\in M$.   
\begin{enumerate}
\item Let $H\subset G$ be a semisimple Lie subgroup with no compact factors.  Then $N_H:=H/(H\cap K)$ is a totally geodesic submanifold of $M$.
\item Let $N\subset M$ be a totally geodesic submanifold of noncompact type such that $p_0\in N$.  Then there exists a semisimple Lie subgroup $H_N\subset G$ with no compact factors such that $H_N/(H_N\cap K)=N$.
\end{enumerate}
\end{lem}

\begin{proof}${}$

(1).  Note that $N_H$ is an immersed submanifold of $M$.  Geodesics of $M$ arise from the exponential map of $G$.  Given an element  $X\in\mathrm{Lie}(H)$ we know that $\exp_G(tX)\in H$ for all $t\in \R$, and hence $N$ must be totally geodesic.

(2). 
Let $\mathrm{Lie}(G)=\mathfrak{k}\oplus \mathfrak{p}$ be the Cartan decomposition.  Let $\mathfrak{s}\subset \mathfrak{p}$ be the subspace associated with the tangent space of $N$.  Then $\mathfrak{k}$ acts on $\mathfrak{p}$ by the adjoint representation and let $\mathfrak{k}'=N_{\mathfrak{k}}(\mathfrak{s})=\{X\in \mathfrak{k}\ | \ \mathrm{ad}(X)(\mathfrak{s})\subset\mathfrak{s}\}$.  Then $\mathfrak{h}:=\mathfrak{k}'\oplus \mathfrak{s}$ is a Lie subalgebra of $\mathrm{Lie}(G)$.  Let $H_N$ be the unique connected Lie subgroup of $G$ with Lie algebra $\mathfrak{h}$.  It follows that $H_N$ has the desired properties.
\end{proof}

A good Riemannian orbifold $M$ is a  \textbf{locally symmetric space} if $M$ has universal cover $\widetilde{M}$ that is a globally symmetric space.  
In which case $M=\Gamma\backslash \widetilde{M}$ where  $\Gamma$ is a discrete subgroup of $\mathrm{Isom}^\circ(\widetilde{M})$.  
A locally symmetric space is of \textbf{noncompact type} if its universal cover is a globally symmetric space of  noncompact type.  
Totally geodesic subspaces of a locally symmetric space are also locally symmetric.
The study of locally symmetric spaces of noncompact type translates to the study of discrete subgroups of semisimple Lie groups with no compact factors, as we shall now record with the following well known proposition.  

\begin{prop}
Let $M_1=\Gamma_1\backslash G_1/K_1$ and $M_2=\Gamma_2\backslash G_2/K_2$ be locally symmetric spaces of noncompact type where $G_1$ and $G_2$ are connected, adjoint, semisimple Lie groups with no compact factors.  Then $M_1$ and $M_2$ are isometric if and only if there is a Lie group isomorphism $\varphi: G_1\to G_2$ such that $\varphi(K_1)=K_2$ and $\varphi(\Gamma_1)=\Gamma_2$
\end{prop}

Since the image of a maximal compact (resp. discrete) subgroup under an automorphism is always a maximal compact (resp. discrete) subgroup, understanding isometry classes of locally symmetric spaces of noncompact type with universal cover $G/K$ reduces to understanding $\mathrm{Aut}(G)$-orbits of discrete subgroups of $G$.   In particular, understanding the commensurability classes of locally symmetric spaces is equivalent to understanding the commensurability classes of discrete subgroups of $G$ up to $G$-automorphism.

Let $G$ be a semisimple Lie group and $\Gamma\subset G$ be a discrete subgroup.  The Haar measure on $G$ naturally descends to a $G$-invariant measure on $\Gamma\backslash G$.  When the Haar measure on $G$ descends to a measure of finite volume on $\Gamma\backslash G$, $\Gamma$ is called a \textbf{lattice}.  When $\Gamma\backslash G$ is compact, $\Gamma$ is said to be \textbf{cocompact} or a \textbf{uniform lattice}.  Cocompact discrete subgroups are always lattices.  A lattice is \textbf{irreducible} if, up to commensurability, it is not a product of smaller lattices.  Being cocompact, a lattice, or irreducible is an invariant of commensurability class.

Henceforth, our orbifolds will be good and our locally symmetric spaces will be of noncompact type.

%-----------------------------------------------------------------------------------------------------------
%----------------------ARITHMETIC MANIFOLDS SECTION----------------------------------
%-----------------------------------------------------------------------------------------------------------

\section{Arithmetic Groups and Arithmetic Locally Symmetric Spaces}\label{sectionarithmetic}

\subsection*{Arithmetic Subgroups of Algebraic $\Q$-Groups}\label{arithmeticsectionqgroups}${}$

Let $\mathbf{G}$ be an algebraic group defined over $\Q$.  
There exists a faithful $\Q$-rational embedding $\rho: \mathbf{G}\to \mathbf{GL}(V)$ for some $\Q$-vector space $V$ \cite[1.10]{B1}.   
Let $L\subset V$ be a $\Z$-lattice of $V$, i.e., a free $\Z$-module such that $L\otimes_{\Z} \Q=V$.  Define the group
$$G_{\rho, L}:=\{g\in \mathbf{G}(\Q)\ | \ \rho(g)(L)=L\}.$$
Any subgroup $\Gamma\subset \mathbf{G}(\Q)$ commensurable
with $G_{\rho, L}$ is an  \textbf{arithmetic subgroup} of $\mathbf{G}(\Q)$.  
Were we to chose a different embedding, $\rho'$, and different $\Z$-lattice, $L'$, we would have obtained a different group $G_{\rho', L'}$, however, any such $G_{\rho',L'}$ is commensurable with $G_{\rho, L}$ 
(see \cite[7.12]{B} and preceding discussion).
It follows that the commensurability class of an arithmetic group is independent of the choices of $\rho$ and $L$.    In other words, the $\Q$-isomorphism class of $\mathbf{G}$ determines a commensurability class of arithmetic groups.  

Often we will assume the existence of some embedding $\rho$ and lattice $L$, and we will denote $\mathbf{G}(\Z):=G_{\rho, L}$.  
Note however that not all arithmetic groups arise as the stabilizer of a lattice.  
This can be seen from the fact that every lattice stabilizer contains a congruence subgroup \cite[7.12]{B} but there are arithmetic groups  that  do not contain any congruence subgroups (for example, there are such groups in $\mathbf{SL}_2(\Z)$) \cite[\S2.1]{PR2}.

One way to construct algebraic $\Q$-groups is to start with a $k$-group, where $k$ is a number field, and then apply the \textbf{Weil restriction of scalars functor} $R_{k/\Q}$ \cite[\S 2.1.2]{PlRa}, \cite[\S10.3]{MaR2}.  
This functor has the property that if $\mathbf{G}$ is an algebraic $k$-group, then $R_{k/\Q}\mathbf{G}$ is an algebraic $\Q$-group and there is an abstract group isomorphism between $\mathbf{G}(k)$ and $(R_{k/\Q}\mathbf{G})(\Q)$. 
With this identification, it makes sense to talk about arithmetic subgroups of $\mathbf{G}(k)$.  Furthermore, it is not hard to see that arithmetic subgroups of $\mathbf{G}(k)$ are precisely the groups commensurable with the stabilizer of an $\mathcal{O}_k$-lattice of a $k$-vector space $V$ where there is a $k$-rational embedding of $\mathbf{G}$ into $\mathbf{GL}(V)$.  

An \textbf{absolutely (resp. absolutely almost) simple algebraic $F$-group} is an algebraic $F$-group that, upon extending scalars to $\overline{F}$, is (resp. isogenous to) a simple semisimple algebraic $\overline{F}$-group.  
For example the $\C$-group $\mathbf{SL}_n$ is absolutely almost simple but not absolutely simple since it has nontrivial center equal to the group of $n^{th}$ roots of unity.  
The semisimple $\R$-group $R_{\C/\R}\mathbf{SL}_2$, which is related to the study of hyperbolic 3-manifolds, is not absolutely almost simple, since it is $\C$-isomorphic to $\mathbf{SL}_2\times \mathbf{SL}_2$.  
If we start with an absolutely almost simple $k$-group, then $R_{k/\Q}(\mathbf{G})$ is always a semisimple $\Q$-group.  An \textbf{$F$-simple $F$-group} is an algebraic $F$-group which, up to isogeny, does not contain a proper nontrivial normal $F$-subgroup.  
Absolutely almost simple $F$-groups are $F$-simple and $R_{\C/\R}\mathbf{SL}_2$ is  $\R$-simple.  
All semisimple $k$-groups are built from absolutely almost simple groups over number fields \cite[6.21(ii)]{BoTi} and \cite[Prop. A.5.14]{CGP}.  
For the reader's convenience, we record a corollary of \cite[Prop. A.5.14]{CGP} that will be useful in what follows.

\begin{prop}\label{restrictionofscalars}
Let $\mathbf{G}$ be a semisimple $k$-simple $k$-group.  
\begin{enumerate}
\item(Existence) There exists a number field $k'$ containing $k$ and an absolutely almost simple $k'$-group $\mathbf{H}'$ such that $\mathbf{G}$ and $R_{k'/k}(\mathbf{H}')$ are $k$-isogenous.  Furthermore, if $\mathbf{G}$ is adjoint, $\mathbf{G}$ and $R_{k'/k}(\mathbf{H}')$ are $k$-isomorphic.  
\item (Uniqueness)  The pair $(\mathbf{H}',k')$ is unique in the following sense:  If $k''$ is a number field containing $k$, and $\mathbf{H}''$ an absolutely almost simple $k''$-group such that $\mathbf{G}$ and $R_{k''/k}\mathbf{H}''$ are $k$-isogenous, then there is a field isomorphism $\tau:k'\to k''$ and a $k''$-isogeny between $\mathbf{H}'\times_{\tau}\mathrm{spec}\,k'$ and $\mathbf{H}''$.
\end{enumerate}
\end{prop}

\subsection*{Arithmetic Lattices in Semisimple Lie Groups}\label{arithmeticsectionliegroups}${}$

Let $\overline{G}$ be a connected, adjoint, semisimple Lie group with no compact factors.  Let $\Gamma \subset \overline{G}$ be a lattice.  
Then $\Gamma$ is \textbf{arithmetic} if there exists a semisimple algebraic $\Q$-group $\mathbf{G}$ and a surjective analytic homomorphism $\pi: \mathbf{G}(\R)^\circ \to \overline{G}$ with compact kernel such that $\pi(\mathbf{G}(\Z)\cap \mathbf{G}(\R)^\circ)$ and $\Gamma$ are commensurable up to $\overline{G}$-automorphism.  

\centerline{\hfill
%\xymatrixcolsep{5pc}
\xymatrix{
	&\mathbf{G}(\Z)\cap \mathbf{G}(\R)^\circ \ar[d]_\pi \ar[rr]	&&\mathbf{G}(\R)^\circ\ar[d]^\pi\ar[r]&\mathbf{G}(\R)  \\
	&\pi(\mathbf{G}(\Z)\cap \mathbf{G}(\R)^\circ)\ar[r]^-{\sim_{c}}&\varphi(\Gamma) \ar[l]\ar[r]		&\overline{G}&}\hfill}

In what follows, we shall say that \textbf{$\mathbf{G}$ gives rise to $\Gamma$}.    
If $\mathbf{H}\subset \mathbf{G}$ is a $\Q$-simple factor, we will always assume that it is $\R$-isotropic, since otherwise $\mathbf{H}(\R)^\circ\subset \mathrm{ker}(\pi)$, and we may just replace $\mathbf{G}$ with $\mathbf{G}/\mathbf{H}$.  
Observe that if $\Gamma, \Gamma' \subset \overline{G}$ are subgroups that are commensurable up to $\overline{G}$-automorphism and one is an arithmetic lattice, then so is the other.

It may appear as though arithmetic lattices are rather specific and potentially rare type of lattice.  
However, thanks to Margulis's  arithmeticity theorem \cite{Mar} and the work of Gromov and Schoen \cite{GS}, irreducible lattices in groups not locally isomorphic to $\mathbf{SO}(n, 1)$ or $\mathbf{SU}(n, 1)$ are always arithmetic.

\subsection*{Arithmetic Locally Symmetric Spaces}\label{arithmeticsectionsymspaces}${}$

In this section we adopt the following notation:
\begin{itemize}
\item $\overline{G}$ is a connected, adjoint, semisimple Lie group with no compact factors,
\item $\overline{K}\subset \overline{G}$ is a maximal compact subgroup,
\item $\mathbf{G}$ is a semisimple algebraic $\Q$-group with no $\R$-anisotropic $\Q$-simple factors,
\item $\mathbf{G}(\Z)\subset \mathbf{G}(\Q)$ is the lattice stabilizer $G_{\rho, L}$ for some choice of $\rho$ and $L$,
\item $\pi$ is projection $\pi:\mathbf{G}(\R)^\circ\to \overline{G}$ with compact kernel,
\item $\Gamma\subset \overline{G}$ is a subgroup commensurable up to $\overline{G}$-automorphism to $\pi(\mathbf{G}(\Z)\cap \mathbf{G}(\R)^\circ)$,
\item $\varphi\in \mathrm{Aut}(\overline{G})$ is such that $\pi(\mathbf{G}(\Z) \cap \mathbf{G}(\R)^\circ)$ and $\varphi(\Gamma)$ are commensurable,
\item $K\subset \mathbf{G}(\R)$ is a maximal compact subgroup containing $\pi^{-1}(\varphi(\overline{K}))$.
\end{itemize}
An \textbf{arithmetic locally symmetric space (of noncompact type)}
is a space $M$ of the form $\Gamma \backslash \overline{G} / \overline{K}$.  
When $\Gamma$ is torsion-free, $M$ is a Riemannian manifold, and since every $\Gamma$ has a finite index torsion-free subgroup \cite{Sel}, $M$ is always a good Riemannian orbifold in the sense of Thurston \cite[Chp. 13]{Th}.

In this paper, we primarily study totally geodesic subspaces of arithmetic locally symmetric spaces.
As we show in Theorem \ref{totallygeoarearith}, totally geodesic subspaces inherit arithmeticity from its ambient space.
%While this result may be known in the community, we are unaware of references, and as such, we provide a complete proof here.

\begin{thm}\label{totallygeoarearith}
Let $M$ be an arithmetic locally symmetric space and let $N\subset M$ be a nonflat, finite volume, totally geodesic subspace.  
Then $N$ is arithmetic.  
\end{thm}

\begin{proof}
By Lemma \ref{prop24}, there exists a connected, semisimple Lie subgroup $\overline{H}\subset \overline{G}$ with no compact factors such that  $\widetilde{N}:= \overline{H} / (\overline{K} \cap \overline{H})$ is the universal cover of $N$ and $\overline{\Lambda}  \backslash \overline{H} / (\overline{K} \cap \overline{H})$ is commensurable to $N$ where $\overline{\Lambda} := \Gamma \cap \overline{H}$ is a lattice in $\overline{H}$. 
Let $H$ denote the connected component of the intersection of $\pi^{-1}(\varphi^{-1}(\overline{H}))$ with the noncompact factors of $\mathbf{G}(\R)$.  (This group can also be viewed as the unique connected Lie subgroup of $\mathbf{G}(\R)$ with Lie algebra $\mathrm{Lie}(\varphi^{-1}(\overline{H}))$.)    
It follows that 
$N':=\Lambda \backslash H/ (K^\circ \cap H)$, where $\Lambda:=\mathbf{G}(\Z) \cap H$, is commensurable with $N$.   
Arithmeticity is an invariant of commensurability class so it suffices to show the arithmeticity of $N'$.  The result then follows by Lemma \ref{arithmeticsubgrp} below.
\end{proof}

\begin{lem}\label{arithmeticsubgrp}
Let \begin{enumerate}
\item $\mathbf{G}$ be an  semisimple $\Q$-group, 
\item $H\subset \mathbf{G}(\R)$ be a connected semisimple Lie subgroup with no compact factors, and 
\item $\Lambda\subset \mathbf{G}(\Z)$ be a subgroup which is also a lattice in $H$.  
\end{enumerate}
Then $H=\mathbf{H}(\R)^\circ$ where $\mathbf{H}\subset \mathbf{G}$ is a semisimple $\Q$-subgroup and $\Lambda\subset \mathbf{H}(\Q)$ is arithmetic.
\end{lem}

\begin{proof}
%Let $\Gamma'\subset \Gamma$ be the finite index subgroup whose intersection with the compact factors is trivial.  
Since $H$ is a semisimple Lie group sitting inside the real points of a linear group, $H$ is the connected component of the real points of some semisimple $\R$-subgroup $\mathbf{H}\subset \mathbf{G}$.  
By Borel's Density Theorem  \cite{B60} $\Lambda$ is Zariski dense in $\mathbf{H}$.  
%Let $\mathbf{H}'$ be the Zariski closure of $\Lambda$ in $\mathbf{G}$.  
The Zariski closure of an abstract subgroup sitting inside the $\Q$-points of a group is also a $\Q$-group \cite[Chp 1 Prop 1.3(b)]{B1}.  
Hence $\mathbf{H}$ is defined over $\Q$.  
Now let $V:=\mathrm{Lie}(\mathbf{G})$ and $W:=\mathrm{Lie}(\mathbf{H})$.  
The adjoint representation $\mathrm{Ad}: \mathbf{G}\to \mathbf{GL}(V)$ is defined over $\Q$.
There exists a lattice $L\subset V$ which $\Gamma$ stabilizes \cite[Prop 7.12]{B}.  
Since $\Lambda$ stabilizes $W$, it stabilizes $L\cap W$ and hence $\Lambda$ is an arithmetic subgroup of $H$.
\end{proof}

%In our analysis, we need to be able to algebraically detect when two totally geodesic subspaces are commensurable.
If $\mathbf{G}_1$ and $\mathbf{G}_2$ are absolutely simple algebraic groups over number fields $k_1$ and $k_2$, respectively,  by \cite[Prop 2.5]{PR}, they give rise to commensurable arithmetic groups if and only if there is a field isomorphism $\tau:k_1\to k_2$ such that $\mathbf{G}_2$ and $\mathbf{G}_1\times_{\tau}\mathrm{spec}\,k_2$ are isomorphic as $k_2$-groups.
We now give the following slight generalization of \cite[Prop 2.5]{PR} that is useful when looking for totally geodesic subspaces.

\begin{prop}\label{commimpliesisothm}
Let $M_1$ and $M_2$ be arithmetic locally symmetric spaces arising from semisimple $\Q$-groups $\mathbf{G}_1$ and $\mathbf{G}_2$ respectively.  
Then $M_1$ and $M_2$ are commensurable if and only if  
$\mathbf{G}_1$ and $\mathbf{G}_2$ are $\Q$-isogenous.
%$\mathrm{Ad}_{\mathbf{G}_1}(\mathbf{G}_1)$ and $\mathrm{Ad}_{\mathbf{G}_2}(\mathbf{G}_2)$ are $\Q$-isomorphic.
\end{prop}

\begin{proof}
First suppose $\mathbf{G}_1$ and $\mathbf{G}_2$ are $\Q$-isogenous.  Then $\mathrm{Ad}_{\mathbf{G}_1}(\mathbf{G}_1)$ and $\mathrm{Ad}_{\mathbf{G}_2}(\mathbf{G}_2)$ are $\Q$-isomorphic via a $\Q$-isomorphism $\psi$.
Since $M_i$ is commensurable with $\mathrm{Ad}_{\mathbf{G}_i}(\mathbf{G}_i(\Z)) \backslash \mathrm{Ad}_{\mathbf{G}_i}(\mathbf{G}_i(\R)) / \mathrm{Ad}_{\mathbf{G}_i}(K_i)$.
The result then immediately follows from the fact that $\psi(\mathrm{Ad}_{\mathbf{G}_1}(\mathbf{G}_1(\Z)))$ and $\mathrm{Ad}_{\mathbf{G}_2}(\mathbf{G}_2(\Z))$ are commensurable \cite[Cor 7.13(2)]{B}.
%By assumption, there are $\Gamma'_i\subset \mathbf{G}_i(\Q)$ arithmetic such that $M_i$ is commensurable with $\Gamma_i \backslash \mathbf{G}_i(\R)^\circ / K_i$, and the result then immediately follows from the fact that $\varphi(\Gamma_1)$ and $\Gamma_2$ are commensurable \cite[Cor 7.13(2)]{B}.
 
Now suppose $M_1$ and $M_2$ are commensurable.  By assumption, there exists a connected adjoint semisimple Lie group with no compact factors, $\overline{G}$, and two arithmetic lattices $\Gamma_1,\Gamma_2\subset \overline{G}$ which are commensurable up to $\overline{G}$-automorphism, say $\psi$, such that $M_1=\Gamma_1 \backslash \overline{G}/ K$ and $M_2=\Gamma_2 \backslash \overline{G}/ \psi(K)$ where $K$ is a maximal compact subgroup.  Replacing $\mathbf{G}_i$ with $\mathrm{Ad}_{\mathbf{G}_i}(\mathbf{G}_i)$, the result then follows from Lemma \ref{commimpliesiso} below.
\end{proof}

\begin{lem}\label{commimpliesiso}
Let $\overline{G}$ be a connected adjoint semisimple Lie group with no compact factors.  Let $\Gamma_1, \Gamma_2\subset \overline{G}$ be arithmetic lattices which are commensurable up to $\overline{G}$-automorphism.  Let $\mathbf{G}_1$ and $\mathbf{G}_2$ be the connected adjoint semisimple $\Q$-groups with no $\R$-anisotropic $\Q$-simple factors giving rise to $\Gamma_1$ and $\Gamma_2$ respectively.  Then $\mathbf{G}_1$ and $\mathbf{G}_2$ are $\Q$-isomorphic.
\end{lem}

\begin{proof}
Let $\psi$ be an analytic automorphism of $\overline{G}$ for which $\Gamma_1$ and $\psi(\Gamma_2)$ are commensurable.    
Let $\mathbf{H}_i\subset \mathbf{G}_i$ be the product of the connected $\R$-simple $\R$-isotropic components of $\mathbf{G}_i$.  
Then $\pi_i|_{\mathbf{H}_i(\R)^\circ}:\mathbf{H}_i(\R)^\circ\to \overline{G}$ is an isomorphism.
Picking sufficiently small finite index $\Gamma'_i\subset \Gamma_i$ which are isomorphic via $\psi$,  we may identify $\varphi_i^{-1}(\Gamma_i)$ with a finite index subgroup $\Lambda_i:=\pi_i|_{\mathbf{H}_i(\R)^\circ}^{-1}(\varphi_i^{-1}(\Gamma_i))\subset \mathbf{H}_i(\R)^\circ\cap \mathbf{G}_i(\Q)$.  

Since $\pi_i$ induces an $\R$-rational isomorphism between $\mathbf{H}_i$ and $\mathbf{Aut}(\mathrm{Lie}(\overline{G})\otimes_\R \C)$, and $\psi$ induces an $\R$-rational automorphism on  $\mathbf{Aut}(\mathrm{Lie}(\overline{G})\otimes_\R \C)$, it follows that there is an $\R$-rational isomorphism, which we also denote $\psi$, from $\mathbf{H}_1$ to $\mathbf{H}_2$ which sends $\Lambda_1$ to $\Lambda_2$.

For each $i$, by \cite[6.21 (ii)]{BoTi}, $\mathbf{G}_i\cong \prod_{j=1}^{r_i} R_{k_{i,j}/\Q} \mathbf{S}_{i,j}$ where $\mathbf{S}_j$ is an absolutely simple group over a number field $k_{i,j}$.  
Then $\Lambda_{i,j}:=\Lambda_i \cap (R_{k_{i,j}/\Q}(\mathbf{S}_{i,j}))(\Q)$ is an arithmetic group in $(R_{k_{i,j}/\Q}(\mathbf{S}_{i,j}))(\Q)=\mathbf{S}_{i,j}(k_{i,j})$ \cite[6.11]{BoHC}.  
Borel's Density Theorem \cite{B65} implies that  $\Lambda_{i,j}$ is Zariski dense in $\mathbf{S}_{i,j}$.
Since each $\Lambda_{i,j}$ is a normal subgroup of $\Lambda_i$ and an irreducible lattice in $(R_{k_{i,j}/\Q}(\mathbf{S}_{i,j}))(\R)$, the isomorphism $\psi$ must send each $\Lambda_{1,j}$ to some $\Lambda_{2, j'}$, from which we conclude $r_1=r_2:=r$ and $\psi$ induces a permutation also denoted $\psi\in S_r$.  
Our assumption on $\Q$-simple factors implies that each $R_{k_{i,j}/\Q} \mathbf{S}_{i,j}$ contains an $\R$-simple $\R$-isotropic factor.  
Since $\psi$ sends $\R$-isotropic $\R$-simple factors of $R_{k_{1,j}/\Q} \mathbf{S}_{1,j}$ to $\R$-isotropic $\R$-simple factors of $R_{k_{2,\psi(j)}/\Q}( \mathbf{S}_{2,\psi(j)})$, we conclude $\mathbf{S}_{2,j}$ and $\mathbf{S}_{2,\psi(j)}$ have the same Killing--Cartan type.  
Let $\mathbf{H}_{i,j}$ be a fixed $\R$-simple $\R$-isotropic component of $R_{k_{i,j}/\Q} \mathbf{S}_{i,j}$. 
Then $\psi$ induces an $F$-isomorphism between $\mathbf{S}_{1,j}$ and $\mathbf{S}_{2,\psi(j)}$, where $F=\R$ when $\mathbf{H}_{1,j}$ is absolutely simple, and $F=\C$ otherwise.  
Furthermore, this isomorphism sends $\Lambda_{1,j}$ to $\Lambda_{2,\psi(j)}$, hence by \cite[Prop 2.5]{PR}, $k_{1,j}$ and $k_{2,\psi(j)}$ are isomorphic and, letting $k_j$ denote this isomorphism class (and changing the base of these groups), $\mathbf{S}_{1,j}$ and $\mathbf{S}_{2,\psi(j)}$ are $k_j$-isomorphic.  
The conclusion follows.
\end{proof}

%-----------------------------------------------------------------------------------------------------------
%---------------------------QUADRATIC FORMS SECTION-------------------------------------
%-----------------------------------------------------------------------------------------------------------

\section{Arithmetic Locally Symmetric Spaces Arising From Quadratic Forms}\label{sectionquadforms}

In this section we discuss the theory of quadratic forms and the results we need to construct and analyze arithmetic locally symmetric spaces coming from quadratic forms.  
For a complete treatment of the classical theory of quadratic forms over local and global fields, we refer the reader to \cite{OM}, \cite{Sch}, and \cite{Lam}.   

Recall $F$ is a field that is not of characteristic 2.  
In what follows, $(V,q)$ will denote a quadratic space over $F$ where $V$ is a finite dimensional vector space over $F$ and $q$ is a quadratic form on $V$.  
When it will not cause confusion, we will omit $V$ and simply refer to the quadratic form $q$. 
We shall say $q$ is a quadratic form over $F$, or more succinctly, $q$ is a quadratic $F$-form.   
If $E/F$ is a field extension then $(V,q)$ determines a quadratic space $(V_E, q_E)$ over $E$ by extending scalars (i.e., where $V_E:= V\otimes_FE$ and $q_E$ is the extension of $q$ to $V_E$).  
When it will not cause confusion, we will sometimes denote the extended form by the symbol $q$ as well.  
Every quadratic space $(V,q)$ determines an algebraic $F$-group, $\mathbf{SO}(V,q)$ whose $E$ points are given by
$$\mathbf{SO}(V,q)(E)=\{T\in \mathbf{SL}(V_E)\ | \ q_E(Tv)=q_E(v) \mbox{ for all } v\in V_E\}.$$ 
\begin{Def}Let $(V_1,q_1)$ and $(V_2,q_2)$ be quadratic spaces over $F$.  
Then $q_1$ and $q_2$ are 
\begin{enumerate}
\item \textbf{isometric} if there some $F$-linear isomorphism $T:V_1\to V_2$ such that $q_2(Tv)=q_1(v)$ for all $v\in V_1$.
\item \textbf{similar} if there exists some $a\in F^\times$ such that $q_1$ and $aq_2$ are isometric. 
\item  \textbf{isogroupic} if $\mathbf{SO}(q_1)$ and $\mathbf{SO}(q_2)$ are isomorphic as algebraic $F$-groups.
\end{enumerate}
\end{Def}

The first two definitions are standard, while the third we introduce in the paper.  
It is not hard to see that each of these determine an equivalence relation among quadratic $F$-forms.  
Furthermore, the following lemma begins to shows how they are related.

\begin{lem}\label{isosimrep}${}$
\begin{enumerate}
\item Isometric forms are isogroupic.  
\item Similar forms are isogroupic.
\end{enumerate}
\end{lem}

\begin{proof}${}$

\noindent \textit{(1)}  
Let $(V_1,q_1)$ and $(V_2,q_2)$ be isometric forms.  
By assumption there exists an $F$-linear isomorphism $T:V_1\to V_2$ preserving the forms.
Then $T$ induces an $F$-isomorphism $T_*:\mathbf{SL}(V_1)\to \mathbf{SL}(V_2)$ via $g\mapsto TgT^{-1}$.  
Upon restricting to $\mathbf{SO}(V_1,q_1)$, for any $v\in V_2$, we have
$$q_2(T_*(g)v)=q_2\big((TgT^{-1})v\big)=q_2\big(T(g(T^{-1}v))\big)=q_1(g(T^{-1}v))=q_1(T^{-1}v)=q_2(v).$$
Hence $T_*(\mathbf{SO}(V_1,q_1))\subset \mathbf{SO}(V_2, q_2)$ and by symmetry of argument it follows they are $F$-isomorphic.\\

\noindent  \textit{(2)}  
Let $(V_1,q_1)$ and $(V_2,q_2)$ be similar forms.   
By assumption there exists $a\in F^\times$ such that $aq_1$ and $q_2$ are isometric.
By part (1), it suffices to show that $aq_1$ and $q_1$ are isogroupic.
Pick $g\in \mathbf{SO}(V_1, q_1)$ and $v\in V_1$, then
$$(aq_1)(gv)=a(q_1(gv))=a(q_1(v))=(aq_1)(v).$$
Therefore $g\in \mathbf{SO}(V_1, aq_1)$, and by symmetry of argument, $\mathbf{SO}(V_1,q_1)=\mathbf{SO}(V_1, aq_1)$.  
The result follows.
\end{proof}

In general there are many isometry classes in a given isogroupy class.  
If $\mathbf{G}:=\mathbf{SO}(q)$, then any $q'$ in the isogroupy class of $q$ shall be said to \textbf{represent} $\mathbf{G}$.  

A quadratic form $r$ is a \textbf{subform} of a quadratic form $q$  if there is some third form $t$ such that $r\oplus t$ is isometric to $q$.   
We say a symmetric bilinear form $b$ is \textbf{nondegenerate} when $b(v,w)=0$ for all $w\in V$ implies that $v=0$.  
A quadratic form corresponding to a nondegenerate symmetric bilinear form is said to be \textbf{regular}.  
In this paper, all quadratic forms will be assumed to be regular.  
The \textbf{dimension} of $q$, denoted $\dim q$, is the dimension of its associated vector space.  
When possible, we shall reserve the symbol $m$ to denote the dimension of $q$.   
Upon choosing a basis, every quadratic form may be represented by an  $m\times m$ matrix.  
The \textbf{determinant} of $q$, denoted $\det q$, is the determinant of some $Q\in \mathbf{GL}_m(F)$ representing $q$.  
Note however that since this should be be independent of the choice of basis and $\det ({}^tTQT)=\det Q (\det T)^2$, the determinant is only well defined up to square class of $F$, and hence we view $\det q\in F^\times/(F^\times)^2$.  
Though the determinant is a square class, we will often omit the $(F^\times)^2$ and write $\det q =a$ as opposed to $\det q =a (F^\times)^2$, where $a\in F^\times$.  
A common renormalization of the determinant is the \textbf{discriminant},  denoted $\mathrm{disc}(q)$, where $\mathrm{disc}(q)=(-1)^{\dim(q)(\dim(q) -1)/2}\det(q)$.  
It contains the same information as the determinant if one knows the dimension, but often results in simpler expressions.

For $a,b\in F^\times$, 
the \textbf{Hilbert symbol} $\left(\frac{a,b}{F}\right)=(a,b)_F$ denotes the isomorphism class of the quaternion algebra generated by symbols $i$ and $j$ where $i^2=a, j^2=b$, and $ij=-ji$.  
When the field $F$ is understood, we simply write $(a,b)$.  
The Hilbert symbol satisfies the following four properties that will be used frequently in this paper:
\begin{enumerate}[\qquad(H1)]
\item Defined up to square class: $(a,bc^2)=(a,b)$,
\item Symmetry: $(a,b)=(b,a)$,
\item Multiplicativity: $(a_1a_2,b)=(a_1,b)(a_2,b)$,
\item Nondegeneracy: For $a\in F^\times$ not a square, there exists a $b\in F^\times$ such that $(a,b)\ne 1$.\end{enumerate}

Given a diagonal representation $\langle a_1, a_2, \ldots, a_m\rangle$ of $q$ \cite[I.2.4]{Lam},
the \textbf{Hasse invariant} $c(q)$ is defined to be
\begin{equation}
c(q):=\begin{cases} \prod_{i< j}(a_i,a_j) & \mbox{if $m\ge2$, and}\\  1 &\mbox{if $m=1$.}\end{cases}
\end{equation}
While this definition is common \cite{Cassels,Lam}, some authors use different normalizations of this invariant.
In \cite{OM}, the Hasse invariant is defined to be
$c_{OM}(q):=\prod_{i\le j}(a_i,a_j)$, and in 
\cite{Belolipetsky}, the Hasse invariant $\epsilon_{HW}$ is normalized so that the split quadratic form $h$ always has $\epsilon_{HW}(h)=1$.
It follows that 
\begin{equation}
c_{OM}(q)=c(q)(-1,\det q),
\end{equation}
and, by our computations in Section \ref{sectiontits}, which are summarized in Table \ref{tablequadratic}, for $m\ge4$,
\begin{equation}
\epsilon_{HW}(q)=\begin{cases}
c(q)(-1,-1)^{\frac{n(n-1)}{2}} & \mbox{ if } m=2n, \\
c(q)(-1,-1)^{\frac{n(n-3)}{2}}(-1,\det q)^n & \mbox{ if } m=2n+1.
\end{cases}
\end{equation}
The Hasse invariant satisfies the following useful product formula:
\begin{align}\label{productformula}
c(q_1\oplus q_2)= c(q_1)c(q_2)(\det q_1, \det q_2).
\end{align}

While the Hasse invariant is a well defined invariant of the isometry class of $q$  \cite[V.3.8]{Lam}, 
it is not an invariant of the similarity class of $q$.
The relationship between $c(q)$ and $c(\lambda q)$, $\lambda\in F^\times$, is given by the following lemma.  

\begin{lem}\label{simc}
Let $F$ be a field of characteristic not 2, let $q$ be a quadratic form over $F$ of dimension $m$, and let $\lambda\in F^\times$.  Then $$c(\lambda q)=\left(\lambda,(-1)^{\frac{m(m-1)}{2}}(\det q)^{m-1}\right)c(q).$$
In particular this reduces to 
\begin{align}
c(\lambda q)
&=\begin{cases}(\lambda,\mathrm{disc}(q))\ c(q)& \mbox{when $m$ is even,}\\ \left(\lambda, (-1)^{\frac{m-1}{2}}\right)\ c(q)& \mbox{when $m$ is odd.}\end{cases}
\end{align}
\end{lem}

\begin{proof}
A direct computation gives:
\begin{align*}
c(\lambda q)	&=\prod_{i<j}(\lambda a_i, \lambda a_j)\\
	&=\prod_{i<j}(\lambda, \lambda)(\lambda,a_i)(\lambda,a_j)(a_i,a_j)\\
	&=(\lambda,-1)^{\frac{m(m-1)}{2}}(\lambda,\det q)^{m-1}c(q)\\
	&=\left(\lambda,(-1)^{\frac{m(m-1)}{2}}(\det q)^{m-1}\right)c(q).
\end{align*}
The reduction when $m$ is even and odd immediately follows.
\end{proof}

The extent to which the Hasse invariant varies within an isogroupy class will be explored in Section \ref{sectiontits}.
In general the Hasse invariant is difficult to compute, however, when $F$ is a nonarchimedean local field or $\R$, then $c(q)$ can only take values $\pm 1$, and over $\C$,  $c(q)$ is identically $1$.  

Every isometry class of quadratic forms over $\R$ can be diagonally represented 
with the first $m_+$ terms positive and the remaining $m_-:=m-m_+$ terms negative.  
The \textbf{signature} of $q$ is the pair $\mathrm{sgn}(q):=(m_+, m_-)$.    
Some authors define to the signature of $q$ to be the number $s=m_+-m_-$.  
Observe that the two pairs $(m,s)$ and $(m_+, m_-)$ contain equivalent information.
The signature is an invariant of the isometry class of $q$, and the unordered pair $\{m_+, m_-\}$ is an invariant of the similarity class of $q$.

These invariants determine the isometry classes of quadratic forms over local and global fields.
For the reader's convenience, we state the uniqueness and existence theorems for quadratic forms over local and global fields.  
These will be essential in our analysis in later sections.  

\begin{thm}[Local Uniqueness]\label{localuniqueness}Let $F$ be $\C$, $\R$, or a finite extension of $\Q_p$ that we denote $L$,  and  let $q$ and $q'$ be quadratic $F$-forms.  Then $q$ and $q'$ are isometric if and only if
\begin{enumerate}
\item When $F=\C$,  $\dim q=\dim q'$.
\item When $F=\R$,  $\dim q=\dim q'$ and $\mathrm{sgn}(q)=\mathrm{sgn}(q')$.
\item When $F=L$,  $\dim q=\dim q'$, $\det(q)=\det(q')$, and $c(q)=c(q')$.
\end{enumerate}
\end{thm}

\begin{thm}[Local Existence]\label{localexistence}${}$
\begin{enumerate}
\item For each $m\in \Z_{\ge 1}$, there exists a quadratic $\C$-form $q$ such that 
$$\dim q = m.$$
\item For each pair $(m_+,m_-)\in \Z_{\ge 0}\times \Z_{\ge 0}$, there exists a quadratic $\R$-form $q$ such that 
$$\dim q = m:=m_++m_-\quad \mbox{ and }\quad \mathrm{sgn}(q)=(m_+,m_-).$$
\item For each triple $(m,d,c)\in \Z_{\ge 1} \times L^\times/(L^\times)^2\times \{\pm1\}$, there exists a quadratic $L$-form $q$ such that 
$$\dim q = m, \quad\det q= d \quad \mbox{ and } \quad c(q)=c,$$ 
\phantomsection
\label{exception}
$\mathbf{(*)}$ with the exception that $c=1$ when either $m=1$ or  $m=2$ and $d=-1$.
\end{enumerate}
\end{thm}

While the  \hyperref[exception]{exceptional restrictions $(*)$} on the Hasse invariant in dimensions $m=1$ and $m=2$ may seem inconsequential, they will play an integral role in the construction of subforms later in the paper.  
For more on the above results over $\R$ and $L$, we refer the reader to \cite{Sy} and \cite[VI.63:23]{OM}, respectively.

\begin{thm}[Local-to-Global Uniqueness]\label{localglobaluniqueness} \cite[VI.66:4]{OM}
Let $k$ be a number field and 
 $q$ and $q'$ be quadratic $k$-forms.   
 Then $q\cong q'$ if and only if $q\otimes k_v\cong q'\otimes k_v$ for all $v\in V_k$.
\end{thm}

\begin{thm}[Local-to-Global Existence]\label{localglobal} \cite[VII.72:1]{OM}  Let $k$ be a number field and let
\begin{itemize}
\item $m\in \Z_{\ge 1}$,
\item $d\in k^\times/(k^\times)^2$, and 
\item $S\subset V_k$ be a finite subset of even cardinality.
\end{itemize}
For each family $\{q_v\}_{v\in V_k}$ where $q_v$ is a quadratic form over $k_v$ satisfying 
\begin{itemize}
\item $\dim q_v=m$,
\item $\det q_v= d$, and
\item $c_v(q_v)=-1$ if and only if $v\in S$,  
\end{itemize}there exists a quadratic form $q$ over $k$ such that $q\otimes k_v=q_v$ for all $v\in V_k$.
\end{thm}

Quadratic forms are used to construct irreducible arithmetic lattices of semisimple Lie groups of the form $$G=\prod_{i=1}^r \mathbf{SO}(p_i,m-p_i) \times (\mathbf{SO}_m(\C))^s.$$
In the literature, these lattices have been called: ``standard'' \cite{Lub97}, ``lattice of the simplest type'' \cite{Vin2}, and, ``coming from quadratic forms.'' 
%\cite{M13}.  
We shall use the terminology ``standard'' when convenient, or otherwise we shall say explicitly ``coming from quadratic forms.''

\begin{con}\label{quadcon}Fix the following notation:
\begin{enumerate}
\item $k$ is a number field with infinite places $V_k^\infty$; % =\{v_1, \ldots, v_l\}$.
\item $(V,q)$ is an $m$-dimensional quadratic $k$-space, $m\ge 3$; 
\item $\mathbf{G}:=\mathbf{SO}(V,q)$ is the absolutely almost simple $k$-group defined by $(V,q)$ and $SO(q):=\mathbf{G}(k)$;
\item For each $v\in V_k^\infty$,   $V_{k_{v}}:=V\otimes_kk_{v}$, $q_{v}:=q\otimes k_{v}$, and $\mathbf{G}_{v}$ is the algebraic $k_{v}$-group $\mathbf{SO}(V_{k_{v}},q_{v})$,
\begin{itemize}
\item If $v$ is real, then $\mathbf{G}_{v}(k_v)\cong \mathbf{SO}(m_+^{(v)}, m_-^{(v)})$,  
\item If $v$ is complex, then $\mathbf{G}_{v}(k_v)\cong \mathbf{SO}_m(\C)$,
%\end{itemize}
%%\item $(m_+^{(v)}, m_-^{(v)})$ is the signature of $(V_{k_{v}},q_{v})$ at the real place $v\in V_k^\infty$.
%%\item  where, for each  $v\in V_k^\infty$,
%\item Letting: 
%	\begin{itemize}
		\item $r$ is the number of real places where $q$ is isotropic,  
		\item $s$ is the number of complex places, and
		\item $p_i:=m_+^{(v_i)}$ where $\{v_1, \ldots, v_r\}$ is the set  of real places where $q$ is isotropic;
	\end{itemize} 
\item $\mathbf{G}':=R_{k/\Q}\mathbf{G}$  is the semisimple $\Q$-group formed by restriction of scalars.   
Then $\mathbf{G}'(\R)=\prod \mathbf{G}_{v}(k_v)$ is a semisimple Lie group that has compact factors at precisely the real places where $q$ is anisotropic.  
There is an isomorphism $\mathbf{G}(k)\cong \mathbf{G}'(\Q)$ and diagonal embedding $SO(q)\to  \mathbf{G}'(\R)$;
\item $G$ is the projection of $\mathbf{G}'(\R)$ onto its noncompact factors and denote the projection map by $\pi:\mathbf{G}'(\R)\to G$.   
 Observe that $G$ is a semisimple Lie group with no compact factors and is $\R$-simple when $r+s=1$;
\item $L\subset V$ is an $\mathcal{O}_k$-lattice and $G_L:=\{T\in \mathbf{G}(k) \ | \ T(L)\subset L\}$. Then $G_L$ sits as a discrete arithmetic subgroup of the semisimple Lie group $\mathbf{G}'(\R)$;
\item $\Gamma\subset G$ is commensurable up to $G$-automorphism with $\pi(G_L)$.  Then $\Gamma$ is said to be a \textbf{standard arithmetic lattice of $G$}.  Figure \ref{figure:standarddiagram} below  summarizes this construction.
%\item Figure \ref{figure:standarddiagram} below illustrates the construction of standard arithmetic lattices in $G$.
%\vspace{-1pc}
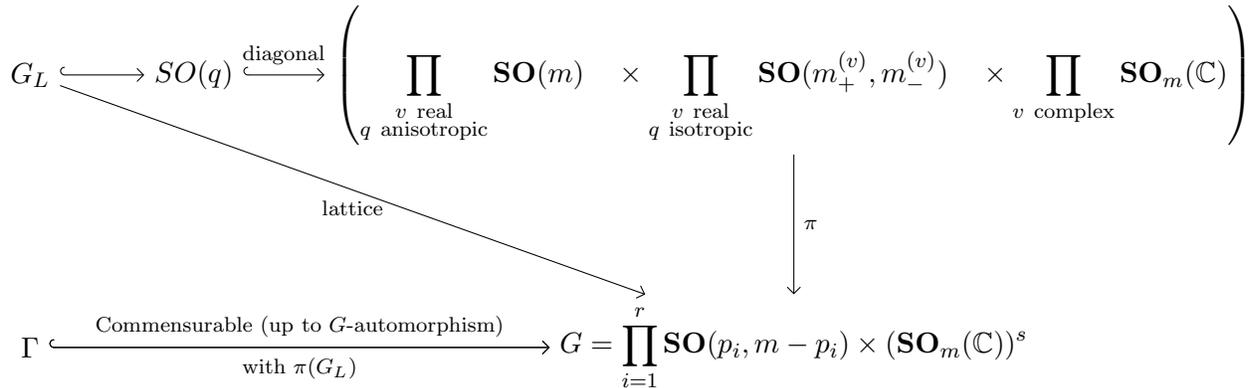
\begin{figure}[htbp]
%\hspace*{-1pc} 
\begin{tikzpicture}[scale=1.45]
\node (A) at (-1,1) {$SO(q)$};
\node (B) at (4.5,1) {$\left(\displaystyle\prod\limits_{\substack{v \ \mathrm{real} \\ q\ \mathrm{anisotropic}}} \mathbf{SO}(m) \quad \times \displaystyle\prod\limits_{\substack{v \ \mathrm{real} \\ q\ \mathrm{isotropic}}} \mathbf{SO}(m_+^{(v)}, m_-^{(v)}) \quad \times \displaystyle\prod\limits_{v\ \mathrm{complex}}\mathbf{SO}_m(\C)\right)$};
\node (C) at (-2.5,1) {$G_L$};
\node (D) at (4.5,-1.5) {$G=\displaystyle\prod\limits_{i=1}^r \mathbf{SO}(p_i,m-p_i) \times (\mathbf{SO}_m(\C))^s$};
\node (E) at (-2.5,-1.5) {$\Gamma$};
\path[right hook->,font=\scriptsize,>=angle 90]
(A) edge node[above]{diagonal} (B)
(C) edge node[above]{} (A)
(E) edge node[above]{Commensurable (up to $G$-automorphism)} (D)
(E) edge node[below]{with $\pi(G_L)$} (D);
\path[->,font=\scriptsize,>=angle 90]
(C) edge node[below]{lattice} (D)
(B) edge node[right]{$\pi$} (D);
\end{tikzpicture}
\vspace{-1pc}
		\caption{Construction of Standard arithmetic lattices in $G$.}
		\label{figure:standarddiagram}
\end{figure}	
%\vspace{-1pc}
\item  $K\subset G$ is a maximal compact subgroup and $M_{\Gamma}:=\Gamma \backslash G /K$. 
\begin{enumerate}
	\item $M_\Gamma$ is an \textbf{arithmetic locally symmetric space coming from a quadratic form}, 
	(or a  \textbf{standard arithmetic locally symmetric space of type $B_n$ or $D_n$}), 
	and
	\phantomsection
	\label{fieldalgebradef}
	\item $k(M_\Gamma):=k$ is the \textbf{field of definition} of $M_{\Gamma}$.
\end{enumerate} 
%
%\item Let  $K\subset G$ its maximal compact subgroup and let $M_{\Gamma}:=\Gamma \backslash G /K$.  This space $M_\Gamma$ is an \textbf{arithmetic locally symmetric space coming from a quadratic form.}  We call $M_\Gamma$ \textbf{simple} if $G$ is simple as a Lie group (i.e., $r+s=1$).
\end{enumerate}
\end{con}

When $M_\Gamma$ is simple (i.e. $r=1$ and $m\ne 4$),  \cite[Lemma 2.6]{PR} implies that $k(M_\Gamma)$ coincides with the minimal field of definition of $\Gamma$ in the sense of Vinberg \cite{Vin}.
A choice of another $\mathcal{O}_k$-lattice $L'\subset V$ and $\Gamma'$ commensurable up to $G$-automorphism with $\pi(G_{L'})$ will produce a space $M_{\Gamma'}$ that is commensurable with $M_{\Gamma}$. 
Hence choosing the pair $(k,q)$ determines a commensurability class which we will sometimes denote by $M_q$.  
Conversely, if the pairs $(k_1,q_1)$ and $(k_2,q_2)$ yield commensurable spaces, by \cite[Prop. 2.5]{PR},
there is a field isomorphism $\tau:k_1\to k_2$ such that $\mathbf{SO}(q_2)$ and $\mathbf{SO}(q_1\otimes_{\tau}k_2)$ are $k_2$-isomorphic groups.

\begin{Def}
If $(V_1,q_1)$ is a quadratic space over $k_1$ and $\tau:k_1\to k_2$ is a field isomorphism,
then $(V_2,q_2):=(V_1\otimes_\tau k_2,q_1\otimes_\tau k_2)$ is a quadratic space over $k_2$.
We call such base change a \textbf{twist by $\tau$}.
If such a $\tau$ is implicit, we just say $(V_2,q_2)$ is a \textbf{twist} of $(V_1,q_1)$.
\end{Def}

By our remarks above, a twist of $(V,q)$ and $(V,q)$ yield commensurable arithmetic lattices.
In particular, a commensurability class of standard arithmetic lattices of type $B_n$ or $D_n$ is uniquely determined by a twist class of an isogroupy class of quadratic $k$-forms.
Given a number field $k$ with a fixed real (resp. complex) place $v_0$, and a quadric $k$-form $q$ isotropic at a unique real (resp. complex) place, there is always a twist $q'$ of $q$ that is isotropic at $v_0$.

\begin{Def}
We call $(k,q)$ an \textbf{admissible hyperbolic pair} if $M_q$ is a commensurability class of hyperbolic orbifolds, that is to say, 
\begin{enumerate}
\item $k$ is totally real (i.e., no $\mathbf{SO}_{m}(\C)$ terms),
\item $q$ is anisotropic at all but one real place, $v_0$ (i.e., only one $\mathbf{SO}(p_i, m-p_i)$ term),
\item $q\otimes k_{v_0}$ has signature $(m-1,1)$ or $(1,m-1)$.  (i.e., $G\cong\mathbf{SO}(m-1,1)$).
\end{enumerate}
\end{Def}

When $m\ge 3$ is odd, all irreducible arithmetic lattices of $G$ arise from Construction \ref{quadcon} (\cite{T} \cite[\S3]{Lub97}).
When $m>3$, $m\neq 8$, is even, all other are irreducible arithmetic lattices of $G$ come from skew Hermitian forms over quaternion division algebras over number fields (See Construction \ref{hermcon}).  
When $m=8$, in addition to lattices coming from skew Hermitian forms, there are also lattices which come from triality.

\begin{con}\label{subformsubspacedef}
Let $(W,r)$ be a quadratic $k$-subspace of $(V, q)$.  
Then $\mathbf{H}=\mathbf{SO}(W,r)$ is an absolutely almost simple $k$-subgroup of $\mathbf{G}$. 
Let  $\mathbf{H}':=R_{k/\Q}\mathbf{H}$.  
Then $\mathbf{H'}$ is a semisimple $\Q$-subgroup of $\mathbf{G}'$.  
It follows that $L\cap W$ is an $\mathcal{O}_k$-lattice of $W$, hence $G_L\cap \mathbf{H}'(\R)$ is an arithmetic subgroup of $\mathbf{H}'(\R)$.  
Let $H$ be the image of $\mathbf{H}'(\R)$ under the projection map $\pi$ onto the noncompact factors of $\mathbf{G}'(\R)$.  Then $\pi(G_L\cap \mathbf{H}'(\R))$ is an arithmetic subgroup of $H$.  
Note that $H$ may be trivial.  It follows that $N_{\Gamma\cap H}:=(\Gamma\cap H)\backslash H/(H\cap K)$ is commensurable to a totally geodesic subspace of $M_\Gamma$.  
We denote this commensurability class $N_r$.  
In what follows, we shall call such totally geodesic subspaces \textbf{subform subspaces}.
Observe that for a subform subspace $N\subset M$, $k(N)=k(M)$.
Furthermore, if $\dim r\ge 2$ and $r$ is isotropic at a real place of $k$, then $N_r$ is a commensurability class of nontrivial, nonflat, finite volume, locally symmetric spaces of noncompact type.
\end{con}

%-----------------------------------------------------------------------------------------------------------
%-------------------------------- TITS INDEX SECTION -------------------------------------------
%-----------------------------------------------------------------------------------------------------------

\section{The Index of Isometry Groups of Quadratic Forms}\label{sectiontits}

Let $\mathbf{G}$ be an absolutely almost simple algebraic $k$-group. 
We will use the conventions of \cite{T} and denote the \textbf{Tits index} of $\mathbf{G}$ by ${}^gX^{(d)}_{n,r}$ where:
\begin{enumerate}
\item $X_n$ is the Killing--Cartan type of $\mathbf{G}\otimes \overline{k}$,
\item $n$ is the $\overline{k}$-rank of $\mathbf{G}$
\item $g$ is the order of the image of the $*$-action map
\item $r$ is the $k$-rank of $\mathbf{G}$, and
\item $d$ is the degree of a division algebra associated with $\mathbf{G}$. 
\end{enumerate}

When $\mathbf{G}=\mathbf{SO}(q)$, $g=1$ or $g=2$ (depending on whether $\mathbf{G}$ in an inner or outer form), and $d=1$.
The Tits index encodes a large amount of information about a semisimple algebraic group's isogeny class  \cite[Theorem 2.7.1]{T}. 
We refer the reader to \cite{T} for more information.
One of the goals of this section is to relate the local index of $\mathbf{SO}(q)$ to the local invariants of $q$ (see Table \ref{tablequadratic}).

For the reader's convenience, we state two basic results relating a quadratic form's invariants to whether or not it is isotropic.

\begin{lem}\cite[Chp. 4 Lem 2.5 \& Lem 2.6]{Cassels}\label{34iso}
Let $L$ be a nonarchimedean local field.
\begin{enumerate}
\item Let $q'$ be a 3-dimensional quadratic form over $L$.  Then $q'$ is isotropic if and only if $c(q')=(-1,-\det q')$.
\item Let $q'$ be a 4-dimensional quadratic form over $L$.  Then $q'$ is anisotropic if and only if $\mathrm{disc}(q')=1$ and $c(q')=-(-1,-1)$.
\end{enumerate}
\end{lem}

Though the proofs in \cite{Cassels} are explicitly written with $k=\Q$, they are generalizable to an arbitrary number field.  

\begin{prop}
Let $k$ be a number field, $v\in V_k$ be a finite place, and $q$ be a $(2n+1)$-dimensional quadratic form over $k_v$.  
Then the local index of the $k_v$-group $\mathbf{SO}(q)$ is $B_{n,n}$ if and only if \begin{align}\label{formula1}c(q)=(-1,-1)^{\frac{n(n-3)}{2}}(-1, \det q )^n.\end{align}
\end{prop}

\begin{proof}We will show that the following statements are equivalent:
\begin{enumerate}
\item $\textbf{SO}(q)$ is of type $B_{n,n}$.
\item $q\cong \langle 1, -1\rangle^{n-1} \oplus q'$ where $q'$ is an isotropic 3-dimensional form.
\item $q\cong \langle 1, -1\rangle^{n-1} \oplus q'$ where $c(q')=(-1,-\det q')$.
\item $c(q)=(-1,-1)^{\frac{n(n-3)}{2}}(-1, \det q )^n$.
\end{enumerate}

First (1) is equivalent to (2) by the classification of algebraic $k_v$-groups in \cite{T}.  
Next, (2) is equivalent to (3) by Proposition \ref{34iso} (1).  
Lastly (3) is equivalent to (4) by the following computation:
 \begin{align*}
c(q)	&=c\left( \langle 1, -1\rangle^{n-1} \oplus q'\right)\\
		&=c(\langle 1, -1\rangle^{n-1})\  c(q')\  ((-1)^{n-1}, \det q')\\
		&=(-1,-1)^{\frac{(n-1)(n-2)}{2}}\  (-1,-\det q')\  (-1, \det q')^{n-1}\\
		&=(-1,-1)^{\frac{(n-1)(n-2)}{2}+1}\  (-1, \det q' )^{n}\\
		&=(-1,-1)^{\frac{(n^2-3n+2+2)}{2}}\  (-1, (-1)^{n-1}\det q)^{n}\\
		&=(-1,-1)^{\frac{n^2-3n}{2}+2}\  (-1, \det q)^{n}\\
		&=(-1,-1)^{\frac{n(n-3)}{2}}\  (-1, \det q)^{n}.
\end{align*}
\end{proof}

\begin{prop}
Let $k$ be a number field, $v\in V_k$ be a finite place, and $q$ be a $(2n)$-dimensional quadratic form over $k_v$.  
Then the local index of the $k_v$-group $\mathbf{SO}(q)$  is ${}^1D_{n,n-2}$ if and only if \begin{align}\label{formula2}\mathrm{disc}\, q=1\qquad \mbox{and} \qquad   c(q)=-(-1,-1)^{\frac{n(n-1)}{2}}.\end{align}
\end{prop}

\begin{proof}We will show that the following statements are equivalent:
\begin{enumerate}
\item $\mathbf{SO}(q)$ is of type ${}^1D_{n,n-2}$.
\item $q=\langle 1, -1\rangle^{n-2} \oplus q'$ where $q'$ is an anisotropic 4-dimensional form.
\item $q=\langle 1, -1\rangle^{n-2} \oplus q'$ where $\mathrm{disc}\,q'=1$ and $c(q')=-(-1,-1)$  .
\item $\mathrm{disc}\, q=1$ and $c(q)=-(-1,-1)^{\frac{n(n-1)}{2}}$.
\end{enumerate}

First (1) is equivalent to (2) by the classification of algebraic $k$-groups in \cite{T}.  
Next, (2) is equivalent to (3) by Proposition \ref{34iso} (2).  
Lastly (3) is equivalent to (4) by the following computations:
\begin{equation*}
\begin{aligned}
\mathrm{disc}\, q	&=\mathrm{disc}\left(\langle 1, -1\rangle^{n-2} \oplus q'\right)\\
				&=\mathrm{disc}(\langle 1, -1\rangle^{n-2})\  \mathrm{disc}\, q'\\
				&=1.\\ \\ \\ \\ \\ \\ \\
\end{aligned}	
\hspace{6pc}
\begin{aligned}		
c(q)		&=c\left(\langle 1, -1\rangle^{n-2} \oplus q'\right)\\
		&=c(\langle 1, -1\rangle^{n-2})\  c(q')\  ((-1)^{n-2}, \det q')\\
		&=(-1,-1)^{\frac{(n-2)(n-3)}{2}}\  -(-1,-1)\\
		&=-(-1,-1)^{\frac{(n-2)(n-3)}{2}+1}\\
		&=-(-1,-1)^{\frac{(n^2-5n-6+2)}{2}}\\
		&=-(-1,-1)^{\frac{(n^2-n)}{2}-2(n-1)}\\
		&=-(-1,-1)^{\frac{n(n-1)}{2}}.
\end{aligned}
\end{equation*}
\end{proof}

\begin{table}
\begin{center}
{\footnotesize
	\begin{tabular}{|l|l|l|} \hline
ÊÊÊÊÊÊ		\hfill Type \hfill \hfill& \hfillÊ  Classical InvariantsÊ\hfill \hfill & \hfill Tits Index \hfill \hfill	\\ 
		\hline  
		& & \\
	
		$B_{n,n}$ &ÊÊÊÊÊÊÊÊÊ\specialcell{$\dim(q)=2n+1$\\ $\det(q)=$anything\\ $c(q)=(-1,-1)^{\frac{n(n-3)}{2}}(-1, \det(q) )^n$}Ê &  $\xy 
		
		\POS (0,0) *\cir<0pt>{} ="a",
		\POS (20,0) *\cir<0pt>{} ="b",
		\POS (30,0) *\cir<0pt>{} ="c",
		\POS (55,0) *\cir<0pt>{} ="d",
		\POS (70,0) *\cir<0pt>{} ="e",
		\POS (62,0) *+{>} =">",
		
		\POS "a" \ar@{-}^<{} "b",
		\POS "b" \ar@{.}^<{} "c",
		\POS "c" \ar@{-}^<{} "d",
		\POS "d" \ar@{=}^<{} "e",
		
		\POS (0,0) *\xycircle(0,2){-} ="p1",
		\POS (15,0) *\xycircle(0,2){-}="p2",
		\POS (40,0) *\xycircle(0,2){-}="pn-2",
		\POS (55,0) *\xycircle(0,2){-}="pn-1",
		\POS (70,0) *\xycircle(0,2){-}="pn",
		
		\POS (0,0) *\xycircle(3,4){-} ="a1",
		\POS (15,0) *\xycircle(3,4){-} ="a2",
		\POS (40,0) *\xycircle(3,4){-} ="an-2",
		\POS (55,0) *\xycircle(3,4){-} ="an-1",
		\POS (70,0) *\xycircle(3,4){-} ="an",
		
		\endxy$ \\ & & \\
		
		 $B_{n,n-1}$&ÊÊÊÊÊÊÊÊÊ\specialcell{$\dim(q)=2n+1$\\ $\det(q)=$anything\\ $c(q)=-(-1,-1)^{\frac{n(n-3)}{2}}(-1, \det(q) )^n$} & $\xy 
		
		\POS (0,0) *\cir<0pt>{} ="a",
		\POS (20,0) *\cir<0pt>{} ="b",
		\POS (30,0) *\cir<0pt>{} ="c",
		\POS (55,0) *\cir<0pt>{} ="d",
		\POS (70,0) *\cir<0pt>{} ="e",
		\POS (62,0) *+{>} =">",
		
		\POS "a" \ar@{-}^<{} "b",
		\POS "b" \ar@{.}^<{} "c",
		\POS "c" \ar@{-}^<{} "d",
		\POS "d" \ar@{=}^<{} "e",
		
		\POS (0,0) *\xycircle(0,2){-} ="p1",
		\POS (15,0) *\xycircle(0,2){-}="p2",
		\POS (40,0) *\xycircle(0,2){-}="pn-2",
		\POS (55,0) *\xycircle(0,2){-}="pn-1",
		\POS (70,0) *\xycircle(0,2){-}="pn",
		
		\POS (0,0) *\xycircle(3,4){-} ="a1",
		\POS (15,0) *\xycircle(3,4){-} ="a2",
		\POS (40,0) *\xycircle(3,4){-} ="an-2",
		\POS (55,0) *\xycircle(3,4){-} ="an-1",
		
		\endxy$ 
		\\ & & \\
	
		${}^1D_{n,n}^{(1)}$ &ÊÊÊÊÊ\specialcell{$\dim(q)=2n$\\ $\det(q)=(-1)^n$ \quad (i.e. $\mathrm{disc}(q)=1$)\\ $c(q)=(-1,-1)^{\frac{n(n-1)}{2}}$} Ê &  $\xy 
		
		\POS (0,0) *\cir<0pt>{} ="a",
		\POS (20,0) *\cir<0pt>{} ="b",
		\POS (30,0) *\cir<0pt>{} ="c",
		\POS (55,0) *\cir<0pt>{} ="d",
		\POS (70,10) *\cir<0pt>{} ="e",
		\POS (70,-10) *\cir<0pt>{} ="f",
		
		\POS "a" \ar@{-}^<{} "b",
		\POS "b" \ar@{.}^<{} "c",
		\POS "c" \ar@{-}^<{} "d",
		\POS "d" \ar@{-}^<{} "e",
		\POS "d" \ar@{-}^<{} "f",
		
		\POS (0,0) *\xycircle(0,2){-} ="p1",
		\POS (15,0) *\xycircle(0,2){-}="p2",
		\POS (40,0) *\xycircle(0,2){-}="pn-3",
		\POS (55,0) *\xycircle(0,2){-}="pn-2",
		\POS (70,10) *\xycircle(0,2){-}="pn-1",
		\POS (70,-10) *\xycircle(0,2){-}="pn",
		
		\POS (0,0) *\xycircle(3,4){-} ="a1",
		\POS (15,0) *\xycircle(3,4){-} ="a2",
		\POS (40,0) *\xycircle(3,4){-} ="an-3",
		\POS (55,0) *\xycircle(3,4){-} ="an-2",
		\POS (70,10) *\xycircle(3,4){-} ="an-1",
		\POS (70,-10) *\xycircle(3,4){-} ="an",
		
		\endxy$ \\ & & \\
		
		 		${}^1D_{n,n-2}^{(1)}$ &ÊÊÊÊÊ\specialcell{$\dim(q)=2n$\\ $\det(q)=(-1)^n$ \quad (i.e. $\mathrm{disc}(q)=1$)\\ $c(q)=-(-1,-1)^{\frac{n(n-1)}{2}}$} Ê &  $\xy 
		
		\POS (0,0) *\cir<0pt>{} ="a",
		\POS (20,0) *\cir<0pt>{} ="b",
		\POS (30,0) *\cir<0pt>{} ="c",
		\POS (55,0) *\cir<0pt>{} ="d",
		\POS (70,10) *\cir<0pt>{} ="e",
		\POS (70,-10) *\cir<0pt>{} ="f",
		
		\POS "a" \ar@{-}^<{} "b",
		\POS "b" \ar@{.}^<{} "c",
		\POS "c" \ar@{-}^<{} "d",
		\POS "d" \ar@{-}^<{} "e",
		\POS "d" \ar@{-}^<{} "f",
		
		\POS (0,0) *\xycircle(0,2){-} ="p1",
		\POS (15,0) *\xycircle(0,2){-}="p2",
		\POS (40,0) *\xycircle(0,2){-}="pn-3",
		\POS (55,0) *\xycircle(0,2){-}="pn-2",
		\POS (70,10) *\xycircle(0,2){-}="pn-1",
		\POS (70,-10) *\xycircle(0,2){-}="pn",
		
		\POS (0,0) *\xycircle(3,4){-} ="a1",
		\POS (15,0) *\xycircle(3,4){-} ="a2",
		\POS (40,0) *\xycircle(3,4){-} ="an-3",
		\POS (55,0) *\xycircle(3,4){-} ="an-2",
		
		\endxy$ \\ & & \\

		${}^2D_{n,n-1}^{(1)}$ &ÊÊÊÊÊ\specialcell{$\dim(q)=2n$\\ $\det(q)\neq(-1)^n$ \quad (i.e. $\mathrm{disc}(q)\neq1$)\\ $c(q)=$ anything} Ê &  $\xy 
		
		\POS (0,0) *\cir<0pt>{} ="a",
		\POS (20,0) *\cir<0pt>{} ="b",
		\POS (30,0) *\cir<0pt>{} ="c",
		\POS (55,0) *\cir<0pt>{} ="d",
		\POS (70,10) *\cir<0pt>{} ="e",
		\POS (70,-10) *\cir<0pt>{} ="f",
		
		\POS "a" \ar@{-}^<{} "b",
		\POS "b" \ar@{.}^<{} "c",
		\POS "c" \ar@{-}^<{} "d",
		
		\POS (0,0) *\xycircle(0,2){-} ="p1",
		\POS (15,0) *\xycircle(0,2){-}="p2",
		\POS (40,0) *\xycircle(0,2){-}="pn-3",
		\POS (55,0) *\xycircle(0,2){-}="pn-2",
		\POS (70,10) *\xycircle(0,2){-}="pn-1",
		\POS (70,-10) *\xycircle(0,2){-}="pn",

		\POS (0,0) *\xycircle(3,4){-} ="a1",
		\POS (15,0) *\xycircle(3,4){-} ="a2",
		\POS (40,0) *\xycircle(3,4){-} ="an-3",
		\POS (55,0) *\xycircle(3,4){-} ="an-2",
		\POS (70,0) *\xycircle(3,15){-} ="an-1",

		{\POS (70,10);(70,-10) \crv{(70,10)&(61,10)&(52,0)&(61,-10)&(70,-10)}}

		\endxy$ \\ & & \\
		\hline
		\end{tabular}
		\caption{Dictionary between the classical invariants of $q$ and the index of $\mathbf{SO}(q)$.}
		\label{tablequadratic}
}
\end{center}
\end{table}

\begin{prop}\label{simparamgrp}
Let $k$ be a number field, $m$ be odd, and $q$ and $q'$ be $m$-dimensional quadratic form over $k$.  
Then $q$ and $q'$ are isogroupic if and only if they are similar.
\end{prop}

\begin{proof}
In Lemma \ref{isosimrep}, we showed similar forms are isogroupic.  
Now suppose $q'$ represents $\mathbf{G}:=\mathbf{SO}(q)$.   
Let $a\in k^\times/(k^\times)^2$ such that $\det q'=a\det q$.  
We shall show $aq$ and $q'$ are isometric.  
Note that $aq$ also represents $\mathbf{G}$, and since $m$ is odd, $\det (aq)=a \det q=\det q'$.  
We now look at the forms locally.
\begin{enumerate}
\item At each complex place $v\in V_k$, $aq$ and $q'$ have the same dimension, and hence are isometric by Theorem \ref{localuniqueness} (a).
\item At each real place $v\in V_k$, since $m$ is odd,  the index of $\mathbf{G}$ together with the determinant $\det q'$ uniquely determines the signature of $q'\otimes k_v$.  Hence at each real place, $\mathrm{sgn}(q')=\mathrm{sgn}(aq)$.  By Theorem \ref{localuniqueness} (b), $aq$ and $q'$ are isometric at each real place.
\item  At each finite place $v\in V_k$, since $m$ is odd, equation \eqref{formula1} shows that the index of $\mathbf{G}$ together with $\det q'$ uniquely determines $c(q')$. Hence at each finite place, $c(q')=c(aq)$.   By Theorem \ref{localuniqueness} (c), $aq$ and $q'$ are isometric at each finite place.
\end{enumerate}
By Theorem \ref{localglobaluniqueness}, $aq$ and $q'$ are isometric over $k$ and the result follows.
\end{proof}

In \cite[2.6]{GPS}, there is an analogous result for admissible hyperbolic pairs of any dimension.
Their proof heavily uses hyperbolic geometry while our proof is algebraic in nature and applies to all odd dimensional forms.

\begin{prop}\label{titsdeterminesbn}
Let $k$ be a number field, $q_1$ and $q_2$ be $m=2n+1$-dimensional quadratic forms over $k$, and  $\mathbf{G}_i=\mathbf{SO}(q_i)$.  
Then $\mathbf{G}_1$ and $\mathbf{G}_2$ are $k$-isomorphic if and only if the groups $\mathbf{G}_1\otimes k_v$ and $\mathbf{G}_2\otimes k_v$ have the same local index for all $v\in V_k$.  

In particular, 
the $k$-isomorphism class of $\mathbf{G}:=\mathbf{SO}(q)$ is determined by its index at all places.
\end{prop}

\begin{proof}
If $\mathbf{G}_1$ and $\mathbf{G}_2$ are $k$-isomorphic, then $\mathbf{G}_1\otimes k_v$ and $\mathbf{G}_2\otimes k_v$ are $k_v$-isomorphic for all $v\in V_k$, and hence by the Tits Classification Theorem \cite[Theorem 2.7.1]{T}, they have the same index at every place.

We now prove the other direction and suppose that $\mathbf{G}_1\otimes k_v$ and $\mathbf{G}_2\otimes k_v$ have the same index for all $v\in V_k$.  
We may replace $q_2$ with the similar form $\frac{\det q_1}{\det q_2} q_2$, and since $m$ is odd, we may now assume $\det q_1 = \det q_2$.  
As we observed in the proof of the previous proposition, at local places the index  and the determinant determine the isometry class of a representing form.  
Therefore $q_1 \otimes k_v $ and $q_2\otimes k_v$ are isometric for all $v\in V_k$, and hence by Theorem \ref{localglobaluniqueness}, $q_1$ and $q_2$ are isometric.  
The result follows from Lemma \ref{isosimrep}.
\end{proof}

%-----------------------------------------------------------------------------------------------------------
%------------------------FIELDS OF DEFINITION SECTION------------------------------------
%-----------------------------------------------------------------------------------------------------------

\section{Fields of Definition and the Proof of Theorem A}\label{sectionfieldthrma}

Let $\mathbf{G}$ be a semisimple algebraic group over $\C$ and let $\Gamma\subset \mathbf{G}(\C)$ be a Zariski-dense subgroup.  
A \textbf{field of definition}
for $\Gamma$ is a field $F\subset \C$ for which there exits an $F$-form $\mathbf{G}'$ of $\mathbf{G}$ and an isomorphism $\varphi: \mathbf{G}\to \mathbf{G}'$ defined over a finite extension of $F$ such that $\varphi(\Gamma)\subset \mathbf{G}'(F)$ \cite[10.3.10]{MaR2}.  
Vinberg showed \cite{Vin} that for Zariski-dense groups, there is a unique minimal field of definition
$$k_{\mathbf{G}}(\Gamma):=\Q(\mathrm{Tr}(\mathrm{Ad}_{\mathbf{G}}(\gamma))\ | \ \gamma \in \Gamma),$$ where $\mathrm{Ad}_{\mathbf{G}}$ is the adjoint representation of $\mathbf{G}$.
Furthermore this is an invariant of the commensurability class.
In general, the minimal field of definition of a Zariski-dense $\Gamma$ need not coincide with the field that $\mathbf{G}$ is defined over.  
Furthermore, the same abstract group can have different fields of definition depending on the ambient group.
However, \cite[Prop. 2.6]{PR} showed that for an absolutely almost simple group $\mathbf{G}$ over a number field $k$, and $\Gamma\subset \mathbf{G}(k)$ arithmetic and Zariski-dense, the minimal field of definition of $\Gamma$ coincides with the field of definition of the group (i.e. $k_{\mathbf{G}}(\Gamma)=k)$.   
As such, the field of definition of Construction \ref{quadcon} coincides with the minimal field of definition in the sense of Vinberg \cite{Vin} for all quadratic forms $q$ of dimension $m>2$ and $m\ne 4$ (since it is in this case and only in this case that $\mathbf{SO}(q)$ is not absolutely almost simple).

For an arbitrary totally geodesic subspace $N\subset M$, we do not expect to see a relationship between $k(N)$ and $k(M)$ as is demonstrated by the following examples.

\begin{ex}We show three methods of constructing of totally geodesic subspaces $N\subset M$ where both $N$ and $M$ come from quadratic forms and where each realizes a different relationship between  $k(N)$ and $k(M)$.
\begin{enumerate}
\item Subforms produce $N\subset M$ such that $k(N)=k(M)$.

\noindent Let $k$ be an arbitrary number field and let $q$ be a quadratic form over $k$ of dimension $\ge 4$.  Let $r\subset q$ be a subform of dimension $\ge 3$.  Then $\mathbf{SO}(r)$ naturally sits inside $\mathbf{SO}(q)$ as a $k$-subgroup.  Then $k(N)=k=k(M)$.

\vspace{0.5pc}

\item Extension of scalars produce $N\subset M$ such that  $k(N)\subsetneq k(M)$.

\noindent Let $k/\Q$ be a nontrivial finite extension and let $q$ be a quadratic form over $\Q$ of dimension $\ge 3$.  Then $\mathbf{SO}(q)$ naturally sits as a $\Q$-subgroup in the diagonal of $R_{k/\Q}(\mathbf{SO}(q\otimes_{\Q} k))$.  Then $k(N)=\Q\subsetneq k =k(M).$

\vspace{0.5pc}

\item Killing form produces $N\subset M$ such that  $k(N)\supsetneq k(M)$.

\noindent Let $k/\Q$ be a nontrivial finite extension, let $q$ be a quadratic form over $k$ of dimension $\ge 3$,  let $\mathbf{H}=\mathbf{SO}(q)$, and 
%$\mathbf{G}=\mathbf{SO}(\underbrace{q\oplus q\oplus \cdots \oplus q}_{d\ times})$
$\mathbf{G}=\mathbf{SO}(\mathrm{Lie}(R_{k/\Q}(\mathbf{H})),\kappa)$ where $\kappa$ is the Killing form on $\mathrm{Lie}(R_{k/\Q}(\mathbf{H}))$.  
%$\mathbf{G}=\mathbf{PSL}_n$ where $n=d\left(\frac{m(m-1)}{2}\right)$.  
Then, via the adjoint representation,
$$\mathbf{H}(k)= (R_{k/\Q}(\mathbf{H}))(\Q)\subset (\mathbf{SO}(\mathrm{Lie}(R_{k/\Q}(\mathbf{H})),\kappa))^\circ(\Q)\subset\mathbf{G}(\Q).$$
Then $k(N)=k\supsetneq \Q =k(M).$
\end{enumerate}
\end{ex}

Observe that in the above examples, when $k(N)\neq k(M)$, the difference between $\dim N$ and $\dim M$ was quite large.  As the next results show, if the dimensions of $N$ and $M$ are sufficiently close, there is a relationship between their fields of definition.

\begin{lem}\label{fielddefdim}
For $i=1,2$, let $\mathbf{H}_i$ be semisimple $k_i$-groups such that $\mathbf{H}_1$ is absolutely almost simple and  $R_{k_1/\Q}(\mathbf{H}_1)$ is $\Q$-isogenous to $R_{k_2/\Q}(\mathbf{H}_2)$.
\begin{enumerate}
\item Then $k_2$ is isomorphic to a subfield of $k_1$.
\item If $\dim \mathbf{H}_2<2\dim \mathbf{H}_1$, then there is a field isomorphism $\tau:k_1\to k_2$ and $\mathbf{H}_1\times_{\kappa}\mathrm{spec}\, k_2$ and $\mathbf{H}_2$ are $k_2$-isogenous.
\end{enumerate}
\end{lem}

\begin{proof}${}$\\
\textit{(1)}  Replacing $\mathbf{H}_i$ by their adjoint groups, we have $R_{k_1/\Q}(\mathbf{H}_1)$ and $R_{k_2/\Q}(\mathbf{H}_2)$ are $\Q$-isomorphic.
Since  $\mathbf{H}_1$ is absolutely simple, $R_{k_1/\Q}(\mathbf{H}_1)$ is $\Q$-simple and hence 
$R_{k_2/\Q}(\mathbf{H}_2)$ is $\Q$-simple.
It follows that $\mathbf{H}_2$ must be $k_2$-simple, and by Proposition \ref{restrictionofscalars} (1), there exists a field extension $k_1'/k_2$ and absolutely simple $k_1'$-group $\mathbf{H}_1'$ such that $R_{k_1'/k_2}(\mathbf{H}_1')$ and $\mathbf{H}_2$ are $k_2$-isomorphic.  
It follows that $R_{k_1/\Q}(\mathbf{H}_1)$ and $R_{k_1'/k_2}(R_{k_2/\Q}(\mathbf{H}_1'))\cong R_{k_1'/\Q}(\mathbf{H}_1')$ are $\Q$-isomorphic.  
By Proposition \ref{restrictionofscalars} (2), there is a field isomorphism $\tau:k_1\to k_1'$  and $\mathbf{H}_1\times_{\kappa}\mathrm{spec}\, k_2$ and $\mathbf{H}_1'$ are $k_1'$-isomorphic. 

\vspace{0.5pc}

\noindent \textit{(2)}
Our initial assumptions imply that $\mathbf{H}_2$ is $\overline{\Q}$-isomorphic to 
%$\frac{\deg_\Q \psi(k_1)}{\deg_\Q k_2}$ 
$\deg_{k_2}(k_1')$ copies of $\mathbf{H}_1$.  The restriction on dimension implies that $\mathbf{H}_2$ has precisely one such simple factor.  Hence $k_1'=k_2$ and the result follows.
\end{proof}

\begin{prop}\label{fielddefcontain}
Let $\mathbf{H}_1$ be an absolutely almost simple $k_1$-group and $\mathbf{G}$ be absolutely almost simple $k_2$-group, both of which are isotropic at precisely one infinite place, such that $\dim \mathbf{G}<2\dim \mathbf{H}_1.$  Suppose $R_{k_1/\Q}(\mathbf{H}_1)$ is $\Q$-isogenous to a $\Q$-subgroup of $R_{k_2/\Q}(\mathbf{G})$.
Then 
$k_1$ and $k_2$ are isomorphic.
\end{prop}

\begin{proof}
Replace $\mathbf{H}_1$ and $\mathbf{G}$ by their adjoint groups and let $v_1$ (resp. $v_2$) denote the unique infinite place of $k_1$ (resp. $k_2$) where $\mathbf{H}_1$ (resp. $\mathbf{G}$) is isotropic.
Then there is an injective $\Q$-rational map, $$\varphi: R_{k_1/\Q}(\mathbf{H}_1) \to R_{k_2/\Q}(\mathbf{G}),$$
which induces an injective map of  $\R$-simple Lie groups
$$\varphi: \mathbf{H}_1(k_{1,v_1})\to \mathbf{G}(k_{2,v_2}).$$
Let $\mathbf{H}_2$ denote the Zariski-closure of $\varphi(\mathbf{H}_1(k))$ in $\mathbf{G}(k_{2,v_2})$.
Since $\varphi(\mathbf{H}_1(k_1))\subset \mathbf{G}(k_2)$, $\mathbf{H}_2$ is defined over $k_2$.
Observe that $R_{k_1/\Q}(\mathbf{H}_1)$ is $\Q$-isogenous to $R_{k_2/\Q}(\mathbf{H}_2)$  and by our assumption on dimension, $\dim \mathbf{H}_2\le \dim \mathbf{G}<2\dim \mathbf{H}_1$.  Therefore by Lemma \ref{fielddefdim} (2), the result follows.
\end{proof}

%%%%%%%%%%%Observe that the restriction on isotropic places is necessary via the following example.
%%%%%%%%%%%
%%%%%%%%%%%Let $q$ be a form $\langle 1,1,1,1, \sqrt{2}$ over $k:=\Q(\sqrt{2})$ and let $\mathbf{SO}(q)$ is isotropic at  one place.  Let $k'':= \Q(\srqt{\sqrt{2}})$, and $\mathbf{G}:=\mathbf{SO}(q\otimes_{k'}k'')$ is $k''$-simple, and $\mathbf{H}$ sits in the diagonal.
%%%%%%%%%%%

Observe that in the proof of Proposition \ref{fielddefcontain}, we use the fact that our groups are isotropic at precisely one infinite place to ensure that $\mathbf{H}_2$ is defined over $k_2$, instead of a proper subfield, 
and that the dimension of $\mathbf{H}_2$ satisfies the bounds of Lemma \ref{fielddefdim} (2).

\begin{prop}\label{samedim}
Let $M_1$ and $M_2$ be arithmetic locally symmetric spaces coming from quadratic forms  $q_1$ and $q_2$ of dimension $\ge 4$ over number fields $k_1$ and $k_2$ respectively.  
If $\Q TG(M_1)=\Q TG (M_2)$, then $\dim q_1= \dim q_2$.
\end{prop}

%BETTER BUT UNNECESSARY:!!!!!!!!!!
%\begin{prop}\label{samedim}
%Let $M_1$ and $M_2$ be arithmetic locally symmetric spaces coming from quadratic forms  $q_1$ and $q_2$ of dimension $\ge 4$ over number fields $k_1$ and $k_2$ respectively.  Then $\Q TG(M_1)=\Q TG (M_2)$ implies each of the following 
%\begin{enumerate}
%\item  $\dim q_1= \dim q_2$
%\item  The number of complex places of $k_1$ equals the number of complex places of $k_2$.
%\item  The number of real places of $k_1$ where $q_1$ is isotropic equals the number of real places of $k_2$ where $q_2$ is isotropic.
%\item Their is a bijection of real places preserving signature of $q_i$.
%\end{enumerate}
%\end{prop}  

\begin{proof}
We shall prove the contrapositive.  
If $\dim q_1\ne \dim q_2$, then (potentially after relabeling), $\dim q_1> \dim q_2$.  
Let $v_0\in V_{k_1}$ be a real place where $q_1\otimes k_{1,v_0}$ is isotropic.  
By deleting one entry in a diagonal representation of $q_1$ we have a $(\dim q_1 -1)$-dimensional form $r$ that is isotropic at $v_0$.
Let $\mathbf{H}:=\mathbf{SO}(r)$, $H$ denote the noncompact factors of $R_{k_1/\Q}(\mathbf{H})$, and $\mathbf{G}_2:=\mathbf{SO}(q_2)$.
Considering the dimensions of the simple factors of $H$,  it follows that $H$ cannot be $\R$-isogenous to a proper subgroup of $R_{k_2/\Q}(\mathbf{G}_2)$.
Since $\dim r\ge 3$,  $r$  gives rise to a nonflat finite volume totally geodesic subspace $N$ of $M_1$ that cannot be a proper totally geodesic subspace of $M_2$.   
The result then follows.
\end{proof}

\begin{proof}[Proof of Theorem A]
Let $q_1$ and $q_2$ be quadratic forms over $k_1$ and $k_2$, giving rise to $M_1$ and $M_2$, respectively.
Since $\Q TG(M_1)=\Q TG (M_2)$, Proposition \ref{samedim} implies $\dim q_1= \dim q_2=:m$.
Let $r$ be an $(m-1)$-dimensional quadratic $k_1$-subform of $q_1$ that is isotropic at the real place where $q_1$ is isotropic and let $\mathbf{H}_1:=\mathbf{SO}(r)$.  
By Proposition \ref{commimpliesisothm}, $R_{k_1/\Q}(\mathbf{H}_1)$ is $\Q$-isogenous to a $\Q$-subgroup of $R_{k_2/\Q}(\mathbf{SO}(q_2))$.
Observe that 
$$\dim \mathbf{SO}(q_2)=\frac{m(m-1)}{2}=\frac{(m-1)(m-2)}{2}+m<2\left(\frac{(m-1)(m-2)}{2}\right)=2\dim \mathbf{SO}(r).$$
Since $M_1$ and $M_2$ are $\R$-simple, we may apply Proposition \ref{fielddefcontain} and the result follows.
\end{proof}

\begin{rem}
By examining the proof, for $k(M_1)$ and $k(M_2)$ to be isomorphic, it is sufficient that $M_1$ and $M_2$ both contain, up to commensurability, the same  totally geodesic subspace coming from a codimension one quadratic form.   
\end{rem}

%-----------------------------------------------------------------------------------------------------------
%---------------------TECHNICAL CONSTRUCTIONS SECTION----------------------------
%-----------------------------------------------------------------------------------------------------------

\section{Technical Results: Construction of Subforms of Quadratic Forms}\label{sectionconstructions}

This section is dedicated to showing that over number fields, nonisogroupic forms cannot have the same isogroupy classes of subforms.  
Toward these ends, we construct proper quadratic subforms with very specific local properties that exploit the exceptional restrictions on the Hasse invariant in dimensions 1 and 2. 
The results of this section heavily rely upon the following fundamental lemma.

\begin{lem}[Square Existence Lemma]\label{squareexistence}
Let 
\begin{enumerate}
\item $k$ be a number field,
\item $S$ be a finite set of places of $k$, and
\item for each $v\in S$, let $\alpha_v$ be a square class in $k_v^\times$. 
\end{enumerate}
Then there exists an $s\in k^\times$ for which $s\in \alpha_v$ for all $v\in S$. 
\end{lem}

\begin{proof}
Each nontrivial (resp. trivial) square class $\alpha_v$ corresponds to a unique quadratic (resp. trivial ) extension $L_v/k_v$.  By local class field theory, this corresponds to a character $\chi_v$ of $k_v^\times$ of order 2 (resp.~order 1).  
By the Grunwald--Wang Theorem \cite[Chp VIII Thm 2.4]{M}, there exists a character $\chi$ of $GL_1(\mathbb{A}_k)/GL_1(k)$ whose restriction to $k_v^\times$ is $\chi_v$, for all $v\in S$.  
Since $n=2$ and $k[\zeta_2]= k$ is trivially cyclic, we may choose  $\chi$ to have order 2.  
By global class field theory, this gives a quadratic extension $L/k$ where $L=k(s)$.  
Then $s\in \alpha_v$ for all $v\in S$.
\end{proof}

\subsection*{Constructing Nonrepresentable Subforms 1: Nonisogroupic at a real place}\label{sectionconstructionsrep1}
Let $k$ be a number field and let $q$ be an $m$-dimension quadratic form over $k$.  
If $v\in V_k$ is a real place, then we shall say $q$ is \textbf{ordered at $v$} if the signature $(m_{+}^{(v)},m_-^{(v)})$ of $q\otimes k_v$  satisfies $m\ge m_{+}^{(v)}\ge m_-^{(v)}\ge 0$.  We call $q$ \textbf{ordered} if it is ordered at all real places.  
%We now show that every similarity class contains an ordered representative.  

\begin{lem}\label{ordered}
Let $k$ be a number field and $q$ a quadratic form over $k$.  
Then there exists an $a\in k^\times$ so that $aq$ is ordered.
\end{lem}

\begin{proof}
Let $S\subset V_k$ denote the set of all real places and let $S_0\subset S$ denote the set of all real places where $q$ is not ordered.  For each $v\in S$, let $$\alpha_v=\begin{cases}-(k_v^\times)^2& \mbox{if } v\in S_0,\\(k_v^\times)^2& \mbox{if } v\not\in S_0.\\\end{cases}$$  By Lemma \ref{squareexistence}, there exists $a\in k^\times$ such that $a(k_v^\times)^2=\alpha_v$ for all $v\in S$ and hence $aq$ is ordered.  
\end{proof}

Recall that two quadratic forms over $\R$  are $\R$-isogroupic if and only if they are similar.  
%Using this fact we begin constructing subforms of one form that are not similar to a subform of the other.  
%Motivating this first lemma is the observation that the form of signature $r:=\langl(3,1)$ divides  $(4,1)$ but neither $(3,1)$ nor $(1,3)$ divide $(5,0)$.

\begin{lem}\label{realrep1}
Let $q_1$ and $q_2$ be nonisometric $m$-dimensional quadratic forms over $\R$ with signatures $(m_1, n_1)$ and $(m_2, n_2)$ respectively such that $m_1>m_2\ge n_2>n_1$.
Then for all $j\in \Z_{\ge 1}$ such that  $$n_1+n_2<j<m$$ there exists an isotropic $j$-dimensional form dividing $q_2$ that is not similar to a form dividing $q_1$.  Furthermore, this form can be realized by deleting $m-j$ entries in a diagonal representation of $q_2$.
\end{lem}

\begin{proof}
The idea of the proof is that we pick a subform $r$ of $q_2$ such that neither $r$ nor $-r$ divides $q_1$.   
%By proposition 3.1, up to isometry we may represent $q_1$ and $q_2$ with diagonal matrices $I_{m_1, n_1}$ and $I_{m_2, n_2}$ respectively.  
%$$I_{m_+, m_-}=\mathrm{diag}(\underbrace{1,1,\ldots, 1}_{m_+}, \underbrace{-1, -1, \ldots -1}_{m_-}).$$ 
We may represent  
$$q_1=\langle \underbrace{a_1,\ldots , a_{m_1}}_{>0}, \underbrace{a_{m_1+1},\ldots, a_m}_{<0}\rangle \quad \mbox{and} \quad q_2=\langle \underbrace{b_1,\ldots ,  b_{m_2}}_{>0}, \underbrace{b_{m_2+1},\ldots, b_m}_{<0}\rangle,$$with  $a_i, b_j\in \R$.
%For the moment suppose $m_1<m_2\le n_2<n_1$.  Since $q_1$ and $q_2$ are not isometric, $m_2>m_1$.  
The desired subform may be obtained by deleting the first $m-j$ entries of $q_2$, namely  let
$$r:=\langle b_{m-j+1}, b_{m-j+2}, \ldots b_{m-1}, b_{m}\rangle.$$
By construction, $r$ has signature $(j-n_2, n_2)$ from which we can see that $r$ is always isotropic and both
\begin{itemize}
\item $j-n_2>n_1+n_2-n_2= n_1$, and 
\item $n_2>n_1.$
\end{itemize}
Hence neither $r$ nor $-r$ is a subform of $q_1$.
\end{proof}

\begin{rem}
The more isotropic both forms are, the fewer subforms arise from this construction.  In particular, there are no subforms precisely when $m$ is even and the two forms have signatures $$\left(\frac{m}{2}-1, \frac{m}{2}+1\right)  \quad \mbox{and} \quad  \left(\frac{m}{2}, \frac{m}{2}\right).$$
\end{rem}
%
%In the end, our goal is to construct locally symmetric spaces of noncompact type, and hence want isotropic subforms.  Hence Lemma \ref{realrep1} largely succeeds, but we need to address the case in the above remark.

\begin{lem}\label{realrep2}
Let $q_1$ and $q_2$ be nonisometric $m$-dimensional quadratic forms over $\R$ with signatures $(m_1, n_1)$ and $(m_2, n_2)$ respectively such that $m_1>m_2\ge n_2>n_1>0$.
Then for all $j\in \Z_{\ge 1}$ such that  $$m_1<j<m$$ there exists an isotropic $j$-dimensional form dividing $q_1$ that is not similar to a form dividing $q_2$.  
%If $2\le m_1$ and $j>m_1$, then the form can be chosen to be isotropic.  
Furthermore, this form can be realized by deleting $m-j$ entries in a diagonal representation of  $q_1$.
\end{lem}

\begin{proof}
%By proposition 3.1, up to isometry we may represent $q_1$ and $q_2$ with diagonal matrices $I_{m_1, n_1}$ and $I_{m_2, n_2}$ respectively.  
%$$I_{m_+, m_-}=\mathrm{diag}(\underbrace{1,1,\ldots, 1}_{m_+}, \underbrace{-1, -1, \ldots -1}_{m_-}).$$ 
Again we may represent  
$$q_1=\langle \underbrace{a_1,\ldots , a_{m_1}}_{>0}, \underbrace{a_{m_1+1},\ldots, a_m}_{<0}\rangle \quad \mbox{and} \quad q_2=\langle \underbrace{b_1,\ldots ,  b_{m_2}}_{>0}, \underbrace{b_{m_2+1},\ldots, b_m}_{<0}\rangle,$$with  $a_i, b_j\in \R$.
This time the desired subform may be obtained by deleting the last $m-j$ entries of $q_1$, namely  let
$$r:=\langle a_{1}, a_2, \ldots ,  a_{j} \rangle.$$
By construction, $r$ has signature $(m_1, n_1-m+j)$ from which we can see that $r$ is always isotropic and by our initial assumptions, both  $m_1>m_2$ and $m_1> n_2$.
%\begin{itemize}
%\item $m_1>m_2.$
%\item $m_1> n_2$, and 
%\end{itemize}
Hence neither $r$ nor $-r$ is a subform of $q_2$.
\end{proof}

\begin{rem}
The more anisotropic $q_1$ is, the fewer subforms arise from this construction.  In particular, there are no subforms arising from this construction precisely when $m_1= m-1$.  
%$n_2=m-3$, and hence when $m-(m-3)<2$ which occurs when $m\le 4$.
\end{rem}

Combining Lemma \ref{realrep1} and Lemma \ref{realrep2} we obtain the following corollary.

\begin{cor}\label{realrepcor}
Let $q_1$ and $q_2$ be nonisogroupic quadratic forms over $\R$ of dimension $m\ge 5$.  
Then there exists an isotropic $(m-1)$-dimensional subform of one which is not similar to a subform of the other. Furthermore, this form can be realized by deleting one entry in a diagonal representation of either $q_1$ or $q_2$.
\end{cor}

The bound $m\ge 5$ is tight since neither Lemma \ref{realrep1} nor Lemma \ref{realrep2} may be applied to the nonisometric 4-dimensional real forms $q_1$ and $q_2$ with signatures $(3,1)$ and $(2,2)$ respectively.  
It is not hard to see that every isotropic $\R$-subform of one is isogroupic to an $\R$-subform of the other.  
%This failure is related to subtleties related to comparing the groups $\mathbf{SO}(3,1)$ and $\mathbf{SO}(2,2)$.

\begin{prop}\label{realrep}
Let $k$ be a totally real number field and let $m\ge 5$.
Suppose that  $q_1$ and $q_2$ are ordered $m$-dimensional quadratic forms over $k$ that are isotropic at precisely one real place, $v_1$ and $v_2$, respectively, and $q_{1,v_1}$ and $q_{2,v_2}$ are not $\R$-isometric.
%%
%%for each quadratic form $q_2^{(j)}$, $j\in \{1,2,\ldots, \ell\}$, in the $\mathrm{Aut}(k/\Q)$-orbit of $q_2$, there is a real place $v_j$ over which $q_1$ and $q_2^{(j)}$ are not isogroupic.
%%\begin{enumerate}
%%%\item $k$ be a number field,
%%%\item $m\ge 5$, and
%%\item  at each real place $v$,  $q_{1,v}$ and  $q_{2,v}$ are ordered, and
%%%\item $q_1$ and $q_2$ be ordered $m$-dimensional quadratic $k$-forms such that there is a real place $v_0\in V_k$ over which $q_1$ and $q_2$ are not isogroupic.
%%\item no form in the $\mathrm{Aut}(k/\Q)$-orbit of $q_1$ is isometric to a form, denoted $\{q_2^{(j)}\}$, $j\in \{1,2,\ldots, \ell\}$ is isometric to $q_1$ over the real places of $k$.
%%%there is a real place $v_0\in V_k$ over which $q_{1,v_0}$ and $q_{2,v_0}$ are not isogroupic.
%%%\item  $q_{1,v_0}$ and $q_{2,v_0}$ have signatures $(m_1, n_1)$ and $(m_2, n_2)$ respectively such that $m_1<m_2\le n_2<n_1$.
%%\end{enumerate}
%%Then there exists an $(m-1)$-dimensional quadratic $k$-form $r$ that is a subform of $q_1$ and not isogroupic to a subform of $q_2^{(j)}$ for all $j\in \{1,2,\ldots, \ell\}$.   
%%Furthermore if $q_1$ is isotropic at a real place, then $r$ can be chosen to be isotropic at that real place as well.
Then, up to relabelling, there exists an $(m-1)$-dimensional quadratic $k$-subform $r$ of $q_1$ that is not isogroupic to a subform of any twist of $q_2$.   
Furthermore $r$ can be chosen to be isotropic at the real places where $q_1$ is isotropic.
%divides one form, but for which no subform dividing the other form represents $\mathbf{SO}(r)$.
\end{prop}

\begin{proof}
We may replace $q_2$ with a twist such that $q_1$ and $q_2$ are isotropic at the same real place.  The result then follows from Corollary \ref{realrepcor}.
\end{proof}

\subsection*{Constructing Nonrepresentable Subforms 2: Isogroupic at all real places}\label{sectionconstructionsrep2}
In this section, we look at quadratic forms over $k$ that are isogroupic at all infinite places but are not isogroupic at a finite place.  

\begin{thm}\label{finiterepodd}
Let $k$ be a number field and let $m=2n+1$ for $n\ge2$.  
Suppose that $q_1$ and $q_2$ are nonisometric $m$-dimensional quadratic forms over $k$ such that
\begin{enumerate}
%\item $k$ be a number field,
%\item $m=2n+1$ for $n\ge2$,
\item  at each infinite place $v$,  $q_{1,v}$ and  $q_{2,v}$ are ordered and isometric, and
%$q_{1,v}\cong q_{2,v}$ for each infinite place $v\in V_k$,  and
\item there is a finite place $v_0\in V_k$ where:
	\begin{enumerate}
		\item $\det_{v_0} q_1 = 1= \det_{v_0} q_2$,
		\item $c_{v_0}(q_1) \ne c_{v_0}(q_2)$.
	\end{enumerate}
\end{enumerate}
Then there exists an $(m-1)$-dimensional quadratic subform $r$ of $q_1$ that is not isogroupic to a subform of $q_2$.   
Furthermore if $q_1$ is isotropic at a real place, then $r$ can be chosen to be isotropic at that real place as well.
\end{thm}

\begin{proof}
We construct $r$ locally and then patch it together into a global form using the local-to-global results of Section \ref{sectionquadforms}.
Let $S=\{v_0\} \cup \{\mbox{infinite real places of $k$}\}$
and for each $v\in S$, we pick square classes $\alpha_v\in k_v^\times/(k_v^\times)^2$ as follows:
\begin{itemize}
\item For $v_0$, let $\alpha_{v_0}$ be such that $\alpha_{v_0}=(-1)^n$.
\item For each infinite $v\in S$ let $\alpha_v = \det q_1(k_v^\times)^2 (= \det q_2(k_v^\times)^2)$.
\end{itemize}
By Lemma \ref{squareexistence}, we may choose an $s\in k^\times$ such that $s\in \alpha_v$ for all $v\in S$.

For each finite place $v\in V_k$, define $t_{v}$ and $r_{v}$ to be the quadratic $k_v$-forms with invariants: 
\begin{align*}
&\dim t_{v}=1 && \dim r_{v}=m-1\\
&\det t_{v}=\frac{\det q_1}{s} && \det r_{v}=s\\
&c_v(t_{v})=1 && c_{v} (r_{v})=c_{v} (q_1)\ \left(s,\frac{\det q_{1}}{s}\right)_{v} .
\end{align*}
Such forms exist by Theorem \ref{localexistence} (3).  

For each infinite place $v\in V_k$, define the form $t_{v}$ by:
\begin{align*}
&t_{v}=\left \langle \frac{\det q_1}{s}\right\rangle.
\end{align*}
At each complex place, $t_{v}$ divides $q_1\otimes k_v$.  
At each real place, $q_1$ and $q_2$ are ordered, and hence $t_{v}=\langle 1\rangle$ is a subform of $q_1\otimes k_v$.  
Therefore at each infinite place it makes sense to take the complement of $t_{v}$ in $q_{1}\otimes k_v$ and we may define the form $r_{v}$ by
\begin{align*}
 r_{v}=t_{v}^\perp.
\end{align*}
At each complex place $v\in V_k$, we trivially have $c_v(r_{v})=1=c_v(q_1).$  
For each real $v\in V_k$, let  $(m_+^{(v)}, m_-^{(v)})$ denote the signature of $q_1\otimes k_v$.  
Observe that $r_{v}$ has signature $(m_+^{(v)}-1, m_-^{(v)})$, and hence is isotropic whenever $q_1\otimes k_v$ is isotropic.  
Also note that at an isotropic real place, $$c_v(r_{v})=(-1)^{\frac{m_-^{(v)}\left(m_-^{(v)}-1\right)}{2}}=c_v(q_1).$$

We now show that for each place $v \in V_k$, $t_{v}\oplus r_{v}\cong q_{1} \otimes k_v$.  
This is true by construction at the infinite places.  
When $v$ is finite, we have
$$\dim(t_{v}\oplus r_{v})=1+(n-1)=n=\dim( q_{1} \otimes k_v),$$
$$\det(t_{v}\oplus r_{v})= (\det q_1 / s)s=\det( q_{1} \otimes k_v),$$
and by the product formula for the Hasse invariant
$$c(t_{v}\oplus r_{v})=c(t_{v})c(r_{v})\left(\frac{\det q_1}{s}, s\right)= c_v(q_1)\left(\frac{\det q_1}{s}, s\right)^2=c( q_{1} \otimes k_v).$$
By Theorem \ref{localuniqueness} (3), they are isometric. 

To build a global form, we must check that our forms satisfy the compatibility criteria of Theorem  \ref{localglobal}.  Observe that $c_v(t_{v})=1$ for all $v\in V_k$ and hence $\prod_{v\in V_k} c_v(t_{v})=1$.  
By our choice of $s$, $\left(s,\frac{\det q_{1}}{s}\right)_{v}=1$ at each infinite place, and hence
\begin{align}\label{compatibilityodd}
\prod_{v\in V_k} c_v(r_{v})&=\left(\prod_{v\in V_k\ \mathrm{ finite}} c_v(r_{v})\right)\times \left(\prod_{v\in V_k\ \mathrm{ real}} c_v(r_{v})\right) \times \left(\prod_{v\in V_k\ \mathrm{ complex}} c_v(r_{v})\right)\notag\\
&=\left(\prod_{v\in V_k\ \mathrm{ finite}}c_{v} (q_1)\ \left(s,\frac{\det q_{1}}{s}\right)_{v}\right)\times \left(\prod_{v\in V_k\ \mathrm{ real}} c_v(q_1)\right) \times \left(\prod_{v\in V_k\ \mathrm{ complex}} c_v(q_{1})\right)\notag\\
&=\prod_{v\in V_k} c_v(q_{1}) \times \prod_{v\in V_k} \left(s,\frac{\det q_{1}}{s}\right)_{v}\notag\\
&=1.
\end{align}
The final product is trivial because both the Hasse invariant and the Hilbert symbol of global objects satisfy the product formula.

By Theorem \ref{localglobal}, there exist quadratic forms $t$ and $r$ over $k$ such that for all $v\in V_k$, $t\otimes k_v\cong t_{v}$ and $r\otimes k_v\cong r_{v}$.  
Furthermore, for each $v \in V_k$, we have shown that $t_{v}\oplus r_{v}\cong q_{1} \otimes k_v$ so by Theorem \ref{localglobaluniqueness} we conclude $t\oplus r \cong q_1$, and hence $r$ is a subform of $q_1$.

Suppose that $r'$ is isogroupic to $r$.
%Let $\mathbf{H}_1=\mathbf{SO}(r_1)$.  
%% are no representatives $r^{(j)}$ of $\mathbf{H}_1$ such that $r_2^{(j)}\subset q_2^{(j)}$.  
%Suppose that $r'$ is a representative of $\mathbf{H}_1$.  
Since $\mathbf{H}:=\mathbf{SO}(r)$ is a group of type $D_n$ over $k$, it determines the following invariants of $r'$:
\begin{enumerate}
\item $\dim (r')=2n = \dim(r)$.
\item $\mathrm{disc}_{v}( r')=1$ at precisely the places $v\in V_k$ where $\mathbf{H}\otimes k_v$ is a group of inner type (i.e., the $*$-action is trivial).  This means that  $\mathrm{disc}_{v}( r')=1$ if and only if  $\mathrm{disc}_{v}( r)=1$, or in other words, at the places where $\mathrm{disc}_{v}( r)=1$, then $\det r'=\det r$.
\item $c_v(r')=c_v(r)$ at each place $v$ where  $\mathrm{disc}_v(r)=1$  (see equation \eqref{formula2}).
\end{enumerate}
%Now let $r'_i$ be any quadratic form satisfying these three.  

We now show that 
%$\mathbf{H}_1$ is not isomorphic to a $k$-subgroup of $\mathbf{G}_2^{(j)}:=\mathbf{SO}(q_2^{(j)})$ 
%for all $j\in \{1,2,\ldots,\ell\}$, which reduces to showing that 
no subform of $q_2$ is isogroupic to $r$.
Suppose that $r'$ is isogroupic to $r$ and there exists some form $t'$ such that $r'\oplus t'\cong q_2$.
It immediately follows that $\dim t'=1$, $\det t'= \frac{\det q_2}{\det r'}$, and, by the exceptional restriction, $c(t')=1$.  
%Letting a subscript $0$ denote localization at $v_0$, 
Our choice of $s$ implies:
$$\mathrm{disc}\, r_{v_0}=(-1)^n\det r_{v_0}= (-1)^{2n}=1,$$
hence $\det r_{v_0}= \det r'_{v_0}$ and $c(r_{v_0})= c (r'_{v_0})$.
Therefore:
\begin{align*}
c(q_{2,v_0})	&=c(r_{v_0}'\oplus t'_{v_0}) \\
			&= c(r'_{v_0})c(t'_{v_0})\left(\det r'_{v_0}, \frac{\det q_{2,v_0}}{\det r'_{v_0}}\right)\\
			&= c(r_{v_0})\left(\det r_{1,v_0}, \frac{\det q_{2,v_0}}{\det r_{v_0}}\right)\\
			&=\left(c(q_{1,v_0})\left(\det r_{v_0}, \frac{\det q_{1,v_0}}{\det r_{v_0}}\right)\right)\left(\det r_{v_0}, \frac{\det q_{2,v_0}}{\det r_{v_0}}\right)\\
			&=c(q_{1,v_0})\left(\det r_{1,v_0}, \frac{\det q_{1,v_0} \det q_{2,v_0}}{(\det r_{1,v_0})^2}\right)\\
			&=c(q_{1,v_0})\left(\det r_{1,v_0}, 1\right)\\
			&=c(q_{1,v_0}).
\end{align*}
This contradicts our initial assumption that $c(q_{1,(j)})\ne c(q_{2,v_0})$ and the result follows.
\end{proof}

\begin{ex}\label{hyerbolic5}
Consider the following 5-dimensional quadratic forms over $\Q$:
$$q_1=\langle  1,1,1, 1, -5\rangle\qquad \mbox{and}  \qquad q_2=\langle 1, 1, 3, 3, -5\rangle. $$
Observe that $\det q_1=-5=\det q_2$,  which in $\Q_3$ is a square.  Furthermore, a quick computation shows $c_3(q_1)=1$ and $c_3(q_2)=-1$.  
By Theorem \ref{finiterepodd}, there exists a 4-dimensional quadratic form $r\subset q_1$ so that  $\mathbf{H}:=\mathbf{SO}(r)\subset \mathbf{SO}(q_1)$ but $\mathbf{H}$ is not $\Q$-isomorphic to a subgroup of $\mathbf{SO}(q_2)$.
It is not hard to check that $r=\langle1,1, 1, -5\rangle$ is such a form.
\end{ex}

%%%%%%%%%%%%%%%%%%%%%%%%

\begin{lem}\label{lem:evensplitsubforms}
Let $k_v$ be a nonarchimedian local field, $q$ be a $2n$-dimensional quadratic form over $k_v$, and $r$ be a codimension one subform of $q$.  
\begin{enumerate}
\item If $\mathbf{SO}(q)$ is of type ${}^1D_{n,n}^{(1)}$, then $\mathbf{SO}(r)$ is of  type $B_{n-1,n-1}$.  
\item If $\mathbf{SO}(q)$ is of type ${}^1D_{n,n-2}^{(1)}$, then $\mathbf{SO}(r)$ is of  type $B_{n-1,n-2}$.  
\end{enumerate}
\end{lem}

\begin{proof}
We show $(i)$.   
Let $s=\det r$.
A direct computation yields
\begin{align*}
c(r)			&=c(q)\left(s,\frac{\det q}{s}\right)  \\
			&= (-1,-1)^{\frac{n(n-1)}{2}}(s, -\det q)\\
			&= (-1,-1)^{\frac{(n-4)(n-1)}{2}}(s, -(-1)^n)\\
			&= (-1,-1)^{\frac{(n-1)(n-4)}{2}}(-1,s)^{n-1}.
%			&= \begin{cases}(-1,-1)_{v_0}^{\frac{(n-2)(n-3)}{2}+1}&even\\
%						(-1,-1)_{v_0}^{\frac{(n-2)(n-3)}{2}+1}(-1, -1)_{v_0}^{n+1}&odd\end{cases}
\end{align*}
The result follows from equation \eqref{formula1}.  
By inserting negative signs, $(ii)$ follows analogously.
\end{proof}

\begin{thm}\label{finiterepeven1}
Let $k$ be a number field and let $m=2n\ge 4$.
Suppose that  $q_1$ and $q_2$ are nonisometric $m$-dimensional quadratic forms over $k$ such that
\begin{enumerate}
%\item $k$ be a number field,
%\item $m=2n$ for $n\ge 2$,
\item at each infinite place $v$,  $q_{1,v}$ and  $q_{2,v}$ are ordered and isometric, and
%	\begin{enumerate}
%		\item $\det q_1 = \det q_2$ (and hence $\mathrm{disc}(q_1) = \mathrm{disc}(q_2)$),
%		\item $q_{1,v}\cong q_{2,v}$ at each infinite place $v$,
%	\end{enumerate}		
\item there is a finite place $v_0\in V_k$ where 
	\begin{enumerate}
		\item $\mathrm{disc}_{v_0} (q_1)=1= \mathrm{disc}_{v_0} (q_{2})$, and
		\item $c_{v_0}(q_1)=(-1,-1)_{v_0}^{\frac{n(n-1)}{2}} \ne -(-1,-1)_{v_0}^{\frac{n(n-1)}{2}}=c_{v_0}(q_{2})$.
	\end{enumerate}
%one of following cases hold:
%\begin{description}
%	\item[Case One]${}$
%	\begin{enumerate}
%		\item $\mathrm{disc}_{v_0} (q_1)=1= \mathrm{disc}_{v_0} (q_{2})$, and
%		\item $c_{v_0}(q_1)=(-1,-1)_{v_0}^{\frac{n(n-1)}{2}} \ne -(-1,-1)_{v_0}^{\frac{n(n-1)}{2}}=c_{v_0}(q_{2})$.
%	\end{enumerate}
%	\item[Case Two]${}$
%	\begin{enumerate}
%		\item $n$ is even
%		\item $\mathrm{disc}_{v_0} (q_1)=-1$
%		\item $\mathrm{disc}_{v_0} (q_2)=1$, and
%		\item $c_{v_0}(q_1)=-(-1,-1)_{v_0}^{\frac{n(n-1)}{2}}$.
%	\end{enumerate}
%\end{description}
\end{enumerate}
Then there exists an $(m-1)$-dimensional quadratic  subform $r$ of $q_1$  
that is not isogroupic to a subform of $q_2$.
%form $r$ dividing $q_1$ such that no subform of $q_2$ represents $\mathbf{SO}(r)$.   
Furthermore if the $q_1$ is isotropic at a real place, then $r$ can be chosen to be isotropic at that real place as well.
\end{thm}

\begin{proof}
%By assumption, 
%$$c_v(q_1)=(-1,-1)_v^{\frac{(n-2)(n-3)}{2}+1} \ne -(-1,-1)_v^{\frac{(n-2)(n-3)}{2}+1}=c_v(q_2)$$
Again we are going to construct the desired forms locally and then use the existence, uniqueness, and local-to-global results of Section \ref{sectionquadforms} to create the desired global forms.
Let $S=\{v_0\}\cup\{\mbox{infinite real places of $k$}\}$ and for each $v\in S$, let $\alpha_v = \det q_1(k_v^\times)^2$.
%we pick the trivial square class $\alpha_v\in k_v^\times/(k_v^\times)^2$.
By Lemma \ref{squareexistence}, we choose an $s\in k^\times$ for which $s\in \alpha_v$ for all $v\in S$.

For each finite place $v\in V_k$, define $t_{v}, r_{v}$ to be the quadratic $k_v$-forms with invariants given by:
\begin{align*}
&\dim t_{v}=1 && \dim r_{v}=m-1\\
&\det t_{v}=\frac{\det q_1}{s} && \det r_{v}=s\\
&c_v(t_{v})=1 && c_{v} (r_{v})=c_{v} (q_1)\ \left(s,\frac{\det q_{1}}{s}\right)_{v}.
\end{align*}
Such forms exist by Theorem \ref{localexistence} (3).
For each infinite place $v\in V_k$, define $t_v$ by:
\begin{align*}
&t_v=\left \langle \frac{\det q_1}{s}\right\rangle.
\end{align*}
At each complex place $t_v$ divides $q_1\otimes k_v$.  
By assumption, $q_1$ is ordered at each real place $v\in V_k$, and hence $t_v=\langle 1\rangle$ is a subform of $q_1\otimes k_v$.  
Therefore at each infinite place it makes sense to take the complement of $t_v$ in $q_{1}\otimes k_v$ and we may define forms $r_v$ by
\begin{align*}
 r_v=t_v^\perp.
\end{align*}
At each complex place $v\in V_k$, we trivially have $c_v(r_v)=1=c_v(q_1).$  
For each real $v\in V_k$, let  $(m_+^{(v)}, m_-^{(v)})$ denote the signature of $q_1\otimes k_v$.  
Observe that $r_v$ has signature $(m_+^{(v)}-1, m_-^{(v)})$, and hence is isotropic whenever $q_1\otimes k_v$ is isotropic.  
At such an isotropic real place, $c_v(r_v)=c_v(q_1).$

Just as in the proof of Theorem \ref{finiterepodd},  we have:
\begin{itemize}
\item  The families $\{t_v\}_{v\in V_k}$ and $\{r_v\}_{v\in V_k}$ satisfy the global compatibility conditions (see \ref{compatibilityodd}), and hence by Theorem \ref{localglobal}, there exist quadratic forms $t$ and $r$ over $k$ such that for all $v\in V_k$, $t\otimes k_v\cong t_v$ and $r\otimes k_v\cong r_v$.
\item By Theorem \ref{localuniqueness} (3),  $t_v\oplus r_v$ and $q_1\otimes k_v$ are isometric at each place $v\in V_k$.
\item By Theorem \ref{localglobaluniqueness} we conclude $t\oplus r \cong q_1$, and hence $r$ is a subform of $q_1$.
\end{itemize}

Our assumptions on the Hasse invariant at $v_0$ and equation \eqref{formula2} imply   that $\mathbf{SO}(q_{1,v_0})$ is of type ${}^1D_{n,n}^{(1)}$ while $\mathbf{SO}(q_{2,v_0})$ is of type ${}^1D_{n,n-2}^{(1)}$.
By Lemma \ref{lem:evensplitsubforms}(i), $\mathbf{SO}(r)$ is the split group $B_{n-1,n-1}$ at $v_0$, 
but by Lemma \ref{lem:evensplitsubforms}(ii), every codimension one subbform of $q_2$ yields the nonsplit group of type $B_{n-1,n-2}$.
Hence no subform of $q_2$ is isogroupic to $r$.
%
% of  coming from 
%We claim that $\mathbf{H}:=\mathbf{SO}(r)$ is the split group $B_{n-1,n-1}$ at $v_0$.  
%By equation \eqref{formula1}, this means we must show
%$$c_{v_0}(r)=(-1,-1)_{v_0}^{\frac{(n-1)(n-4)}{2}}(-1,s)_{v_0}^{n-1}=(-1,-1)_{v_0}^{\frac{n(n-1)}{2}}(-1,s)_{v_0}^{n-1}$$
%A direct computation yields
%
%\begin{align*}
%c_{v_0}(r)	&=c_{v_0}(q_1)\left(s,\frac{\det q_{1}}{s}\right)_{v_0}  \\
%			&= (-1,-1)_{v_0}^{\frac{(n-2)(n-3)}{2}+1}(s, -\det q_1)_{v_0}\\
%			&= (-1,-1)_{v_0}^{\frac{(n-2)(n-3)}{2}+1}(s, -(-1)^n)_{v_0}\\
%			&= (-1,-1)_{v_0}^{\frac{(n-2)(n-3)}{2}+1}(s, -1)_{v_0}^{n-1}\\
%			&= (-1,-1)_{v_0}^{\frac{n^2-5n+6+2}{2}}(s, -1)_{v_0}^{n-1}\\
%			&= (-1,-1)_{v_0}^{\frac{n(n-1)}{2}}(s, -1)_{v_0}^{n-1}.
%%			&= \begin{cases}(-1,-1)_{v_0}^{\frac{(n-2)(n-3)}{2}+1}&even\\
%%						(-1,-1)_{v_0}^{\frac{(n-2)(n-3)}{2}+1}(-1, -1)_{v_0}^{n+1}&odd\end{cases}
%\end{align*}
%
%Since $\mathbf{H}$ is split at $k_{v_0}$, we have 
%$$\mathrm{rank}_{k_{v_0}}(\mathbf{H})=\frac{(m-1)-1}{2}=\frac{(2n-1)-1}{2}=n-1>n-2=\mathrm{rank}_{k_{v_0}}(\mathbf{G_2}).$$  
%By rank considerations, $\mathbf{H}\otimes k_{v_0}$ cannot be a subgroup of $\mathbf{G}_2\otimes k_{v_0}$, and hence $\mathbf{H}$ cannot be a subgroup of $\mathbf{G}_2$.  
\end{proof}

%An interesting consequence of the proof is the following result.
%
%\begin{cor}
%Over a local field, the split group of type ${}^1D_{n,n}^{(1)}$ cannot contain a subgroup of type $B_{n-1,n-2}$.
%\end{cor}
%%Say more.... this is cool.

\begin{ex}\label{hyerbolic4}
Consider the following 4-dimensional quadratic forms over $\Q$:
$$q_1=\langle 1,1, 5, -1\rangle\qquad \mbox{and}  \qquad q_2=\langle  3, 3, 5, -1\rangle. $$
Observe that $\det q_1=-5=\det q_2$,  which in $\Q_3$ is a square.  
Hence these have discriminant 1 in $\Q_3$.  
Furthermore, a quick computation shows $c_3(q_1)=1$ and $c_3(q_2)=-1$.  
By Theorem \ref{finiterepeven1}, there exists a 3-dimensional quadratic form $r\subset q_1$ so that  $\mathbf{H}:=\mathbf{SO}(r)\subset \mathbf{SO}(q_1)$ but $\mathbf{H}$ is not $\Q$-isomorphic to a subgroup of $\mathbf{SO}(q_2)$.
It is not hard to check that $r=\langle1,1, -1\rangle$ is such a form.
\end{ex}

\begin{thm}\label{finiterepeven2} 
%Let 
Let $k$ be a number field and let $m=2n\ge 6$.
Suppose that  $q_1$ and $q_2$ are nonisometric $m$-dimensional quadratic forms over $k$ such that
\begin{enumerate}
%\item $k$ be a number field,
%\item $m=2n$ for $n\ge 3$,
\item at each infinite place $v$,  $q_{1,v}$ and  $q_{2,v}$ are ordered and isometric, and
%\item $q_1$ and $q_2$ be nonisometric ordered $m$-dimensional quadratic forms over $k$ such that
%	\begin{enumerate}
%		\item $\det q_1 \neq \det q_2$ (and hence $\mathrm{disc}(q_1) \neq \mathrm{disc}(q_2)$),
%		\item $q_{1,v}\cong q_{2,v}$ at each infinite place $v$,
%	\end{enumerate}		
\item there is a finite place $v_0\in V_k$ where:
	\begin{enumerate}
		\item $\mathrm{disc}_{v_0} q_1=1$,
		\item $\mathrm{disc}_{v_0} q_2\neq 1$,
		\item $c_{v_0} (q_1)\ne c_{v_0}(q_2)(-1, \mathrm{disc}(q_2))_{v_0}^{\frac{m-2}{2}}$
%		\item $-1$ is a square at $v_0$.
	\end{enumerate}
\end{enumerate}
Then there exists an $(m-2)$-dimensional quadratic subform $r$ of $q_2$  
that is not isogroupic to a subform of $q_1$.
%such that no subform of $q_1$ represents $\mathbf{SO}(r)$.  
Furthermore if $q_2$ is isotropic at a real place, then $r$ can be chosen to be isotropic at that real place as well.
\end{thm}

\begin{proof} 
As we did in Theorems \ref{finiterepodd} and \ref{finiterepeven1} we construct the desired forms locally and use the results of Section \ref{sectionquadforms} to create global forms.  
Let $S=\{v_0\} \cup \{\mbox{infinite real places of $k$}\}$ and for  each $v\in S$, we pick square classes $\alpha_v\in k_v^\times/(k_v^\times)^2$ as follows:
\begin{itemize}
\item For $v_0$, let $\alpha_{v_0}$ be such that $\alpha_{v_0}=(-1)^{\frac{m-2}{2}}(k_{v_0}^\times)^2$.
\item For  infinite $v\in S$ let $\alpha_v= \det q_2(k_{v}^\times)^2$.
\end{itemize}
By Lemma \ref{squareexistence} we choose an $s\in k^\times$ for which $s\in \alpha_v$ for all $v\in S$.
 
For each finite place $v\in V_k$, define $t_{v}, r_{v}$ to be the quadratic $k_v$-forms with invariants given by:
\begin{align*}
&\dim t_{v}=2 && \dim r_{v}=m-2\\
&\det t_{v}=\frac{\det q_2}{s} && \det r_{v}=s\\
&c_v(t_{v})=1 && c_{v} (r_{v})=c_{v} (q_2)\ \left(s,\frac{\det q_{2}}{s}\right)_{v} .
\end{align*}
We know such forms exist by Theorem \ref{localexistence} (3).

For each infinite place $v\in V_k$, define form $t_v$ by:
\begin{align*}
&t_v=\left \langle1, \frac{\det q_2}{s}\right\rangle.
\end{align*}
At each complex place $t_v$ divides $q_2\otimes k_v$.  By assumption, $q_2$ is ordered at each real place $v\in V_k$, and hence $t_v=\langle 1, 1\rangle$ is a subform of $q_2\otimes k_v$.  Therefore at each infinite place it makes sense to take the complement of $t_v$ in $q_{2}\otimes k_v$ and we may define forms $r_v$ by
\begin{align*}
 r_v=t_v^\perp.
\end{align*}
At each complex place $v\in V_k$, we trivially have $c_v(r_v)=1=c_v(q_2).$  For each real $v\in V_k$, let  $(m_+^{(v)}, m_-^{(v)})$ denote the signature of $q_2\otimes k_v$.  Observe that $r_v$ has signature $(m_+^{(v)}-2, m_-^{(v)})$, and hence is isotropic whenever $q_2\otimes k_v$ is isotropic.  
Note that $c_v(r_v)=c_v(q_2).$

We shall now show that for each place $v \in V_k$, $t_{v}\oplus r_{v}\cong q_{2} \otimes k_v$.  This is true by construction at the infinite places.  Now suppose $v$ is finite.  Clearly
$$\dim(t_{v}\oplus r_{v})=1+(n-1)=n=\dim( q_{2} \otimes k_v),$$
$$\det(t_{v}\oplus r_{v})= (\det q_2 / s)s=\det( q_{2} \otimes k_v),$$
and by the product formula for the Hasse invariant
$$c(t_{v}\oplus r_{v})=c(t_{v})c(r_{v})\left(\frac{\det q_2}{s}, s\right)= c_v(q_2)\left(\frac{\det q_2}{s}, s\right)^2=c( q_{2} \otimes k_v),$$
and hence by Theorem \ref{localuniqueness} (3), they are isomorphic. 

We now wish to build a global form, and hence must check that our forms satisfy the compatibility criteria of Theorem  \ref{localglobal}.  Observe that $c_v(t_{v})=1$ for all $v\in V_k$ and hence $\prod_{v\in V_k} c_v(t_{v})=1$.  Next observe that by our choice of $s$, $\left(s,\frac{\det q_{2}}{s}\right)_{v}=1$ at each infinite place, and hence
\begin{align}\label{compatibilityeven2}
\prod_{v\in V_k} c_v(r_{v})&=\left(\prod_{v\in V_k\ \mathrm{ finite}} c_v(r_{v})\right)\times \left(\prod_{v\in V_k\ \mathrm{ real}} c_v(r_{v})\right) \times \left(\prod_{v\in V_k\ \mathrm{ complex}} c_v(r_{v})\right)\notag\\
&=\left(\prod_{v\in V_k\ \mathrm{ finite}}c_{v} (q_2)\ \left(s,\frac{\det q_{2}}{s}\right)_{v}\right)\times \left(\prod_{v\in V_k\ \mathrm{ real}} c_v(q_2)\right) \times \left(\prod_{v\in V_k\ \mathrm{ complex}} c_v(q_{2})\right)\notag\\
&=\prod_{v\in V_k} c_v(q_{2}) \times \prod_{v\in V_k} \left(s,\frac{\det q_{2}}{s}\right)_{v}\notag\\
&=1.
\end{align}
Again, the final product is trivial because both the Hasse invariant and the Hilbert symbol of global objects satisfy  product formulas.

By Theorem \ref{localglobal}, there exist quadratic forms $t$ and $r$ over $k$ such that for all $v\in V_k$, $t\otimes k_v\cong t_{v}$ and $r\otimes k_v\cong r_{v}$.  Furthermore, for each $v \in V_k$, we have shown that $t_{v}\oplus r_{v}\cong q_{2} \otimes k_v$ so by Theorem \ref{localglobaluniqueness} we conclude $t\oplus r \cong q_2$, and hence $r$ is a subform of $q_2$.

Let $\mathbf{H}=\mathbf{SO}(r)$.  
We will show that $\mathbf{H}\not\subset \mathbf{G}_1=\mathbf{SO}(q_1)$, and hence that there are no representatives $r'$ of $\mathbf{H}$ such that $r'\subset q_1$.  
Again $\mathbf{H}$ is a group of type $D_n$ over $k$.  
Let $r'$ be any representative of $\mathbf{H}$.  
As in the proof of Theorem \ref{finiterepodd}, the group $\mathbf{H}$ determines the following invariants of $r'$:
\begin{enumerate}
\item $\dim (r')=2n-2 = \dim(r)$.
\item $\mathrm{disc}_{v}( r')=1$ at precisely the places $v\in V_k$ where $\mathbf{H}\otimes k_v$ is a group of inner type (i.e., the $*$-action is trivial).  This means that  $\mathrm{disc}_{v}( r')=1$ if and only if  $\mathrm{disc}_{v}( r)=1$, or in other words, at the places where $\mathrm{disc}_{v}( r)=1$, then $\det r'=\det r$.
\item $c_v(r')=c_v(r)$ at each place $v$ where  $\mathrm{disc}_v(r)=1$  (see equation \eqref{formula2}).
\end{enumerate}

Let $r'$ be any quadratic form satisfying these three properties.  Suppose there exists some form $t'$ such that $r'\oplus t'\cong q_1$.  It follows that $\dim t'=2$, $\det t'= \det q_1/\det r'$. 

Our choice of $s$ implies that $$\mathrm{disc}_{v}( r)=(-1)^{(n-1)}\det r= (-1)^{2n-2}=1.$$
Hence at $v_0$, we have $\det r= \det r'$ and $c_{v_0}( r)= c_{v_0} (r')$.  
Furthermore we have 
\begin{align*}
\mathrm{det}_{v_0} t' 	&=\frac{\det_{v_0} q_1}{\det_{v_0} r'}\\
					&=\frac{(-1)^{\frac{m}{2}}\mathrm{disc}(q_1)}{(-1)^{\frac{m-2}{2}}\mathrm{disc}(r')}\\
					&=\frac{(-1)^{\frac{m}{2}}}{(-1)^{\frac{m-2}{2}}}\\
					&=-1,
\end{align*} 
and thus by the exceptional restriction, $c_{v_0}(t')=1$.  
The product formula at $v_0$ now yields the following contradiction:
\begin{align*}
c_{v_0}(q_1)	&=c_{v_0}(r'\oplus t') \\
			&= c_{v_0}(r')\left(\det r', \frac{\det q_1}{\det r'}\right)_{v_0}\\
			&=c_{v_0}(q_2)\left(\det r', \frac{\det q_2}{\det r'}\right)_{v_0}\left(\det r', \frac{\det q_1}{\det r'}\right)_{v_0}\\
			&=c_{v_0}(q_2)\left((-1)^{\frac{m-2}{2}}, \frac{\det q_1 \det q_2}{(\det r')^2}\right)_{v_0}\\
			&=c_{v_0}(q_2)\left((-1)^{\frac{m-2}{2}},(-1)^{\frac{m}{2}}\det q_2\right)_{v_0}\\
%			&=c_{v_0}(q_2)\underbrace{\left(-1,\mathrm{disc}(q_2)\right)_{v_0}}_{=1\mbox{ at }v_0}{}^{\frac{m-2}{2}}\\
%%			&=c_{v_0}(q_2)\left(-1,\det q_2\right)_{v_0}^{\frac{m-2}{2}}\left(-1,-1\right)_{v_0}^{\frac{m(m-2)}{4}}\\
			&=c_{v_0}(q_2)\left(-1,\mathrm{disc}(q_2)\right)_{v_0}^{\frac{m-2}{2}}.
\end{align*}
Hence $r$ is not isogroupic to a subform of $q_1$, concluding the proof.
\end{proof}
%Note that the final computation in the proof is where we use the fact that at $-1$ is a square at $v_0$.  Note that this assumption is not necessary for $m\equiv 2\ (\mathrm{mod}\ 4)$.  

%\section{Some examples}

\begin{ex}\label{hyperbolic6}
Consider the following 6-dimensional quadratic forms over $\Q$:
$$q_1=\langle 1, 1, 1, 3, 3, -1\rangle \qquad \mbox{and}  \qquad q_2=\langle 1, 1, 1, 1, 1, -5\rangle. $$
Observe that $\det q_1=-1\ne -5=\det q_2$.  Furthermore, $\mathrm{disc}_3(q_1)=1$,  but $\mathrm{disc}_3(q_2)=5$ which is not a square in $\Q_3$.  Furthermore, a quick computation shows $c_3(q_1)=-1$ and $c_3(q_2)=1$.  
By Theorem \ref{finiterepeven2}, there exists a 4-dimensional quadratic form $r\subset q_2$ so that  $\mathbf{H}:=\mathbf{SO}(r)\subset \mathbf{SO}(q_2)$ but $\mathbf{H}$ is not $\Q$-isomorphic to a subgroup of $\mathbf{SO}(q_1)$.
It is not hard to check that $r=\langle1,1, 1, -5\rangle$ is such a form.
\end{ex}

%---------------------------------------------------------------
%---------------------------------------------------------------

\subsection*{Constructing Subforms In Codimension $>2$}\label{sectionconstructionallin}

We have shown that given certain nonisometric forms, we may find codimension one or codimension two subforms of one that are not isogroupic to a subform of the other.  
We now show that this is the best we can hope for.

\begin{prop}\label{allin}Let $k$ be a number field and let $q_1$ and $q_2$ be $m$-dimensional quadratic forms over $k$, $m\ge 4$, which are isometric at each infinite place.  If $r$ is a $j$-dimensional subform of $q_1$, where $0<j<m-2$,  then $r$ is also a subform of $q_2$.
\end{prop}

\begin{proof}
%As usual, we construct forms locally from which we will obtain a global form.
For each finite $v\in V_k$, let $t_v$ be the $k_v$ form uniquely determined by 
\begin{itemize}
\item $\dim t_v = n-m$,
\item $\det t_v= \dfrac{\det q_2}{\det r}$, and
\item $c_v(t_v)= c_v(q_2)\ c_v(r) \ \left(\det(r), \dfrac{\det(q_2)}{\det(r)}\right)_v$.
\end{itemize}
Such forms exist by Theorem \ref{localexistence} (3).  
Since $q_1\otimes k_v$ and $q_2\otimes k_v$ are isometric at each infinite $v\in V_k$,  $r\otimes k_v$ is a subform of $q_2\otimes k_v$, and it makes sense to take its complement.  
We define 
\begin{itemize}
\item $t_v:= (r\otimes k_v)^\perp$,
\end{itemize}

It follows that at each infinite place $v$, $$c_v(t_v)=c_v(q_2)c_v(r)\left(\det(r), \frac{\det(q_2)}{\det(r)}\right)_v.$$

To build a global form, we must check that our forms satisfy the compatibility criteria of Theorem  \ref{localglobal}.  This can be seen with the following computation:

\begin{align*}
\prod_{v\in V_k} c_v(t_v)&=\left(\prod_{v\in V_k} c_v(q_2)c_v(r)\left(\det(r), \frac{\det(q_2)}{\det(r)}\right)_v\right)\\
&=\left(\prod_{v\in V_k} c_v(q_2)\right)\times \left(\prod_{v\in V_k} c_v(r)\right)\times \left(\prod_{v\in V_k} \left(\det(r), \frac{\det(q_2)}{\det(r)}\right)_v\right)\\
&=1.
\end{align*}
The final product is trivial because both the Hasse invariant and the Hilbert symbol of global objects satisfy the product formula.  
By Theorem \ref{localglobal} there is a quadratic form $t$ over $k$ such that for all $v\in V_k$, $t\otimes k_v\cong t_v$.  Furthermore, for each $v \in V_k$, $t_v\oplus r_v$ and $q_2\otimes k_v$ have the same local invariants so by Theorem \ref{localuniqueness} they are isometric, and by Theorem \ref{localglobaluniqueness} we conclude $t\oplus r \cong q_2$, and hence $r$ is a subform of $q_2$.
\end{proof}

%-----------------------------------------------------------------------------------------------------------
%--------------------------PROOF OF THEOREMS B AND C-----------------------------------
%-----------------------------------------------------------------------------------------------------------

\section{Proofs of Theorems \ref{thrmB} and \ref{thrmC}}\label{sectionmain}

\begin{proof}[Proof of Theorem \ref{thrmB}]
By Theorem \ref{thrmA} we may assume that $k(M_1)$ and $k(M_2)$ are isomorphic and let $k$ be a fixed representative of this isomorphism class.
We now prove the contrapositive.  
Suppose that $M_1$ and $M_2$ are not commensurable.
If $q_1$ and $q_2$ are $k$-groups giving rise to the spaces $M_1$ and $M_2$, respectively, then 
$q_1$ is not $k$-isometric to any twist of $q_2$
%and $\mathbf{G}_2$ are not $k$-isomorphic up to the action of $\mathrm{Aut}(k/\Q)$.
%The proof now follows from Proposition \ref{semisimpleext} below.
%\end{proof}
%
%\begin{prop}\label{semisimpleext}
%Let $k$ be a number field and $\mathbf{G}_1$ and $\mathbf{G}_2$ be semisimple $k$-groups coming from quadratic forms of dimension $m\ge 5$, 
%%If $m$ is even, 
%and suppose that the $\mathbf{G}_i$ are isotropic at a unique infinite place.
%If $\mathbf{G}_1$ and $\mathbf{G}_2$ are not $k$-isomorphic up to the action of $\mathrm{Aut}(k/\Q)$, then there exists a semisimple $k$-subgroup $\mathbf{H}$ of one that is not isogeneous to a $k$-subgroup of the other.  
%Furthermore, if either $\mathbf{G}_1$ or $\mathbf{G}_2$ is isotropic at a real place, $\mathbf{H}$ can be chosen to be isotropic at a real place.
%\end{prop}
%
%
%\begin{proof}
%Let $q_1$ and $q_2$ be quadratic forms over $k$ that represent $\mathbf{G}_1$ and $\mathbf{G}_2$ respectively.
By Lemma \ref{ordered}, we may assume that $q_1$ and $q_2$ are ordered and furthermore, we may replace $q_2$ with a twist so that $q_1$ and $q_2$ are isotropic at the same infinite place.
%Let $q_2^{(j)}$, $\{1,2,\ldots, \ell\}$, denote all forms in the $\mathrm{Aut}(k/\Q)$-orbit of $q_2$ that are isometric to $q_1$ at every real place.
%If this collection is empty, then Theorem \ref{realrep} implies the result, and hence we may assume this collection is nonempty.
Since $\mathbf{SO}(q_1)$ and $\mathbf{SO}(q_2)$ are not $k$-isomorphic, the Hasse principle for special orthogonal groups \cite[pg. 348]{PlRa} implies that there exist finite places $v_0$ where  $\mathbf{SO}(q_{1,v_0})$ and $\mathbf{SO}(q_{2,v_0})$ are not $k_{v_0}$-isomorphic, and hence the forms $q_1$ and $q_2$ are not isometric over $k_{v_0}$.

First suppose $m$ is odd.
By Lemma \ref{squareexistence}, we may replace $q_1$ and $q_2$ with similar forms as necessary to guarantee that $\det_{v_0^{(j)}} q_1=\det_{v_0}q_2=1$ while not altering the signatures at the infinite places.  
Hence $c_{v_0}(q_1)\ne c_{v_0}(q_2)$, and by Theorem \ref{finiterepodd} the result follows.  

Now suppose $m=2n$ is even.  
%By our additional assumption on the isotropic infinite places, there is just one $q_2^{(j)}$, which we now write as $q_2$.
If $\det q_1 = \det q_2$ but $\mathrm{disc}_{v_0}(q_i)\ne 1$, then by Lemma \ref{simc} and Lemma \ref{squareexistence}, we may replace $q_2$ with a similar form while not altering the signatures at the infinite place and for which $c_{v_0}(q_1)=c_{v_0}(q_2)$.  
This would imply $q_1$ and $q_2$ are isomorphic over $k_{v_0}$, contradicting our choice of $v_0$.  
Hence if $\det q_1 = \det q_2$, then after possibly relabeling, their  invariants must satisfy both of the following:
\begin{enumerate}
		\item $\mathrm{disc}_{v_0} (q_1)=1= \mathrm{disc}_{v_0} (q_2)$, and
		\item $c_{v_0}(q_1)=(-1,-1)_{v_0}^{\frac{n(n-1)}{2}} \ne -(-1,-1)_{v_0}^{\frac{n(n-1)}{2}}=c_{v_0}(q_2).$
	\end{enumerate}
By Theorem \ref{finiterepeven1} the result follows.  
Otherwise, if  $\det q_1 \neq \det q_2$, then, after possible relabeling, we have $\mathrm{disc}_{v_0} q_1=1$ and $\mathrm{disc}_{v_0} q_2\neq 1$.
%%there is some finite place $v_0\in V_k$ where:
%	\begin{enumerate}
%		\item $\mathrm{disc}_{v_0} q_1=1$,
%		\item $\mathrm{disc}_{v_0} q_2\neq 1$,
%%		\item $c_{v_0} (q_1)\ne c_{v_0}(q_2)$
%%		\item $-1$ is a square at $v_0$.
%	\end{enumerate}
Furthermore, if $c_{v_0}(q_1)=c_{v_0}(q_2)(-1, \mathrm{disc}(q_2))_{v_0}^{\frac{m-2}{2}}$, then we will replace $q_2$ with a similar form in the following way.  
%We will use Corollary \ref{squareexistence} to construct an element of $k^\times$.  
Let $S=\{v_0\}\cup \{\mbox{infinite real places of $k$}\}$  and for each $v\in S$, we pick a square class $\alpha_v\in k_v^\times/(k_v^\times)^2$ as follows:
\begin{itemize}
\item  at $v_0$, $(\alpha_{v_0} , \mathrm{disc}(q_2))_{v_0}=-1$ (note that such a class exists by the  the nondegeneracy of the Hilbert symbol and the  fact that $\mathrm{disc}(q_2)\neq 1$), and
\item for all $v\in S$ real, $\alpha_v$ is trivial. 
\end{itemize}
By Lemma \ref{simc}, it follows that $c_{v_0}(\lambda q_2)=-c_{v_0}(q_2)$ and replacing $q_2$ by $\lambda q_2$, it follows that $c_{v_0}(q_1)\neq c_{v_0}(q_2)(-1, \mathrm{disc}(q_2))_{v_0}^{\frac{m-2}{2}}$.  
By Theorem \ref{finiterepeven2} the result follows.
\end{proof}

Since commensurable spaces are length-commensurable, we have the following corollary.  

\begin{cor}\label{cor:tothrmB}
Let $M_1$ and $M_2$ be $\R$-simple arithmetic locally symmetric spaces coming from quadratic forms of dimension $\ge 5$.  If $\Q TG(M_1)=\Q TG(M_2)$, then $\Q L(M_1)=\Q L(M_2)$.
\end{cor}

Interestingly, Corollary \ref{cor:tothrmB} says that the set of totally geodesic subspaces determines the rational multiples of the lengths of all closed geodesics, even though there can exist closed geodesics that do not lie in any proper nonflat totally geodesic subspace.  
%The existence of such geodesics follows from the existence of $\R$-regular elements in these arithmetic lattices.  
%See \cite{Pr94} for an elementary proof of this fact.
Upon specializing to $\R$-rank one spaces, Theorem \ref{thrmC} follows from Theorem \ref{thrmB}.
Furthermore, unravelling the proof of Theorem \ref{thrmB}
% and Theorem \ref{semisimpleext}, 
we see that we can tell apart noncommensurable spaces of type $B_n$ by solely looking at codimension one subforms, yielding the following theorem.

\begin{thm}\label{hypersurfacessuffice}
Let $M_1$ and $M_2$ be even dimensional arithmetic hyperbolic $n$-orbifolds, $n\ge 4$.  
Suppose every codimension one, totally geodesic subspace in $M_1$ is commensurable to a codimension one totally geodesic subspace in $M_2$ and vice versa.  
Then $M_1$ and $M_2$ are commensurable.
\end{thm}

%-----------------------------------------------------------------------------------------------------------
%-----------------APPLICATIONS TO HYPERBOLIC ORBIFOLDS-------------------------
%-----------------------------------------------------------------------------------------------------------

\section{Hyperbolic Subspace Dichotomy and Other Applications}\label{sectionapplications}

In Construction \ref{subformsubspacedef} we showed that quadratic subforms give totally geodesic subspaces called \hyperref[subformsubspacedef]{subform subspaces}.  
We now show that in the case of standard arithmetic hyperbolic orbifolds, these are the only finite volume totally geodesic subspaces.

\begin{prop}\label{subformsubspaces}
If $M$ is a standard arithmetic hyperbolic $n$-orbifold, $n\ge 4$, and $N\in \Q TG(M)$ then $(i)$ $k(N)=k(M)$ and $(ii)$ 
$N$ is a \hyperref[subformsubspacedef]{subform subspace}.
%\begin{enumerate}
%\item $k(N)=k(M)$ and
%\item $N$ is a \hyperref[subformsubspacedef]{subform subspace}.
%\end{enumerate}
%there is a quadratic subform $r$ of $q$, defined over $k$, such that $N_r=N$.
\end{prop}

\begin{proof}
By assumption, $M=M_q$ where $(V,q)$ is a quadratic $m$-space, $m=n+1\ge 5$, over a totally real number field $k$ with a unique real place $v$ where $q$ is isotropic.
Let $\mathbf{G}=\mathbf{SO}(V,q)$ and denote $k_v$ by $\R$.
Let $H\subset G:=\mathbf{G}(\R)$ be the connected semisimple Lie subgroup giving rise to $N$.  
Since $M_q$ is hyperbolic, it follows that $H=\mathbf{H}(\R)^\circ$ where $\mathbf{H}=\mathbf{SO}(W',r')$ for some $\R$-subspace $W'\subset V_\R$ and $r'$ the restriction of $ q_\R$ to $W'$.  
Let $L\subset V$ be an $\mathcal{O}_k$-lattice and let $G_{L}$ be its stabilizer in $G$.  
Since $\Lambda:=G_{L}\cap H$ is a lattice in $H$, it is Zariski-dense in $\mathbf{H}$.
It follows that the  $\R$-$\mathrm{span}$ of $L\cap W'$ must be all of $W'$.
Let $W$ denote the $k$-span of $L\cap W'$ and let $r$ be the restriction of $q$ to $W$.  
Then $N=N_r$ and the result follows. 
\end{proof}

%We will now give some geometric consequences of these results.  
%The first is the proof of Theorem \ref{thrmD}, our Hyperbolic Dichotomy Theorem.

\begin{proof}[Proof of Theorem \ref{thrmD}]
If $M_1$ and $M_2$ share a single finite volume totally geodesic subspace, Proposition \ref{subformsubspaces} implies $k(M_1)$ and $k(M_2)$ are isomorphic. 
Let $k$ be a fixed representative of the isomorphism class of $k(M_1)$ and $k(M_2)$.
We may now choose quadratic forms $q_1$ and $q_2$ over $k$ that are isotropic at the same real place and that give rise to $M_1$ and $M_2$, respectively.
The result then follows by Proposition \ref{allin}.  
\end{proof}

In particular, since all noncompact arithmetic hyperbolic $n$-orbifolds, $n\ge 4$, come from $k=\Q$, 
we have the following corollary.

\begin{cor}
Let $M_1$ and $M_2$ be $n$-dimensional ($n\ge 4$) noncompact, standard arithmetic hyperbolic orbifolds.
Then, up to commensurability, $M_1$ and $M_2$ have the exact same collection of finite volume totally geodesic subspaces of noncompact type of codimension $>2$.
\end{cor}

Recent work by McReynolds \cite{McRey} shows that certain noncommensurable arithmetic manifolds arising from the semisimple Lie groups of the form $(\mathbf{SL}_d(\R))^r \times (\mathbf{SL}_d(\C))^s$ have the same commensurability classes of totally geodesic surfaces coming from a fixed field.   An immediate consequence of our work above proves the following

\begin{cor}
For each $n\ge 4$, there exist noncommensurable, standard arithmetic, hyperbolic $n$-orbifolds $M_1$ and $M_2$ that have the same commensurability classes of totally geodesic surfaces.
\end{cor}

We conclude this section by addressing the following question was posed to us by Jean-Fran\c{c}ois Lafont:
\textit{If $M_1$ and $M_2$ are good Riemannian orbifolds, when is it the case that $\Q TG(M_1)\subset \Q TG(M_2)$ implies $M_1\subset M_2$?}
%\begin{ques}
%Let $M_1$ and $M_2$ be good Riemannian orbifolds.  When is it the case that $\Q TG(M_1)\subset \Q TG(M_2)$ implies $M_1\subset M_2$?
%\end{ques}
We can answer this question for standard arithmetic hyperbolic orbifolds.  
%When the difference $\dim M_2- \dim M_1$ is large, we have a positive answer.

\begin{prop}
Let $M_1$ and $M_2$ be standard arithmetic hyperbolic spaces.  
\begin{enumerate}
\item If $3 \le \dim M_1 \le \dim M_2-3$ and $\Q TG(M_1)\subset \Q TG(M_2)$, then, $M_1$ is commensurable to a totally geodesic subspace of $M_2$.
\item If $3 \le  \dim M_2-2\le \dim M_1$, there exist examples for which $\Q TG(M_1)  \subset \Q TG(M_2)$ but $M_1$ is not commensurable to a totally geodesic subspace of $M_2$
%There If $3 \le  \dim M_2-3< \dim M_1$ and $\Q TG(M_1)\subset \Q TG(M_2)$, then, $M_1$ is commensurable to a totally geodesic subspace of $M_2$.
\end{enumerate}
\end{prop}

\begin{proof}
We begin by showing part $(i)$. 
By assumption, every totally geodesic surface of $M_1$ is totally geodesic in  $M_2$.
Proposition \ref{subformsubspaces} implies that $k(M_1)= k(M_2)=:k$.  
Let  $q_i$ be quadratic forms over $k$ which give rise to $M_i$ that are isotropic at the same real place of $k$.   
%Our assumption on $\Q TG$ shows that $q_1$ and $q_2$ are isotropic at the same real place of $k$.  
Then by Proposition \ref{allin}, it follows that $q_1$ is a subform of $q_2$ and $(i)$ follows. 
Part $(ii)$ follows from Example \ref{interestingex} below.
\end{proof}

%However when $\dim M_2 - \dim M_1$ is small we can have a negative answer as shown in Example \ref{interestingex}.

\begin{ex}\label{interestingex}
Consider following quadratic forms over $\Q$ described in Example \ref{hyperbolic6}:  
$$q_1=\langle 1,1, 1, -5\rangle \qquad \mbox{and}  \qquad q_2=\langle 1, 1, 1, 3, 3, -1\rangle. $$
By Theorem \ref{finiterepeven2} the 3-dimensional hyperbolic space $M_{q_1}$ is not commensurable to a totally geodesic subspace of the 5-dimensional space $M_{q_2}$, yet by Proposition \ref{allin}, they contain precisely the same totally geodesic surfaces.  
\end{ex}

%\begin{prop}
%There exist arithmetic hyperbolic orbifolds $M_1$ and $M_2$ for which $\Q TG(M_1)  \subset \Q TG(M_2)$ but $M_1$ is not commensurable to a totally geodesic subspace of $M_2$.  
%\end{prop}

%-----------------------------------------------------------------------------------------------------------
%----------------------------------NONSTANDARD D_N-------------------------------------------
%-----------------------------------------------------------------------------------------------------------

\section{Nonstandard Arithmetic Hyperbolic Manifolds\\ and Proofs of Theorems \ref{thrmE} and \ref{thrmF}}\label{sectionmaind}

The arithmetic lattices in groups of type $D_n$ that arise from skew Hermitian forms over division algebras over number fields we call \textbf{nonstandard lattices}.
For the algebraic theory of skew Hermitian forms, we refer the reader to \cite{Sch}.
In this section $m=2n\ge4$ and $\mathbb{H}$ denotes Hamilton's quaternions over $\R$.
%It is similar to Construction \ref{quadcon} for quadratic forms.

\begin{con}\label{hermcon}Fix the following notation:
\begin{enumerate}
\item  $k$ is a number field with infinite places $V_k^\infty$;
\item  $D$ is a quaternion division algebra with center $k$;
\item  $(V,h)$ is an $n$-dimensional skew Hermitian space over $D$, $n\ge2$;
\item  $\mathbf{G}:=\mathbf{SU}(V,h)$ is the absolutely almost simple $k$-group defined by $(V,h)$ 
\item For each $v\in V_k^\infty$,   $V_{k_{v}}:=V\otimes_kk_{v}$, $q_{v}:=q\otimes k_{v}$, and $\mathbf{G}_{v}$ is the algebraic $k_{v}$-group $\mathbf{SU}(V_{k_{v}},q_{v})$
\begin{itemize}
\item If $v$ is real, and $D$ ramifies over $k_v$, then $\mathbf{G}_{v}(k_v)\cong \mathbf{SO}_{n}(\mathbb{H})$.  
\item If $v$ is real, and $D$ splits over $k_v$, then $\mathbf{G}_{v}(k_v)\cong \mathbf{SO}(m_+^{(v)}, m_-^{(v)})$.  
\item If $v$ is complex, then $\mathbf{G}_{v}(k_v)\cong \mathbf{SO}_{2n}(\C)$. 
		\item $q$ is the number of real places where $D$ ramifies,  
		\item $r$ is the number of real places where $D$ splits and $h$ is isotropic,  
		\item $s$ is the number of complex places, and
		\item $p_i=m_+^{(v_i)}$ where $\{v_1, \ldots, v_r\}$ is the set  of real places where $D$ splits and $h$ is isotropic;
\end{itemize}
\item $G':=(R_{k/\Q}\mathbf{G})(\R)$, $G$ is the projection of $G'$ and onto its noncompact factors, and $\pi:G'\to G$ is the projection map.
%\item Let $\mathbf{G}_{v}$ denote the algebraic $k_{v}$-group $\mathbf{SO}(V_{k_{v}},q_{v})$ for each  $v\in V_k^\infty$.
%\begin{itemize}
%\item If $v$ is real, and $D$ splits over $k_v$, then $\mathbf{G}_{v}(\R)\cong \mathbf{SO}(m_+^{(v)}, m_-^{(v)})$.  
%\item If $v$ is real, and $D$ ramifies over $k_v$, then $\mathbf{G}_{v}(\R)\cong \mathbf{SO}_{n}(\mathbb{H})$.  
%\item If $v$ is complex, then $\mathbf{G}_{v}(\R)\cong \mathbf{SO}_{2n}(\C)$. 
%\end{itemize}
%\item $G$ is the projection of $G'$ onto its noncompact factors and denote the projection map by $\pi:G'\to G$.   
 Observe that $G$ is a semisimple Lie group with no compact factors and is $\R$-simple when $q+r+s=1$.
\item  $\mathcal{O}_D$ is an order in $D$, $L\subset V$ is an $\mathcal{O}_D$-lattice,  and $G_L:=\{T\in \mathbf{G}(k) \ | \ T(L)\subset L\}$, which is a discrete subgroup of $G'$
\item $\Gamma\subset G$ is commensurable up to $G$-automorphism with $\pi(G_L)$.  
Then $\Gamma$ is said to be a \textbf{nonstandard arithmetic lattice of $G$}.  
Figure \ref{figure:nonstandarddiagram}   summarizes this construction.
%\item $\Gamma\subset G$ be commensurable up to $G$-automorphism with $\pi(G_L)$.  Then $\Gamma$ is a \textbf{nonstandard arithmetic lattice of $G$}.
%\item Letting: 
%	\begin{itemize}
%		\item $q$ be the number of real places where $D$ ramifies,  
%		\item $r$ be the number of real places where $D$ splits and $h$ is isotropic,  
%		\item $s$ be the number of complex places, and
%		\item $p_i=m_+^{(v_i)}$ where $\{v_1, \ldots, v_r\}$ is the set  of real places where $D$ splits and $h$ is isotropic,
%	\end{itemize} 
%Figure \ref{figure:nonstandarddiagram} below illustrates the construction of nonstandard arithmetic lattices in $G$.
%\vspace{-1pc}
\begin{figure}[htbp]
%\hspace*{-1pc} 
\begin{tikzpicture}[scale=1.45]
%\node (A) at (-1,1) {$SU(V,h)$};
\node (B) at (4.5,1) {$G':=\left(\displaystyle\prod\limits_{\substack{v \ \mathrm{real} \\ D\ \mathrm{ramifies}}} \mathbf{SO}_n(\mathbb{H}) \quad \times \displaystyle\prod\limits_{\substack{v \ \mathrm{real} \\ D\ \mathrm{splits}}} \mathbf{SO}(m_+^{(v)}, m_-^{(v)}) \quad \times \displaystyle\prod\limits_{v\ \mathrm{complex}}\mathbf{SO}_{2n}(\C)\right)\qquad$ };
\node (C) at (-2.5,1) {$G_L$};
\node (D) at (4.5,-1.5) {$G:= (\mathbf{SO}_n(\mathbb{H}))^q\times \displaystyle\prod\limits_{i=1}^r \mathbf{SO}(p_i,2n-p_i) \times (\mathbf{SO}_{2n}(\C))^s$};
\node (E) at (-2.5,-1.5) {$\Gamma$};
\path[right hook->,font=\scriptsize,>=angle 90]
(C) edge node[above]{diagonal} (B)
%(C) edge node[above]{} (A)
(E) edge node[above]{Commensurable (up to $G$-automorphism)} (D)
(E) edge node[below]{with $\pi(G_L)$} (D);
\path[->,font=\scriptsize,>=angle 90]
(C) edge node[below]{lattice} (D)
(B) edge node[right]{$\pi$} (D);
\end{tikzpicture}
\vspace{-2.0pc}
		\caption{Construction of Nonstandard arithmetic lattices in $G$.}
		\label{figure:nonstandarddiagram}
\end{figure}	
%\vspace{-0.5pc}
\item   $K\subset G$ is a maximal compact subgroup and $M_{\Gamma}:=\Gamma \backslash G /K$.  Then
\begin{enumerate}
	\item $M_\Gamma$ is a \textbf{nonstandard arithmetic locally symmetric space of $D_n$}, 
	\label{fieldalgebradef}
	\item $k(M_\Gamma):=k$ is the \textbf{field of definition} of $M_{\Gamma}$, and
	\item $D(M_\Gamma):=D$ is the \textbf{algebra of definition} of $M_{\Gamma}$.
\end{enumerate} 
\end{enumerate}
\end{con}
A choice of another order in $D$ and another lattice in $V$ will produce a space commensurable with $M_{L}$. 
Hence choosing $h$ determines a commensurability class that we denote by $M_h$.  
More on this construction can be found in \cite[\S2]{LM}.
Just as for quadratic forms (Construction \ref{subformsubspacedef}), skew Hermitian subforms $h'\subset h$ give rise to commensurability classes totally geodesic subspaces $N_{h'}\subset M_h$, which we also call subform subspaces.

\begin{Def}
We call $(k,D,h)$ an \textbf{admissible hyperbolic triple} if $M_h$ is a commensurability class of hyperbolic orbifolds, that is to say, 
\begin{enumerate}
\item $k$ is totally real (i.e., no $\mathbf{SO}_{2n}(\C)$ terms),
\item $D$ splits at all real places (i.e., no $\mathbf{SO}_{m}(\mathbb{H})$ terms),
\item $h$ is anisotropic at all but one real place, $v_0$ (i.e., only one $\mathbf{SO}(p_i, 2n-p_i)$ term),
\item $q\otimes k_{v_0}$ has signature $(2n-1,1)$ or $(1,2n-1)$.  (i.e., $G\cong\mathbf{SO}(2n-1,1)$).
\end{enumerate}
\end{Def}
%With Construction \ref{hermcon}, we are now able to prove Theorems \ref{thrmE} and \ref{thrmF}.

\begin{proof}[Proof of Theorem \ref{thrmE}] 
If $\dim M_1\ne \dim M_2$, it is not hard to find a subform subspace of one that is not commensurable to a proper totally geodesic subspace of the other, so suppose $\dim M_1= \dim M_2$.
Then there exists a $2n$-dimensional quadratic form over a number field $k_1$ and an $n$-dimensional skew Hermitian form over a number field $k_2$ giving rise to $M_1$ and $M_2$ respectively.
Furthermore, as we have already seen, we may choose a codimension one subform $r\subset q$ that is isotropic at the same real places as $q$.
Suppose that $N_r\in \Q TG(M_2)$.
By Proposition \ref{commimpliesisothm},  $R_{k_1/\Q}\mathbf{SO}(r)$ is $\Q$-isogenous to a $\Q$-subgroup of $R_{k_2/\Q}\mathbf{SU}(h)$.
Observe that, when $n\ge 4$, 
$$\dim \mathbf{SU}(h)=n(2n-1)<(2n-1)(2n-2)=2\left(\frac{(2n-1)(2n-2)}{2}\right)=2\dim \mathbf{SO}(r),$$
hence by Proposition \ref{fielddefcontain}, $k_1$ and $k_2$ are isomorphic. 
Let $k$ be a fixed representative of this isomorphism class and replace $h$ with a twist so that $q$ and $h$ are forms over $k$ that are isotropic at the same infinite place.
If $v\in V_k$ is a finite place where $D$ ramifies, then  
$$\mathrm{rank}_{k_v}(\mathbf{SU}(h)) \le \frac{n}{2} \le n-2 \le \mathrm{rank}_{k_v}(\mathbf{SO}(r)).$$
By local rank considerations $\mathbf{SO}(r)$ cannot be a subgroup of $\mathbf{SU}(h)$.  
\end{proof}

\begin{proof}[Proof of Theorem \ref{thrmF}] 
Let $\mathbf{G}_i=\mathbf{SU}(h_i)$, $i=1,2$, be groups giving rise to $M_i$, where $h_i$ is an $n$-dimensional skew Hermitian form over $D_i$ .
Let $r$ be an $(n-1)$-dimensional Hermitian subform of $h_1$ which is isotropic at the real place $h_1$ is isotropic.  
Let $\mathbf{H}_1:=\mathbf{SU}(r)$.  
By assumption, the subform subspace $N_r\in \Q TG(M_2)$.
By Proposition \ref{commimpliesisothm},  $R_{k_1/\Q}\mathbf{H}_1$ is $\Q$-isogenous to a $\Q$-subgroup of $R_{k_2/\Q}\mathbf{G}_2$.
When $n\ge 4$, 
$$\dim \mathbf{SU}(h_2)=n(2n-1)<2(n-1)(2n-3)=2\dim \mathbf{SU}(r),$$
and hence by Proposition \ref{fielddefcontain}, $k_1$ and $k_2$ are isomorphic.  
Let $k$ be a fixed representative of this isomorphism class and replace $h_2$ with a twist so that $h_1$ and $h_2$ are forms over $k$ that are isotropic at the same infinite place and $D_1$ and $D_2$ are quaternion algebras over $k$.
Suppose that $D_1$ and $D_2$ are not isomorphic.
Then there is a finite 
place $v\in V_k$ where one splits and the other ramifies.  
After relabeling if necessary, we may assume $D_1$ splits and $D_2$ ramifies.  
When $n\ge 4$, $$\mathrm{rank}_{k_v}(\mathbf{G}_2) \le \frac{n}{2} \le n-2 \le \mathrm{rank}_{k_v}(\mathbf{H}).$$
Hence again by local rank considerations $\mathbf{H}$ cannot be a subgroup of $\mathbf{G}_2$ and the result follows.
\end{proof}

\begin{ques}\label{conjecturedn}
Let $M_1$ and $M_2$ be $\R$-simple, nonstandard arithmetic, locally symmetric spaces.
Does $\Q TG(M_1)= \Q TG(M_2)$ imply $M_1$ and $M_2$ are commensurable?
\end{ques}

Question \ref{conjecturedn} remains open.  
The primary obstacle to answering this question is the lack of local and global existence theorems for skew Hermitian forms over division algebras.  
Answering this would complete the analysis of the rational totally geodesic spectrum for $\R$-simple arithmetic spaces of type $D_n$ for all $n\ge 2$ not arising from triality in $D_4$.

%-----------------------------------------------------------------------------------------------------------
%-------------------------------ACKNOWLEDGEMENTS------------------------------------------
%-----------------------------------------------------------------------------------------------------------

\section*{Acknowledgments}\label{sectionack}
We thank Matthew Stover  for suggesting this problem and for countless helpful conversations.  
The author was supported by the NSF RTG grant 1045119.  
We would like to also thank Jean-Fran\c{c}ois Lafont, Lucy Lifschitz, Benjamin Linowitz, D. B. McReynolds, and Ralf Spatzier for their interest in this project and for many valuable and interesting discussions.

\appendix

%-----------------------------------------------------------------------------------------------------------
%---------------------------- MACHLACHLAN'S SECTION --------------------------------------
%-----------------------------------------------------------------------------------------------------------

\section{Machlachlan's Theorem: Parametrizing Commensurability Classes}\label{sectionmaclachlan}

In this section, we show how the techniques of Section \ref{sectiontits} may be used to parametrize commensurability classes of even dimensional arithmetic hyperbolic orbifolds.  
In so doing, we provide an alternate proof of the results of Maclachlan \cite{Mac},
whose proof uses techniques from the theory of quaternion algebras and Clifford algebras.  

\begin{thm}[Maclachlan \cite{Mac} Theorem 1.1]\label{macthrm}
The commensurability classes of arithmetic subgroups of $\mathrm{Isom}(\mathbb{H}^{2n})$, $n\ge 1$, are parametrized for each totally real number field $k\subset \R$ by sets $\{\mathfrak{p}_1, \mathfrak{p}_2, \ldots, \mathfrak{p}_r\}$ of prime ideals in the ring of integers $\mathcal{O}_k$ where 
\begin{align}\label{paramequation}r\equiv 
\begin{cases} 
0\ (\mathrm{mod}\ 2) & \mbox{if } n\equiv 0\ (\mathrm{mod}\ 4),\\
[k:\Q]-1\ (\mathrm{mod}\ 2) & \mbox{if } n\equiv 1\ (\mathrm{mod}\ 4),\\
[k:\Q]\ (\mathrm{mod}\ 2) & \mbox{if } n\equiv 2\ (\mathrm{mod}\ 4),\\
1\ (\mathrm{mod}\ 2) & \mbox{if } n\equiv 3\ (\mathrm{mod}\ 4).
\end{cases}\end{align}
\end{thm}

A place $v\in V_k$ is called \textbf{dyadic} if $k_v$ is nonarchimedean with residue field of characteristic 2.  

\begin{lem}\label{dyadlemma}
Let $k$ be a totally real number field and define
$$\delta(k):=\left\{\mbox{number of dyadic places where $\left(\frac{-1,-1}{\Q}\right)$ ramifies}\right\}.$$Then $\delta(k)\equiv [k:\Q] \ (\mathrm{mod}\ 2)$.  
\end{lem}

\begin{proof}
Over $\Q$, Hamilton's quaternions ramify at precisely 2 and $\infty$.   
Hence over $k$, Hamilton's quaternions ramify at precisely $\delta(k)$ places over $2$, $[k:\Q]$ places over $\infty$, and nowhere else.  
Since a quaternion algebra ramifies at an even number of places, the result follows.
\end{proof}

\begin{proof}[Proof of Theorem \ref{macthrm}]
We shall show that $k$-isomorphism classes of groups giving rise to standard arithmetic hyperbolic manifolds are parametrized by sets of the form $(v,  \{\mathfrak{p}_1, \mathfrak{p}_2, \ldots, \mathfrak{p}_r\})$
where $v\in V_k$ is a real place and $\{\mathfrak{p}_1, \mathfrak{p}_2, \ldots, \mathfrak{p}_r\}$ is a set of prime ideals  satisfying  \eqref{paramequation}.  
The theorem then follows from \cite[Prop 2.5]{PR} and our remarks after Construction \ref{quadcon} regarding twists of quadratic forms.

By Proposition \ref{simparamgrp}, similarity classes of $(2n+1)$-dimensional quadratic forms over $k$ parametrize groups of type $B_n$ over $k$.  
Picking the determinant 1 representative of each similarity class, the set
$$\mathcal{F}:=\{ q \ | \ \dim q = 2n+1, \det q = 1, \mbox{  and $(k,q)$ is an admissible hyperbolic pair}\}$$
parametrizes $k$-isomorphism classes of groups giving rise to arithmetic hyperbolic $2n$-orbifolds.
Let $v_1, \ldots v_\ell$ denote the real embeddings of $k$.    
For a fixed $v_i$, for $1\le i\le \ell$, let  $$\mathcal{F}_i:=\{q\in \mathcal{F}\ |\ \mbox{$q$ is isotropic at $v_i$}\}.$$ 

For $q\in \mathcal{F}_i$, the fact that $\det q=1$ now implies that $q$ has signature $(1, 2n)$ at $v_i$ and signature $(2n+1,0)$ at all other real places.  
A basic computation shows that the Hasse invariants at the real places are
$$c_{v_{j}}(q)=\begin{cases} (-1)^n & i=j\\ 1&i\ne j.\end{cases}$$

Let $V_k^s=\{ v\in V_k \ | \ (-1,-1)_v=+1\}$ and $V_k^r=\{ v\in V_k \ | \ (-1,-1)_v=-1\} $.  
These sets correspond to the finite places where Hamilton's quaternions split
and ramify, respectively.
For $q\in \mathcal{F}_i$, let $e_s(q)$ (resp. $e_r(q)$) denote the number of finite places in $V_k^s$ (resp. $V_k^r$) where $\mathbf{SO}(q)$ is not split.  
Clearly $r(q):=e_s(q)+e_r(q)$ is the total number of finite places where $\mathbf{SO}(q)$ is not split.  
(Note that this is always finite because any $k$-group is quasi-split at all but finitely many places and quasi-split groups of type $B_n$ are split.)

Since $q$ has determinant 1, \eqref{formula1} may be simplified to state that $\mathbf{SO}(q)$ splits over $v$ if and only if 
$c_v(q)=(-1,-1)_v^{\frac{n(n-3)}{2}}.$
Let $f_s(q)$ (resp. $f_r(q)$) denote the number of finite places $v$ in $V_k^s$ (resp. $V_k^r$) where $c_v(q)=-1$.  If as in Lemma \ref{dyadlemma}, $\delta(k)$ is the number of dyadic places where $\left(\frac{-1,-1}{\Q}\right)$ ramifies, then it follows that:
\begin{itemize}
\item $f_s(q)=e_s(q)$, and
\item $f_r(q)=\begin{cases}e_r(q)& \mbox{if }n\equiv 0,3 \ (\mathrm{mod}\ 4),\\  \delta(k)-e_r(q)  &  \mbox{if }n\equiv 1,2 \ (\mathrm{mod}\ 4).\\   \end{cases}$
\end{itemize}

By Theorem \ref{localglobal}, the local Hasse invariants of $q$ must satisfy the compatibility condition $\prod_{v\in V_k}c_v(q)=1$.   It follows that 
$$(-1)^n (-1)^{f_s(q)}(-1)^{f_r(q)}=1$$
and hence \begin{align}\label{formulaparam}n+f_s(q)+f_r(q) \equiv 0 \ (\mathrm{mod}\ 2).\end{align}
We now have the following four cases:
\begin{itemize}
\item \textbf{Case 1}:  $n\equiv 0\ (\mathrm{mod}\ 4)$\\
Equation \eqref{formulaparam} immediately gives $r(q)\equiv 0\ (\mathrm{mod}\ 2)$.
\item \textbf{Case 2}:  $n\equiv 1\ (\mathrm{mod}\ 4)$\\
Equation \eqref{formulaparam} gives 
$$n+e_s(q)+\delta(k)-e_r(q)\equiv 0\ (\mathrm{mod}\ 2).$$
By Lemma \ref{dyadlemma} and simplifying,
$$1+e_s(q)+[k:\Q]-e_r(q)\equiv 0\ (\mathrm{mod}\ 2),$$
and hence
$$r(q)\equiv[k:\Q]-1\ (\mathrm{mod}\ 2).$$
\item \textbf{Case 3}:  $n\equiv 2\ (\mathrm{mod}\ 4)$\\
Again using Lemma \ref{dyadlemma}, equation \eqref{formulaparam} gives
$$0+e_s(q)+[k:\Q]-e_r(q)\equiv 0\ (\mathrm{mod}\ 2),$$
and hence
$$r(q)\equiv[k:\Q]\ (\mathrm{mod}\ 2).$$
\item \textbf{Case 4}:  $n\equiv 3\ (\mathrm{mod}\ 4)$\\
Equation \eqref{formulaparam} immediately gives $r(q)\equiv 1\ (\mathrm{mod}\ 2)$.
\end{itemize}

We conclude that every form $q\in \mathcal{F}$ determines a set $(v_q, \{v_1, v_2,\ldots, v_{r(q)}\})$ where $v_q$ is the unique real place where $q$ is isotropic,  $\{v_1, v_2,\ldots, v_{r(q)}\}$ is the set of finite places where $\mathbf{SO}(q)$ is not split over $k_v$, and  $r(q)$ satisfies equation \eqref{paramequation}.
Furthermore, by the local-to-global uniqueness of Theorem \ref{localuniqueness}, no two forms $q,q'\in \mathcal{F}$ determine the same set of places.

Lastly we show that any collection $(v_0,  \{v_1, v_2, \ldots, v_r\})$ where $v_0\in V_k$ is a real place, $\{v_1, v_2, \ldots, v_r\}$ is a set of finite places, and $r$ satisfies equation \eqref{paramequation}, determines a form in $\mathcal{F}$.  
Let $\{q_v\}_{v\in V_k}$ be a family of $(2n+1)$-dimensional forms of determinant 1 satisfying the following:
\begin{itemize}
\item $q_{v_0}$ has signature $(1,2n)$,
\item $q_v$ has signature $(2n+1,0)$ at all other real places,
\item for $v\in V_k$ finite, $\mathbf{SO}(q_v)$ is not split if and only if $v\in \{v_1, v_2,\ldots, v_r\}$, and hence $c_v(q_v)$ is determined by equation \eqref{formula1}.
\end{itemize}
The above computations show that this family satisfies the compatibility condition of Theorem \ref{localglobal}, and hence there exists a global form $q\in \mathcal{F}$ with localizations $q_v$.   
By Construction \ref{quadcon}, we obtain a commensurability class and the result follows.
\end{proof}

With proper modification, these techniques may be used to rederive Maclachlan's parametrization of commensurability classes of odd dimensional standard arithmetic hyperbolic orbifolds \cite[Cor. 7.5]{Mac}.  
Furthermore, with additional modifications, these techniques are generalizable to give parameterizations of commensurability classes of standard arithmetic lattices in groups of type $B_n$ and $D_n$.


\begin{thebibliography}{47}
%
%\bibitem[AT]{AT}E. Artin and J. Tate
%\newblock\emph{Class Field Theory}.
%\newblock  Benjamin, New York, (1968).

%\bibitem[BMS]{BMS}H. Bass, J. Milnor, and J.-P. Serre
%\newblock\emph{Solution of the congruence subgroup problem for $SL_n$ $(n\ge3)$ and $Sp_{2n}$ $(n\ge2)$}.
%\newblock Inst. Hautes \'Etudes Sci. Publ. Math. No. 33 (1967) 59-137.
\bibitem{Belolipetsky}M. Belolipetsky
\newblock\emph{Finiteness theorems for congruence reflection groups}.
\newblock  Transform. Groups 16 (2011), no. 4, 939-954. 

\bibitem{B1}A. Borel
\newblock\emph{Linear Algebraic Groups}.
\newblock  Graduate Texts in Mathematics, \textbf{126}, Springer-Verlag, (1991).

\bibitem{B}A. Borel
\newblock\emph{Introduction aux groupes arithmŽtiques}.
\newblock  Publications de l'Institut de Math\'ematique de l'Universit\'e de Strasbourg, XV. Actualit\'es Scientifiques et Industrielles, \textbf{1341}, Paris: Hermann, (1969).

\bibitem{B60}A. Borel
\newblock\emph{Density properties for certain subgroups of semi-simple groups without compact components}.
\newblock Ann. of Math. (2)  \textbf{72} (1960) 179-188.

%\bibitem[B63]{B3}A. Borel
%\newblock\emph{Compact Clifford-Klein forms of symmetric spaces}.
%\newblock Topology \textbf{2} (1963) 111-122.

\bibitem{B65}A. Borel
\newblock\emph{Density and maximality of arithmetic subgroups}.
\newblock J. Reine Angew. Math. \textbf{224} (1966) 78-89.

%\bibitem[BoHa]{BoHa}A. Borel and G. Harder
%\newblock\emph{Existence of discrete cocompact subgroups of reductive groups over local fields}.
%\newblock  J. Reine Angew. Math. \textbf{298} (1978), 53Ð64.

\bibitem{BoHC}A. Borel and Harish-Chandra
\newblock\emph{Arithmetic Subgroups of Algebraic Groups}.
\newblock  Ann. of Math. (2)  \textbf{75}, No. 3, (1962).

%\bibitem{BoPr}A. Borel and G. Prasad
%\newblock\emph{Finiteness theorems for discrete subgroups of bounded covolume in semi-simple groups}.
%\newblock  Publ. Math. I.H.E.S., \textbf{69} (1989), 119-171.


\bibitem{BoTi}A. Borel and J. Tits
\newblock\emph{Groupes r\'eductifs}.
\newblock  Publ. Math. I.H.E.S., \textbf{27} (1965), 55-150.


\bibitem{Cassels}J. W. S. Cassels
\newblock\emph{Rational Quadratic Forms}.
\newblock  London Mathematical Society Monographs, 13. Academic Press, Inc., London-New York, (1978).

%\bibitem[Ch]{Ch}I. Chavel
%\newblock\emph{Eigenvalues in Riemannian Geometry}.
%\newblock  Pure and Applied Mathematics, \textbf{115}. Academic Press, Inc., Orlando, FL, (1984).

\bibitem{CHLR}T. Chinburg, E. Hamilton, D. D. Long, and A. W. Reid
\newblock\emph{Geodesics and commensurability classes of arithmetic hyperbolic 3-manifolds}.
\newblock  Duke Math. J, \textbf{145} No. 1 (2008), 25-44.

\bibitem{CGP}B. Conrad, O. Gabber, and G. Prasad
\newblock\emph{Pseudo-reductive groups}.
\newblock  New Mathematical Monographs, 17. Cambridge University Press, Cambridge, (2010).


\bibitem{DC}M.~P. Do Carmo.
\newblock\emph{Riemannian Geometry}.
\newblock Birkh\"auser, Boston, MA, (1992).

\bibitem{G}S. Garibaldi
\newblock\emph{Outer automorphisms of algebraic groups and determining groups by their maximal tori}.
\newblock  Michigan Mathematical Journal, \textbf{61}, \#2 (2012), 227-237.

\bibitem{GPS}M. Gromov and I. Piatetski-Shapiro
\newblock\emph{Nonarithmetic groups in Lobachevsky spaces}.
\newblock  Inst. Hautes \'Etudes Sci. Publ. Math. No. 66 (1988), 93-103. 

\bibitem{GS}M. Gromov and R. Schoen
\newblock\emph{Harmonic maps into singular spaces and p-adic superrigidity for lattices in groups of rank one}.
\newblock  Inst. Hautes \'Etudes Sci. Publ. Math. No. 76 (1992), 165-246.

%\bibitem[Ha]{Ha}A. Hatcher
%\newblock\emph{Algebraic topology}.
%\newblock Cambridge University Press, Cambridge, (2002).

\bibitem{H}S. Helgason
\newblock\emph{Differential Geometry, Lie Groups, and Symmetric Spaces}.
\newblock AMS, Providence, (2001).

\bibitem{Kac}M. Kac
\newblock\emph{Can one hear the shape of a drum?}
\newblock  Amer. Math. Monthly, \textbf{73} (1966) no. 4.

\bibitem{Lam}T.~Y. Lam
\newblock\emph{Introduction To Quadratic Forms Over Fields}.
\newblock  Graduate Studies in Mathematics, \textbf{67}, AMS, Providence, RI (2005).

%\bibitem[La]{La}S. Lang
%\newblock\emph{Algebraic number theory}.
%\newblock  Graduate Texts in Mathematics, \textbf{110}, Springer-Verlag, (1994).

%\bibitem[LMNR]{LMNR}C. J. Leininger, D. B. McReynolds, W. D. Neumann and A. W. Reid
%\newblock\emph{Length and Eigenvalue Equivalence}.
%\newblock  International Mathematics Research Notices, \textbf{24} (2007).

\bibitem{LM}J.-S. Li and J. Millson
\newblock\emph{On the first Betti number of a hyperbolic manifold with an arithmetic fundamental group}.
\newblock  Duke Math. J. \textbf{71} (1993) no. 2, 365-401. 

\bibitem{Lub97}A. Lubotzky
\newblock\emph{Eigenvalues of the Laplacian, the First Betti Number and the Congruence Subgroup Problem}.
\newblock  Ann. of Math. (2) \textbf{145} (1976) 441-452. 

\bibitem{LSV}A. Lubotzky, B. Samuels, and U. Vishne
\newblock\emph{Division algebras and noncommensurable isospectral manifolds}.
\newblock  Duke Math. J. \textbf{135} (2006), no. 2, 361-379. 

\bibitem{Mac}C. Maclachlan
\newblock\emph{Commensurability classes of discrete arithmetic hyperbolic groups}.
\newblock  Groups Geom. Dyn., \textbf{5}, No. 4, (2011).

%\bibitem[MaR1]{MaR1}C. Maclachlan and A. W. Reid
%\newblock\emph{Commensurability classes of arithmetic Kleinian groups and their Fuchsian subgroups}.
%\newblock  Math. Proc. Cambridge Philos. Soc., \textbf{102}, No 2, (1987).

\bibitem{MaR2}C. Maclachlan and A.~W. Reid
\newblock\emph{The Arithmetic of Hyperbolic 3-Manifolds}.
\newblock  Graduate Texts in Mathematics, \textbf{219}, Springer-Verlag, (2003).

\bibitem{Mar}G.~A. Margulis
\newblock\emph{Arithmeticity of the irreducible lattices in the semi-simple groups of rank greater than 1}.
\newblock  Invent. Math, \textbf{76}, (1984) 93-120.

%\bibitem[McK]{McK}H.P. McKean
%\newblock\emph{Selberg's trace formula as applied to a compact Riemann surface}.
%\newblock  Comm. Pure Appl. Math. \textbf{25} (1972), 225-246.

\bibitem{McRey} D.~B.~McReynolds 
\newblock\emph{Geometric spectra and commensurability}.
\newblock Canad. J. Math. 67 (2015), no. 1, 184-197.

\bibitem{McR} D.~B.~McReynolds  and  A. W. Reid
\newblock\emph{The genus spectrum of hyperbolic 3-manifolds}.
\newblock  Math. Res. Lett. 21 (2014), no. 1, 169-185.

\bibitem{M}J.~S. Milne
\newblock\href{http://www.jmilne.org/math/CourseNotes/CFT.pdf}{\emph{Class Field Theory}}.
\newblock  Course Notes.  Version 4.02, (2013).

%\bibitem[M2]{M2}J.S. Milne
%\newblock\href{http://www.jmilne.org/math/CourseNotes/ANT.pdf}{\emph{Algebraic Number Theory}}.
%\newblock  Course Notes.  Version 3.05, (2013).

\bibitem{Mi}J. Milnor
\newblock\emph{Eigenvalues of the Laplace operator on certain manifolds}.
\newblock  Proc. Nat. Acad. Sci. U.S.A., \textbf{51}, (1964) 542.

%\bibitem[Mill76]{Mill76}J. Millson
%\newblock\emph{On the First Betti Number of a Constant Negatively Curved Manifold}.
%\newblock  Ann. of Math. (2) \textbf{104} (1976) 235-247. 


%\bibitem[M55]{M55}G. Mostow
%\newblock\emph{Self-Adjoint Groups}.
%\newblock  Ann. of Math. (2) \textbf{62} (1955) 44-55. 


%\bibitem{MT}G. Mostow and T. Tamagawa
%\newblock\emph{On the compactness of arithmetically defined homogeneous spaces}.
%\newblock  Ann. of Math. (2) \textbf{76} (1962) 446-463. 

\bibitem{OM}O.~T. O'Meara
\newblock\emph{Introduction to Quadratic Forms}.
\newblock  Grundlehren der mathematischen Wissenschaften, \textbf{117}, Springer-Verlag, Berlin (1973).

\bibitem{PlRa}V. Platonov and A. Rapinchuk. 
\newblock\emph{Algebraic groups and number theory}.
\newblock    Translated from Russian by Rachel Rowen.  Pure and Applied Mathematics, \textbf{139}. Academic Press, Inc., Boston, MA, (1994).

%\bibitem{Pr89}G. Prasad
%\newblock\emph{Volumes of S-arithmetic quotients of semi-simple groups}.
%\newblock   Inst. Hautes \'Etudes Sci. Publ. Math. \textbf{69} (1989), no. 91-117.

%\bibitem{Pr94}G. Prasad
%\newblock\emph{$\R$-regular elements in Zariski-dense subgroups}.
%\newblock  Quart. J. Math. Oxford Ser. (2) \textbf{45} (1994), no. 180, 541-545.


\bibitem{PR}G. Prasad and A. S. Rapinchuk
\newblock\emph{Weakly commensurable arithmetic groups and isospectral locally symmetric spaces}.
\newblock  Inst. Hautes \'Etudes Sci. Publ. Math. \textbf{109} (2009), 113-184.

\bibitem{PR2}G. Prasad and A. S. Rapinchuk
\newblock\emph{Developments on the congruence subgroup problem after the work of Bass, Milnor and Serre}.
\newblock  Collected Papers of John Milnor: V. Algebra. AMS, (2010), 307-326.

%\bibitem[Rag]{Rag}M.S. Raghunathan
%\newblock\emph{Discrete Subgroups of Lie Groups}.
%\newblock Springer-Verlag, New York, NY (1972).

%\bibitem{R87}A. W. Reid
%\newblock\emph{Ph. D. thesis}.
%\newblock  Univ. of Aberdeen (1987).


\bibitem{R}A. W. Reid
\newblock\emph{Isospectrality and commensurability of arithmetic hyperbolic 2- and 3- manifolds}.
\newblock  Duke Math. J, \textbf{65} (1992), 215-228.

\bibitem{Sch}W. Scharlau
\newblock\emph{Quadratic and Hermitian Forms}.
\newblock  Grundlehren der mathematischen Wissenschaften, \textbf{270}, Springer-Verlag, Berlin (1985).

%\bibitem[Sco]{Sco}P. Scott
%\newblock\emph{The Geometry of 3-Manifolds}.
%\newblock  Bull. London Math. Soc. \textbf{15},  (1983), 401-487.

\bibitem{Sel}A. Selberg
\newblock\emph{On discontinuous groups in higher-dimensional symmetric spaces}.
\newblock ``Contributions to Function Theory,'' Tata Institute of Fundamental Research, Bombay (1960), 147-164.

%\bibitem[Se]{Se}J-P. Serre
%\newblock\emph{Le probl\`eme des groupes de congruence pour $SL_2$}.
%\newblock Ann. Math. \textbf{92} (1970), 489-527.

%\bibitem[Sp]{Sp}R. Spatzier
%\newblock\emph{On isospectral locally symmetric spaces and a theorem of von Neumann}.
%\newblock Duke Math. J. \textbf{59} (1989), no. 1, 289-294. 
%
%\bibitem[St]{St}M. Stover
%\newblock\emph{On the number of ends of rank one locally symmetric spaces}.
%\newblock  Preprint.

\bibitem{S}T. Sunada
\newblock\emph{Riemann coverings and isospectral manifolds}.
\newblock  Ann. Math. (2), \textbf{121} (1985), 169-186.

\bibitem{Sy}J.~J. Sylvester
\newblock\emph{A demonstration of the theorem that every homogeneous quadratic polynomial is reducible by real orthogonal substitutions to the form of a sum of positive and negative squares}.
\newblock  Philosophical Magazine (Ser. 4) \textbf{4} (23) (1852), 138-142.

\bibitem{Th}W. Thurston
\newblock\emph{The Geometry and Topology of 3-Manifolds}.
\newblock  Lecture Notes, Princeton University Math. Dept. (1978).

\bibitem{T}J. Tits
\newblock\emph{Classification of algebraic semisimple groups}.
\newblock  Proc. Summer Inst on Algebraic Groups and Discontinuous Groups (Boulder, 1965), Proc. Sympos. Pure Math.,  \textbf{9}, Amer. Math Soc., Providence, R.I., (1966), 33-62.

\bibitem{Vi}M.-F. Vign\'eras\
\newblock\emph{Vari\'et\'es Riemanniennes isospectrales et non isom\'etriques}.
\newblock  Ann. Math. (2), \textbf{112} (1980), 21-32.

\bibitem{Vin}E.~B. Vinberg
\newblock\emph{Rings of definition of dense subgroups of semisimple linear groups}.
\newblock  Math. USSR Izv., \textbf{5} (1971), 45-55.

\bibitem{Vin2}E.~B. Vinberg
\newblock\emph{Geometry II, Encyclopedia of Mathematical Sciences}.
\newblock   Vol. 29 Springer-Verlag, New York (1993).

%\bibitem[Whit]{Whit}H. Whitney
%\newblock\emph{Elementary Structure of Real Algebraic Varieties}.
%\newblock  Ann. Math. (2), \textbf{66} (1957), 545-556.

\bibitem{W}H. Weyl
\newblock\emph{\"{U}ber die asymptotische Verteilung der Eigenwerte}.
\newblock  Nachrichten der K\"{o}niglichen Gesellschaft der Wissenschaften zu G\"{o}ttingen (1911), 110-117.


\end{thebibliography}
\end{document}